\documentclass[11pt]{article}

\usepackage{latexsym,verbatim,ifthen,graphicx,mathrsfs}

\usepackage[english]{babel}
\usepackage[utf8]{inputenc} 	
\usepackage{amssymb}
\usepackage{amsmath}
\usepackage{amsthm}		
\usepackage{tikz-cd}
\usepackage{mathtools}  
\usepackage{microtype}		 
\usepackage[overload]{empheq}
\usepackage[title]{appendix}
\usepackage{setspace} 
\usepackage[font={footnotesize,it}]{caption} 

\usepackage{calc}
\newlength\myheight
\newlength\mydepth
\settototalheight\myheight{Xygp}
\usetikzlibrary{matrix,arrows,decorations.pathmorphing}
\usepackage[all]{xy}				
\usepackage{hyperref}
\hypersetup{
	colorlinks=true,
	allcolors=blue,
}   		
\usepackage{anysize}
\usepackage{cancel}									
\marginsize{2.5cm}{2.5cm}{2cm}{2cm}
\usepackage{enumerate}

\usepackage{mathrsfs}									
\usepackage{fourier}			
\usepackage[Sonny]{fncychap}	
\usepackage{fancyhdr}

\newtheorem{theorem}{Theorem}[section]
\newtheorem{lemma}[theorem]{Lemma}
\newtheorem{proposition}[theorem]{Proposition}
\newtheorem{corollary}[theorem]{Corollary}
\theoremstyle{definition}

\newtheorem{examples}[theorem]{Examples}
\theoremstyle{definition}
\newtheorem{remark}[theorem]{Remark}
\newtheorem{remarks}[theorem]{Remarks}
\newtheorem{definition}[theorem]{Definition}

\newtheorem{notation}[theorem]{Notation}

\newtheorem{thm}{Theorem}	

\newtheorem*{crl*}{Corollary}

\DeclareMathOperator{\Img}{Im}

\DeclareMathOperator{\id}{id}
\DeclareMathOperator{\Hom}{Hom}

\DeclareMathOperator{\Map}{\mathsf{Map}}

\DeclareMathOperator{\CoEnd}{CoEnd}

\renewcommand{\S}{{\mathsf S}}

\renewcommand{\P}{{\mathcal{P}}}

\newcommand{\D}{\mathcal{\D}}
\newcommand{\C}{\mathcal{\C}}

\newcommand{\R}{{\mathbb{R}}}

\newcommand{\coend}{\mathsf{CoEnd}}
\newcommand{\End}{\mathsf{End}}

\definecolor{red}{rgb}{1,0.1,0.1}
\definecolor{blue}{rgb}{0.1,0.1,1}
\definecolor{green}{rgb}{0.8,0.5,0.5}

\begin{document}

\title{A recognition principle for iterated suspensions as coalgebras over the little cubes operad}

\author{Oisín Flynn-Connolly, José M. Moreno-Fernández, Felix Wierstra}

\date{}
\maketitle

\medskip

\abstract{Our main result is a recognition principle for iterated suspensions as coalgebras over the little cubes operads. 
Given a topological operad, we construct a comonad in pointed topological spaces endowed with the wedge product. 
We then prove an approximation theorem showing that the comonad associated to the little $n$-cubes operad is weakly equivalent to the comonad $\Sigma^n \Omega^n$ arising from the suspension-loop space adjunction. 
Finally, our recognition theorem states that every little $n$-cubes coalgebra is homotopy equivalent to an $n$-fold suspension.     
These results are the Eckmann--Hilton dual of May's  foundational results on iterated loop spaces.}

\begin{spacing}{0.5}
    \tableofcontents
\end{spacing}

\section{Introduction}

Since the invention of operads, 
they have played an essential role in many parts of mathematics and physics. 
Their first application, and the original motivation for their invention,
was to the study of iterated loop spaces 
(see \cite{May72} and \cite{Boa73}). 
Operads provide a coherent framework for studying objects 
equipped with  many  "multiplications", 
i.e.\  operations with multiple inputs and one output, 
satisfying certain homotopical
coherences.
An important class of such objects are the $n$-fold loop spaces,
which are algebras over the little $n$-cubes operad. 
May showed in his recognition principle a homotopical converse, namely that every connected
little $n$-cubes algebra is weakly equivalent to an $n$-fold loop space;
and further proved an approximation theorem which asserts that the monad associated to the little $n$-cubes operad is weakly equivalent to the monad $\Omega^n \Sigma^n$. 
This approximation theorem reduced the study of operations on the homology of iterated loop spaces to the combinatorics of the little cubes operads. This perspective unraveled their complete algebraic structure (see \cite{Coh76}).

It has long been suspected that the recognition principle and the approximation theorem should have their corresponding
Eckmann--Hilton dual versions.
Indeed, work on this topic predates May's recognition theorem itself. 
By the end of the 1950s, Barratt and Stasheff studied in Oxford a preliminary version of these questions, 
trying to characterize $n$-fold suspensions and co-H-spaces in terms of their algebraic structure.
May's proof of the recognition principle reignited interest and there were immediate attempts to prove the Eckmann--Hilton dual;
some of this story can be found in the comments on the MathOverflow question \cite{MathOverflow}.
We are also aware of other more recent attempts to tackle the problem, 
but a solution has remained evasive until now.

The goal of this paper is to prove the Eckmann--Hilton dual results of May's work on iterated loop spaces.
Our proof is the consequence of two new key insights. 
Firstly, in general, without the added assumption of conilpotency, 
cofree coalgebra functors are notoriously difficult to construct and almost impossible to concretely work with. 
We were able to surmount this difficulty by proving that, 
in our case, elements of a cofree coalgebra are determined by their arity 1 component (see Lemma \ref{Lemma:Determined by 1st comp}). 
This is a very special feature of our setting which is surprising compared to the more algebraic setting. 
It is this fact that enabled us to cleanly define the cofree cooperation and perform the concrete manipulations that made the proof possible.
Secondly, we were able to show that the corelations in our comonad lie in arity 2,
something we were able to interpret in a very concrete way 
(see Proposition \ref{Prop: Geometric characterization of C_n}.) 
The Eckmann--Hilton dual of these facts both fail.

First of all, we construct a comonad in the category of pointed spaces associated to an operad. 
Next, we show that $n$-fold suspensions are coalgebras over the little 
$n$-cubes operad $\mathcal C_n$. 
More precisely,
we prove the following result.

\begin{thm} \em
\label{teo Intro 0}
	The $n$-fold reduced suspension of a pointed space $X$ is a $\mathcal C_n$-coalgebra.
	More precisely, 
	there is a natural and explicit operad map
	\begin{equation*}
		\nabla: \mathcal{C}_n \rightarrow \coend_{\Sigma^n X},
	\end{equation*}
	where $\operatorname{CoEnd}_{\Sigma^nX}$ is the coendomorphism operad of $\Sigma^nX$. 
	The map $\nabla$ encodes the homotopy coassociativity and homotopy cocommutativity of the classical pinch map $\Sigma^nX \to \Sigma^nX \vee \Sigma^nX$.
In particular, the pinch map  is an operation associated to an element of $\mathcal C_n(2)$. 
Furthermore, for any based map $X\to Y$, the induced map $\Sigma^n X \to \Sigma^n  Y$ extends to a morphism of $\mathcal C_n$-coalgebras.
\end{thm}

In this new setting, 
the Eckmann--Hilton dual of May's celebrated recognition of iterated loop spaces reads as follows.

\begin{thm} 
	\label{teo Intro 1}  \em
	Let $X$ be a $\mathcal C_n$-coalgebra. 
	Then there is a pointed space $\Gamma^n(X)$, naturally associated to $X$, 
	together with a weak equivalence of $\mathcal C_n$-coalgebras
	\begin{center}
		\begin{tikzcd}
 \Sigma^n\Gamma^n(X) \arrow[r, "\simeq"] &X,
		\end{tikzcd}
	\end{center}
	which is a deformation retract in the category of pointed spaces. Therefore, every $\mathcal C_n$-coalgebra has the homotopy type of an $n$-fold reduced suspension.
\end{thm}

Together, 
our theorems \ref{teo Intro 0}
 and \ref{teo Intro 1} 
provide the following intrinsic characterization of $n$-fold reduced suspensions as $\mathcal C_n$-coalgebras. 

\begin{crl*} \em
    Every $n$-fold reduced suspension is a $\mathcal C_n$-coalgebra, and if a pointed space is a $\mathcal C_n$-coalgebra then it is homotopy equivalent to an $n$-fold reduced suspension. 
\end{crl*}

It is worth noting that this result already exists at the level of $\Sigma^n\Omega^n$-coalgebras,
see Theorem \ref{teo: SuspensionLoops coalgebras}.

\smallskip

Let us turn our attention to the other 
celebrated result in \cite{May72},
the approximation theorem.
It constitutes an essential step for proving the recognition principle for connected $n$-fold loop spaces, 
and it is also the key to unlocking certain computations on the homology of iterated loop spaces.
Roughly speaking, 
the approximation theorem for loop spaces asserts that the free $\mathcal C_n$-algebra on a pointed connected
space $X$ is weakly equivalent to $\Omega^n\Sigma^nX$.
We also prove the Eckmann--Hilton dual of this result. 
It reads as follows.

\begin{thm} \em
	\label{teo Intro 2}
	For every $n\geq 1$, there is a natural morphism of comonads 
	$$\alpha_n : \Sigma^n \Omega^n \longrightarrow C_n.$$
	Furthermore, 
	for every pointed  space $X$, 
	there is an explicit natural deformation retract of pointed spaces
	\begin{center}
\begin{tikzcd}
\Sigma^n \Omega^n X \arrow[r, shift left] & C_n(X) \arrow[l, shift left] \arrow[loop, distance=2em, in=125, out=55]
\end{tikzcd}
	\end{center}
	In particular, $\alpha_n(X)$ is a homotopy equivalence.
\end{thm}

The comonad $C_n$ in the statement above is constructed in a natural way from the little $n$-cubes operad.
It is a non-trivial 
Eckmann--Hilton dualization of May's monad associated to $\mathcal C_n$.
To our knowledge, 
this comonad has not been studied elsewhere, 
and it seems to be an exciting new object that might shed 
light on further understanding $n$-fold reduced suspensions and other objects supporting a coaction of the little $n$-cubes operad.

\medskip

Let us complete a bit more of the historical context. 
It has been known for a long time that any $(n-1)$-connected CW complex of dimension less than or equal to $(2n-1)$
has the homotopy type of a (1-fold) suspension.
In \cite{Berstein63}, \cite{saito76}, \cite{ganea69} and finally \cite{Kle97}, this result was successively improved on.
In modern language, these authors showed that an $(n-1)$-connected co-$H$-space equipped with an $A_k$ 
comultiplication which is of dimension less than or equal to $k(n-1)+3$ is a suspension. 
The case of $k=\infty$
in \cite{Kle97} can be thought of as the $E_1$-version of Theorem \ref{teo Intro 1}, although our proof 
strategy is very different. 
From a different angle,
the case of iterated suspensions considered 
as coalgebras over (a homotopical version of) the $\Sigma^n\Omega^n$-comonad was recently treated in
\cite{Blo22}, where the authors obtained a recognition principle for $(n+1)$-connected, 
$n$-fold (simplicial) suspensions.
This last result differs from our Theorem \ref{teo Intro 1} in several key respects.
Firstly;  our notions of coalgebra differ as they pass to a derived functor in the homotopy category of pointed spaces,
while we consider only $\Sigma^n\Omega^n$-coalgebras in the classical sense of
coalgebras over comonads.
Secondly; our result has the sharpest possible connectivity requirement. 
The most striking difference with all previous  scholarships is that we make
heavy use of the little $n$-cubes operad and the comonad $C_n$;  
whereas 
these objects do not seem to have appeared in previous literature on the homotopy 
theory of iterated suspensions (with the exception of \cite{ginot12} in a very different context). 
In particular, there is no approximation theorem in \cite{Blo22}.

To conclude, a few remarks are in order. 
The first remark is that
to prove our theorems 
$\ref{teo Intro 1}$ and $\ref{teo Intro 2}$, 
we do not follow an Eckmann--Hilton dual approach to May's proof in the case of iterated loop spaces. We have found a framework and proof which depends on explicit homotopies and hence avoids the use of quasi-fibrations and the construction of auxiliary spaces.
In this sense, 
our approach is technically simpler. 
The approximation of suspensions is an independent result that we believe might have potential side applications. 
Finally, most of the results of this paper could have been stated using little $n$-disks instead of little $n$-cubes.
However, using cubes significantly simplifies many of the explicit formulae that appear when proving our results,
and therefore we choose to present things this way.

\subsection*{Notation and conventions}
\label{sec : Conventions}
All topological spaces are compactly generated and Hausdorff.
We denote by $I$ the unit interval in $\R$ and by $J$ its interior: 
$$ J = (0,1) \subseteq [0,1]=I.$$
The symmetric group on $n$ letters is denoted $S_n$.

\medskip

For $X = (X,*)$ a pointed space, it will be convenient to identify the $r$-fold wedge $X^{\vee r}$ as a subspace of the cartesian product $X^{\times r}$.
To do so, consider
$$X^{\vee r} = \bigcup_{i=1}^r \{*\}\times \cdots \times \underbrace{X}_i \times \cdots \times \{*\}\subseteq  X^{\times r}.$$
A point $x$ in the $i$-th factor of the wedge $X^{\vee r}$ is therefore identified with the point $(*,...,*,x,*,...,*)$ having $x$ at its $i$-th component and the base point at all others. We further use the convention that both $X^{\vee 0}$ and $X^{\times 0}$ are equal to the base point.
Given pointed maps $\varphi_1,...,\varphi_r : X\to Y$, 
we denote by $\left(\varphi_1,...,\varphi_r\right)$ the induced map $X \to Y^{\times r}$ to the product. 
Here, we implicitly used the diagonal map $d:X \rightarrow X^{\times r}$ given by $d(x)=(x,...,x)$. 
To simplify the notation we will omit the diagonal from the notation when this is clear from the context.
If the image of this map lands in the wedge subspace $Y^{\vee r}$, we denote the corresponding restriction by $\left\{\varphi_1,...,\varphi_r\right\}$. 
Thus, the curly brackets notation emphasizes that the map lands in the wedge rather than the product.
We reserve the notation $\varphi_1 \vee \cdots \vee \varphi_r $ for the induced map $X^{\vee r} \to Y^{\vee r}$ given by 
$$\left(\varphi_1 \vee \cdots \vee \varphi_r\right) \left(*,...,*,x_i,*,...,*\right) =\left(*,...,*,\varphi_i\left(x_i\right),*,...,*\right).$$
We frequently use the identification $\Sigma^n X = S^n \wedge X$ for the $n$-fold reduced suspension of a pointed space $X$.
Thus, points in $\Sigma^n X$ will be denoted $[t,x]$, where $t\in S^n$ and $x\in X$.
Since points in the suspensions are equivalence classes, we use the square brackets notation.
From now on, we implicitly assume  all suspensions are reduced.

We assume the reader is familiar with operad theory, 
especially in topological spaces,
and we refer to \cite{Fre17}.
We use the following conventions. 
An operad $\P$ in a symmetric monoidal category $\mathcal M = \left(\mathcal M,\otimes, \mathbb{1}\right)$ 
is \emph{unitary} if $\P(0) = \mathbb{1}$,
and \emph{non-unitary} if $P(0)$ is not defined (i.e., the underlying symmetric sequence of $\P$ starts in arity $1$). 
We borrow this nomenclature from \cite[Section 2.2]{Fre17}.
We will make heavy use of the operad of little $n$-cubes $\mathcal C_n$,
considered as a unitary operad where $\mathcal C_n(0) = *$ is a single point.

\medskip

\noindent {\bf Acknowledgments:}
The authors would like to thank Sergey Mozgovoy and Jim Stasheff for useful conversations and comments,
as well as the anonymous referees for their useful comments and suggestions. 
The second author has been partially supported by the MICINN grant PID2020-118753GB-I00. 
The third author was supported by the Dutch Research Organisation (NWO) grant number VI.Veni.202.046 and NWO grant 613.001.651.
This project has received funding from the European Union’s Horizon 2020 research and innovation programme under the Marie Skłodowska-Curie grant agreement No 945322. 
\raisebox{-\mydepth}{\fbox{\includegraphics[height=\myheight]{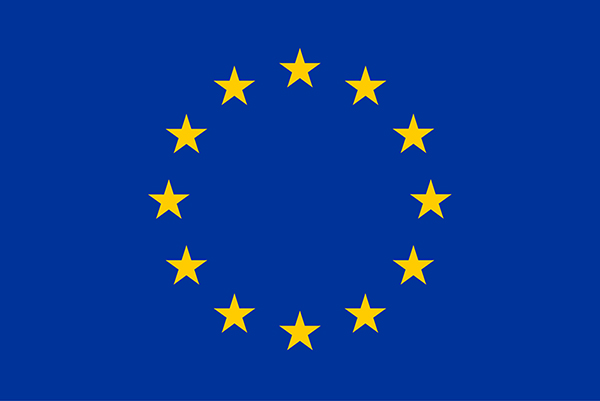}}}

\medskip

\noindent 
\textbf{MSC 2020:} 18M75, 55P40, 55P48, 55P30 

\noindent 
\textbf{Key words and phrases:} Little cubes operad, Suspensions, Coalgebras, Recognition principle, Eckmann-Hilton duality.

\smallskip

\section{Coalgebras over topological operads}

Given a unitary topological operad $\P$, 
we construct an explicit comonad $C_\P$ in \emph{pointed} spaces.
In Section \ref{sec: Topological comonads}, 
we carefully construct this comonad and study some of its basic properties.
The comonad $C_\P$ gives rise to the category of coalgebras over $\P$, 
also called  $\P$-coalgebras.
There is a second way of defining $\P$-coalgebras by using the coendomorphism operad that does not require the explicit construction of the   comonad $C_\P$.
This alternative construction has the advantage that it can be defined for all operads even when they are not necessarily unitary. The disadvantage is that it is not clear how to get an explicit comonad out of this definition.
We explain this alternative construction and show that in the case of unitary operads it gives an equivalent notion of $\P$-coalgebras in Section \ref{sec: Categories of coalgebras}.
We specialize to the case in which $\P$ is the operad $\mathcal C_n$ of little $n$-cubes in Section \ref{sec: Comonad of little cubes},
producing the central comonad of this paper.
Finally,
we prove Theorem \ref{teo Intro 0} in Section \ref{sec:Iterated suspensions are coalgebras} -
that 
the  $n$-fold reduced suspension of a pointed space is naturally a $\mathcal C_n$-coalgebra.
Therefore, the $n$-fold reduced suspensions are the paradigmatic examples of $\mathcal C_n$-coalgebras.

\begin{remark}
In our constructions of coalgebras,
we are mixing pointed and unpointed spaces.
All our operads live in the category of unpointed spaces while the coalgebras over the operads and associated comonads 
live in the category of pointed spaces.
\end{remark}

\subsection{Construction of topological comonads}\label{sec: Topological comonads}

In this section, we construct the mentioned  comonad $C_\P$ in pointed spaces out of a unitary operad $\P$ in unpointed spaces.

\medskip

Let us first establish some preliminary notation.
Denote 
$$
\mathsf{Top} = \left(\mathsf{Top},\times,\{*\}\right) \quad \textrm{ and } \quad \mathsf{Top_*} = \left(\mathsf{Top_*}, \vee, \{*\}\right)
$$ the symmetric monoidal categories of spaces endowed with the cartesian product $\times$, 
and pointed spaces endowed with the wedge  product $\vee$, respectively.
Let $\P$ be a unitary operad in $\mathsf{Top}$ with composition map $\gamma$ and denote the unitary operation by $* \in \P(0)$. 
Define the \emph{restriction operators}, for all $n\geq 1$ and  $1\leq i \leq n$, 
by inserting the unique point 
$* \in \P(0)$ at the $i$-th component: 
\begin{center}
    \begin{tikzcd}[row sep = tiny]
\P(n) \arrow[r, "d_i"]    & \P(n-1)                                        \\
\theta \arrow[r, maps to] & {\gamma\left(\theta;\id,...,*,...,\id\right).}
\end{tikzcd}
\end{center}
Let $X\in \mathsf{Top_*}$. The \emph{wedge collapse maps}, 
defined for all $n\geq 1$ and $1\leq i \leq n$, are given by collapsing the $i$-th factor in the wedge as follows:
\begin{center}
    \begin{tikzcd}[row sep = tiny]
X^{\vee n} \arrow[r, "\pi_i "]     & X^{\vee (n-1)}                    \\
{\left(x_1,...,x_n\right)} \arrow[r, maps to] & {\left(x_1,...,\widehat{x_i},...,x_n\right)}.
\end{tikzcd}
\end{center}
Here, the $r$-fold wedge is seen inside the $r$-fold cartesian product, 
and the notation $\widehat x_i$ means that we are sending the $i$-th component to the base point.

\begin{notation}
If $\P$ is a unitary operad  and $X$ is a pointed space, 
we denote $$\operatorname{Tot}\left(\P,X\right) := \prod_{n\geq 0} \Map_{S_n}\left(\P(n),X^{\vee n}\right).$$
Each space $\Map_{S_n}\left(\P(n),X^{\vee n}\right)$
consists of the equivariant maps from the arity $n$ component of $\P$ equipped with its usual $S_n$-action 
to the $n$-fold wedge of $X$ with itself endowed with the $S_n$-action that permutes the coordinates 
of its points by $\sigma \cdot \left(x_1,...,x_n\right) = \left(x_{\sigma(1)},...,x_{\sigma(n)}\right)$.
We frequently disregard the $0$-th component in the infinite product above, 
since the mapping space $\Map(\P(0),X^{\vee 0})$ is just a point. It can therefore be ignored in all computations that follow.
Thus, the point $\left(f_0,f_1,f_2,...\right) \in \operatorname{Tot}\left(\P,X\right)$ 
will be denoted $\left(f_1,f_2,...\right)$.
The topology on $\operatorname{Tot}\left(\P,X\right)$ is the usual product topology.
\end{notation}

We are ready to define the underlying endofunctor of our comonad $C_{\mathcal P}$.

\begin{definition}
\label{Def: The endofunctor}
Let $\P$ be a unitary operad in $\mathsf{Top}$.
Define the endofunctor in pointed spaces
\begin{center}
    \begin{tikzcd}[row sep = tiny]
C_\P : \mathsf{Top_*} \arrow[r] & \mathsf{Top_*}     \\
X \arrow[r, maps to]             & C_\P\left(X\right),
\end{tikzcd}
\end{center}
where
$$C_{\P}(X) = \left\{\alpha = \left(f_1,f_2,...\right)\in \operatorname{Tot}\left(\P,X\right) \ \mid \   \pi_i f_{n} =  f_{n-1}d_i\textrm{ for all $n\geq 2$ and $1\leq i\leq n$ }\right\}$$ is the subspace of $\operatorname{Tot}\left(\P,X\right)$ 
formed by those sequences $\left(f_1,f_2,...\right)$ that commute with the restriction operators and wedge collapse maps. 
That is, for all $n\geq 2$ and $1\leq i \leq n$, the following diagram commutes:
\begin{center}
	\begin{tikzcd}
		\P(n) \arrow[d, "d_i"'] \arrow[r, "f_n"] & X^{\vee n} \arrow[d, "\pi_i"'] \\
		\P(n-1) \arrow[r, "f_{n-1}"]             & X^{\vee (n-1)}                
	\end{tikzcd}
\end{center}
The base point of $C_\P\left(X\right)$ is the 
sequence $\alpha = \left(f_1,f_2,...\right)$ where each $f_r$ has image the base point of $X^{\vee r}$. 
Since the base point of the wedge $X^{\vee r}$ is fixed by the $S_r$-action, 
the base point is well-defined.
If $h:X\to Y$ is a pointed map, then $C_\P\left(h\right) : C_\P\left(X\right)\to C_\P\left(Y\right)$ 
 is defined by 
 $$C_\P\left(h\right)(\alpha) = \left(h\circ f_1, \ \left(h\vee h\right)\circ f_2 , ..., \left(h \vee ... \vee h\right) \circ f_n,... \right).$$
The $n$th term in the sequence above is given by
 $$\left(h\vee ... \vee h\right) \circ f_n :\P(n) \xrightarrow{\ f_n\ } X^{\vee n} \xrightarrow{\ h \vee ... \vee h\ } Y^{\vee n}.$$
\end{definition}

\begin{remarks} ~
\begin{enumerate}
\item The idea of defining  
$C_{\mathcal P}(X)$ above as a subspace of $\operatorname{Tot}\left(\P,X\right)$ 
arises from an Eckmann--Hilton  dualization of May's definition of the monad associated to an operad \cite{May72}.
Recall that the monad $M_n$ in pointed spaces defined in {\it{loc. cit.}} by using the little $n$-cubes operad is given  by
$$M_n\left(X\right) = \left(\coprod_{r\geq 0} \mathcal C_n(r) \times X^{\times r}\right)/\sim,$$
where $\sim$ is the equivalence relation that glues level $r$ to level $r+1$ by combining the restriction operators with the insertion of the base point, $(d_i(c),y) \sim (c,s_i(y))$, and imposing the compatibility with the group action, $(c\cdot \sigma, y) \sim (c,\sigma \cdot y)$.
\footnote{Here, $(c,y) \in \mathcal C_n(r)\times X^{\times (r-1)}$, $s_i(y)$ is the point of $X^{\times r}$ where we insert the base point at the $i$-th component, and $\sigma\in S_ r.$}
\item  The compatibility condition of a sequence 
$\alpha \in \operatorname{Tot}\left(\P,X\right)$ with the restriction operators and wedge collapse maps, 
\begin{equation}\label{ecu:Compatibility with restriction operator and wedge collapse}
	\pi_if_n = f_{n-1}d_i, \quad \textrm{for all $n\geq 1$ \ and\  $1\leq i \leq n$,}	
\end{equation} 
is the precise condition needed to incorporate a counit to the coalgebras in pointed spaces that result from the comonad $C_\P$. 
See Remark \ref{remark: Counital coalgebras} for further details.
\item The comonad $C_\P$ can be constructed in more general symmetric monoidal categories.
For the applications that we give in this paper, we are only interested in the category of topological spaces.
\end{enumerate}
\end{remarks}

Our next goal is to endow the endofunctor $C_{\mathcal P}$ with a comonad structure.
Before doing so, 
we make two elementary observations that will simplify some of our proofs later on.
We will use the following notation: if $h_1,...,h_r$ is a family of maps such that the composition $$h_1 \circ \cdots \circ h_{i-1}\circ h_{i+1} \circ \cdots \circ h_r$$ makes sense, then we denote the expression above  by $$h_1 \cdots \widehat{h_i} \cdots h_r.$$
That is, 
the hat $\widehat{(-)}$ on top of the $i$-th map indicates that this component is removed from the composition.
The first observation is the following.

\begin{lemma} 
	\label{Lemma:Determined by 1st comp}
	A sequence 
	$ \left(f_1,f_2,...\right) \in C_\P\left(X\right)$ is determined by its first component $f_1:\P(1)\to X$.
	That is, we can recursively write, for all $r\geq 2$,
	\begin{equation*}
		f_r = \left\{
		f_1 \widehat{d_1}d_2\cdots d_r \ ,\
		f_1 d_1\widehat{d_2}d_3\cdots d_r \ ,...,\
		f_1 d_1d_2\cdots d_{r-1}\widehat{d_r}
		\right\},
	\end{equation*}
where the $d_i$'s are the maps that insert $*\in\P(0)$ into the $i$-th entry.
\end{lemma}

Recall that the term on the right hand side above follows the notation from Section \ref{sec : Conventions}.

\begin{proof}
	Let  
	$\alpha = \left(f_1,f_2,...\right) \in C_\P\left(X\right).$ 
    Before we give a general proof of the lemma we first work out the $r=2$ case, since this makes the general argument clearer.
	Let $$f_2 : \P(2)\to X\vee X$$
	be the second component of $\alpha$.
	Denote by $q_i: X\vee X \to X$ the projection onto the $i$-th factor of the wedge, for $i=1,2.$ 
	There are  identifications $q_i = \pi_{3-i}$, 
	where $\pi_1,\pi_2 : X\vee X \to X$ are the corresponding wedge collapse maps.
Then, 
	\begin{align*}
		f_2 &= \left\{q_1f_2,q_2f_2\right\} = \left\{\pi_2f_2,\pi_1f_2\right\} = \left\{f_1d_2,f_1d_1\right\} = \left\{f_1\widehat{d_1}d_2, f_1d_1\widehat{d_2}\right\}.
	\end{align*}
In the third equality above, 
we used Equation (\ref{ecu:Compatibility with restriction operator and wedge collapse}) for $n=2$.
The proof for general $f_r$ follows a slight generalization of the case just proven, where we recursively use the identities of Equation (\ref{ecu:Compatibility with restriction operator and wedge collapse}) for all $n$ between $2$ and $r$. 
Thus, let $$f_r:\P(r)\to X^{\vee r}$$ be the $r$th component of $\alpha.$ 
Denote by $q_i : X^{\vee r}\to X$ the projection onto the $i$-th factor of the wedge, for $i=1,...,r$.
There are identifications $$q_i = \pi_1  \pi_2  \cdots  \widehat{\pi_i} \cdots \pi_r, \quad \textrm{ for all } i = 1,...,r.$$
Recall the hat $\widehat{\pi_i}$ indicates that we omit the $i$-th term.
There is a slight but harmless abuse of notation above, since the $\pi_j$'s that appear in the expression of $q_i$ have different domains.
Then, 
\begin{align*}
	f_r &= \left\{q_1f_r, \ q_2f_r,...,\ q_rf_r\right\} \\[0.25cm]
	&= \left\{\widehat{\pi_1}\pi_2  \pi_3  \cdots \pi_r f_r \ , \ \pi_1 \widehat{\pi_2} \pi_3 \pi_4 \cdots  \pi_r f_r\ ,\  ..., \pi_1  \pi_2  \cdots  \pi_{r-1} \widehat{\pi_r} f_r \right\}\\[0.25cm]
	&= \left\{\widehat{\pi_1}\pi_2  \pi_3  \cdots \left(\pi_r f_r\right)\ , \ \pi_1 \widehat{\pi_2} \pi_3 \pi_4 \cdots  \left(\pi_r f_r\right)\ , \ ..., \pi_1  \pi_2  \cdots  \left(\pi_{r-1}  f_r\right) \right\}\\[0.25cm]
	&= \left\{\widehat{\pi_1}\pi_2  \pi_3  \cdots \left(f_{r-1}d_r\right)\ , \ \pi_1 \widehat{\pi_2} \pi_3 \pi_4 \cdots  \left(f_{r-1}d_r\right)\ , \ ..., \pi_1  \pi_2  \cdots  \left(  f_{r-1}d_{r-1}\right) \right\}\\[0.25cm]
	&= \qquad \cdots \\[0.25cm]
	&= \left\{\widehat{\pi_1}\pi_2\left(f_2 d_3\cdots d_r\right) \ , \ \pi_1 \widehat{\pi_2} \left(f_2 d_3 \cdots d_r\right) \ , \,  \pi_1 f_2\left( d_2\widehat{d_3}d_4\cdots d_r\right)\ , \ \pi_1 f_2 \left(d_2 \cdots d_{r-1}\widehat{d_r}\right) \right\}\\[0.25cm]
	&= \left\{
	f_1 \widehat{d_1}d_2\cdots d_r \ ,\
	f_1 d_1\widehat{d_2}d_3\cdots d_r \ ,...,\
	f_1 d_1d_2\cdots d_{r-1}\widehat{d_r}
	\right\}.
\end{align*}
This completes the proof.
\end{proof}

The result above tells us that any sequence 
$\alpha  = \left(f_1,f_2,...\right)\in C_{\P}(X)$ 
can be written as
$$\alpha = \left(f_1,f_2,f_3,...\right) = \left(f_1, \left\{f_1d_2,f_1d_1\right\}, \left\{f_1d_2d_3, f_1d_1d_3,f_1d_1d_2\right\},...\right).$$
However, it does not assert that any map $\P(1)\to X$ can be extended 
to a sequence in $ C_\P\left(X\right)$ whose first component is the given map.
In fact, that is usually not the case.
Below, we give a characterization 
of when such a map extends for $\P$ a unitary operad in topological spaces.

Let us point out the second observation.
We need the following notation.
If $X$ is a pointed space, and $f: \P(1)\to X$ is any map, 
define for all $r\geq 2$ and $1\leq i \leq r$ the collection of maps 
$$f_r^i := f\left(d_1\cdots \widehat {d_i} \cdots d_r\right) : \P(r)\to X.$$
The map 
$$f_r := \left\{f_r^1,...,f_r^r\right\} : \P(r)\to X^{\vee r}$$
is then defined by first applying the diagonal map $\P(r) \rightarrow \P(r)^{\times r}$ and then the product of the $f_r^i$. 
The map above usually lands in the product but it restricts to the wedge if, and only if,
the map $f$ belongs to the underlying space of the comonad.

\begin{proposition}
\label{Prop:Characterization of C(X) as a subspace of Map}
Let $X$ be a pointed space.
Then the space $C_\P\left(X\right)$ is homeomorphic to the subspace of  
\ $\Map\left(\P(1),X\right)$
given by all those maps $f_1:\P(1)\to X$ such that
for any $r\geq 2$ and $c\in \P(r)$, it follows that $f_r^i(c)=*$ is the base point for all $i$ except at most one.
In particular, the image of the map $$f_r := \left(f_r^1,...,f_r^r\right) : \P(r)\to X^{\times r} $$
is contained in the subspace $X^{\vee r} \subseteq X^{\times r}$.
Furthermore, 
each $$f_r: \P(r)\to X^{\vee r}$$ is $S_r$-equivariant.
Under this identification, 
the value $C_\P\left(\phi\right)$ on a pointed map $\phi:X\to Y$ is the postcomposition with $\phi$: 
	\begin{center}
	    \begin{tikzcd}[row sep = tiny]
C_\P\left(X\right) \arrow[r, "C_\P\left(\phi\right)"] & C_\P\left(Y\right)                 \\
f \arrow[r, maps to]                             & C_\P\left(\phi\right)\left(f\right) = \phi \circ f.
\end{tikzcd}
	\end{center}
\end{proposition}

\begin{proof}
The fact that 
for any $r\geq 2$ and $c\in \P(r)$, it follows that $f_r^i(c)=*$ is the base point for all $i$ except at most one,
implies that the map 
$$f_r =\left(f_r^1 ,..., f_r^r\right) : \P(r)\to X^{\times r}$$ 
has its image in the wedge.
Thus, it is correct to 
 write $f_r = \left\{f_r^1 , ... , f_r^r\right\}$.

\fbox{\strut \ $\Rightarrow$\ }
Let
$\left(f_1,f_2,...\right)\in  C_\P\left(X\right)$. We must show that $f^i_r$ evaluated at any $c\in \P(r)$
is the base point for all $i$ except at most one. 
It is a straightforward consequence of Lemma \ref{Lemma:Determined by 1st comp} that 
the component $f_1$ of the sequence gives rise to the family of maps  $\left\{f_r^i\right\}$ of the statement, 
with $f_r = \left\{f_r^1,...,f_r^r\right\}$. 
So, this implication follows.

\fbox{\strut \ $\Leftarrow$ \ }
Let $f_1:\P(1)\to X$ be a map giving rise to the family of maps $\left\{f_r^i\right\}$ and $f_r$ satisfying the hypotheses of the statement. 
We show next that this indeed belongs to $C_\P(X)$.
Form the sequence 
$$\left(f_1,f_2,...\right) \in \operatorname{Tot}\left(\P,X\right).$$
It suffices to check that for every $r\geq 2$ and $1\leq i \leq r$, the identity $f_{r-1}d_i = \pi_if_r$ holds.
To do so, we will make use of the following fact and notation for maps induced onto a wedge of pointed spaces: 
given pointed spaces $W,Y,Z$ and  maps $\varphi_1,...,\varphi_r : Y \to Z$  
such that $\left\{\varphi_1,...,\varphi_r\right\} : Y \to Z^{\vee r}$ is well-defined, 
then for any map $g:W \to Y$, we have
$$\left\{\varphi_1,...,\varphi_r\right\} \circ g = \left\{\varphi_1 \circ g,...,\varphi_r \circ g\right\} 
: W \to Z^{\vee r}.$$
Thus, fix some $r\geq 2$ and $1\leq i\leq r$.
On the one hand, 
\begin{align}\label{ecu:Comp1}
	\pi_i f_r &= \pi_i \left\{f_1 \widehat{d_1}\cdots d_r \ ,\ ...\ ,\ f_1 d_1\cdots \widehat{d_r} \right\} = \left\{f_1 \widehat{d_1}\cdots d_r \ ,\ ...\ , \cancel{f_1d_1\cdots \widehat{d_i}\cdots d_r} \ , \ ... \ , \ f_1 d_1\cdots \widehat{d_r} \right\}.
\end{align}
Above, the strike-through indicates that the $i$-th component is not part of the sequence.
On the other hand,
\begin{align}\label{ecu:Comp2}
	f_{r-1}d_i &= \left\{f_1 \widehat{d_1}\cdots d_{r-1} \ ,\ ...\ ,\ f_1 d_1\cdots \widehat{d_{r-1}} \right\}\circ d_i 
	= \left\{f_1 \widehat{d_1}\cdots d_{r-1}\circ d_i \ ,\ ...\ ,\ f_1 d_1\cdots \widehat{d_{r-1}}\circ d_i \right\}.
\end{align}	
It suffices to check that, for any $j$ with  $1\leq j \leq r-1$, 
the $j$-th component of the sequence $(\ref{ecu:Comp1})$ is equal to the $j$-th component of the sequence $(\ref{ecu:Comp2})$.
This is a straightforward check, taking into account whether  $j\leq i$ or $j\geq i$, 
and using the simplicial identities satisfied by the $d_k$'s - namely, that $d_id_j=d_{j-1}d_i$ for $i<j$.
\end{proof}

Proposition \ref{Prop:Characterization of C(X) as a subspace of Map} above is very useful,
as we will see in Section \ref{sec:approximation theorem}.
Remark that this result identifies the space $C_\P(X)$ as the subspace of $\Map\left(\mathcal P(1),X\right)$ formed by those
maps satisfying an extra property. 
Bear in mind that, under this identification, the evaluation of $C_\P$ on a morphism $\phi:X\to Y$ corresponds to the postcomposition with $\phi$.

\medskip

Before going on, we introduce some notation that will be useful later.

\begin{notation}
We will occasionally use the following notation for the composition of the restriction operators:
	$$D_i = d_1 \cdots \widehat{d_i} \cdots d_r : \P(r) \to \P(1).$$
	These choices will simplify the formulae in what follows, making our results more readable.
	Remark also that, for any operation $\theta\in \P(r)$, the resulting operation $D_i(\theta) \in \P(1)$ is exactly $$D_i(\theta) = \gamma (\theta;*,...,*,\underbrace{\id_{\P}}_i,*,...,*),$$ where $\gamma$ is the composition map of $\P$, the element $\id_{\P}\in \P(1)$ is the operadic unit,
	and $* \in \P(0)$ is the unitary operation.
	In other words, $D_i(\theta)$ retains the unary operation determined by the $i$-th input of $\theta.$
	For example, if $\mathcal P=\mathcal C_n$ is the little $n$-cubes operad 
	and $\theta = \left(c_1,...,c_r\right)\in \mathcal C_n(r)$ is a configuration of $r$ little $n$-cubes, 
	then $D_i(\theta) = c_i$ 
	is the $i$-th little $n$-cube of the configuration seen as an element of $\mathcal C_n(1)$.
\end{notation}

Let us finally equip the endofunctor $C_{\mathcal P}$ 
with natural transformations $\varepsilon : C_{\mathcal P} \to \id_{\mathsf{Top_*}}$ and $\Delta : C_{\mathcal P} \to C_{\mathcal P} \circ C_{\mathcal P}$
that make it a comonad.
From now on, to lighten notation,
we denote $C = C_\P$,
assuming understood the operad $\P$.

\begin{definition}
	\label{def: Comonad structure maps in Top}
Let $C=C_\P : \mathsf{Top_*} \to \mathsf{Top_*}$ be the endofunctor of Definition \ref{Def: The endofunctor}.
Define the natural transformations
	\begin{center}
		$\varepsilon : C \to \id_{\mathsf{Top_*}}$ \quad and \quad $\Delta : C \to C \circ C$
	\end{center}  level-wise on a pointed space $X$ as follows.

\noindent $\bullet$ The counit structure map is defined by
		\begin{center}
		\begin{tikzcd}[row sep=tiny]
			\varepsilon_X:C(X) \arrow[r]                           & X                                 \\
			{\alpha = \left(f_1,f_2,...\right)} \arrow[r, maps to] & \varepsilon_X(\alpha) := f_1(\id_{\P}).
		\end{tikzcd}
	\end{center}
Here, $\id_{\P} \in \P(1)$ is the operadic unit.

\noindent $\bullet$ 
We next define the coproduct structure map  $$\Delta_X:C(X) \to C(C(X)).$$
To do so, let $\alpha = \left(f_1,f_2,...\right) \in C(X)$.
Then $\Delta_X\left(\alpha\right) = \left(\bar f_1, \bar f_2,...\right)$ is an element of the space $C(Z)$, with $Z=C(X)$.
Thus, it is formed by a sequence of maps 
$$\bar f_r : \P(r) \to C(X)^{\vee r}$$ satisfying the compatibility conditions 
$$\pi_i \bar f_r = \bar f_{r-1}d_i, \quad \textrm{for $r\geq 2$ \ and\  $1\leq i \leq r.$}$$
Because of Lemma \ref{Lemma:Determined by 1st comp},
we only need to define the arity one component $\bar f_1 : \P(1) \to C(X)$ 
and extend it as a sequence by the formula $$\bar f_r = \left\{\bar f_1 D_1,...,\bar f_1D_r\right\},$$
where $D_i = d_1 \cdots \widehat{d_i} \cdots d_r$.
\end{definition}

For the definition above to be complete and correct, 
we require two steps:
\begin{itemize}
	 \item [Step 1.] Define $\bar f_1:\P(1) \to C(X)$.
	 \item [Step 2.] Check that $\bar f_1 D_i = *$ is the base point for all indexes $i$, except for at most one.
\end{itemize}

Let us check these steps.

\smallskip

\noindent \fbox{{\strut \ Step 1 \ }} Denote by $\gamma$ the operadic composition map of $\P.$
Define $\bar f_1 : \P(1) \to C(X)$ by 
$$\bar f_1(\mu) = \left(g_1^\mu, g_2^\mu,...\right) \quad \textrm{ for all $\mu\in \P(1)$},$$ 
where the maps $g_r^\mu:\P(r)\to X^{\vee r}$ in the sequence are as follows. 
The first one is:
\begin{align*}
	g_1^\mu: \P(1)\to X,  \qquad  g_1^\mu(\theta) &:= f_1\left(\gamma\left(\mu;\theta\right)\right),
\end{align*}
for $\theta \in \P(1)$. That is, $g_1^\mu = f_1\left(\gamma\left(\mu;-\right)\right)$.
The rest of the maps $g_r^\mu$ are recursively defined by the formula

$$g_r^{\mu}:\P(r) \to X^{\vee r} $$
$$g_r^\mu(\theta) = \left\{g_1^\mu D_1(\theta) ,..., g_1^\mu D_r(\theta)\right\}=\{g_1^\mu(\gamma(\theta;\id_{\P},*,..,*),...,g_1^\mu(\gamma(\theta;*,..,*,\id_\P))$$
For $\theta \in \P(r)$. We will check below that the image of $g_r^\mu$   is indeed contained in the wedge $X^{\vee r}$.
The family of maps $g_r^\mu$ can be explicitly described.
Let us first describe $g_2^\mu : \P(2)\to X \vee X.$
Using, in the order given, the recursive definition of $g_2^\mu$, the definitions of $D_i$ and of $g_1^\mu$, and the associativity of $\gamma$, we can write
\begin{align*}
	g_2^\mu (\theta) &= \left\{ g_1^\mu D_1 (\theta), g_1^\mu D_2(\theta) \right\} = \left\{ g_1^\mu \left(\gamma\left(\theta;\id_\P,*\right)\right), g_1^\mu \left(\gamma\left(\theta;*,\id_\P\right)\right) \right\} 
	\\[0.15cm]
	&= \left\{ f_1\left(\gamma\left(\mu;\gamma\left(\theta;\id_\P,*\right)\right)\right), f_1\left(\gamma\left(\mu;\gamma\left(\theta;*,\id_\P\right)\right)\right)\right\}\\[0.15cm]
	&= \left\{ f_1\left(\gamma\left(\gamma\left(\mu;\theta\right);\id_{\P},*\right)\right), f_1\left(\gamma\left(\gamma\left(\mu;\theta\right);*,\id_{\P}\right)\right)\right\}.
\end{align*}
Thus, $$g_2^\mu = \left\{f_1D_1\left(\gamma(\mu;-)\right),f_1D_2\left(\gamma(\mu;-)\right)\right\}.$$
Next we need to show that $\overline{f_2}$ has its image in the wedge $C(X) \vee C(X)$. Since $\alpha = \left(f_1,f_2,...\right)$ is an element of $C(X)$, it follows that all $f_1D_i=*$ are the base point, except for at most a single index $i$.
Therefore, indeed, $g_2^\mu$ has its image in the wedge.
Furthermore, so defined, $g_2^\mu$ is $S_2$-equivariant.
In general, exactly the same steps as for the $r=2$ case show that the explicit formula for $g_r^\mu$ is
\begin{equation*}
	g_r^\mu (\theta) = \big\{f_1\left( \gamma\left(\gamma(\mu;\theta);\id_\P,*,...,*\right)\right), ... , f_1\big( \gamma\big(\gamma(\mu;\theta);*,...,*,\underbrace{\id_\P}_j,*,...,*\big)\big), ..., f_1\left( \gamma\left(\gamma(\mu;\theta);*,...,*,\id_\P\right)\right)\big\}.
\end{equation*}
Above, the $j$-th component in the wedge has the identity $\id_\P\in \P(1)$ at the $j$-th component.

\medskip

\noindent  \fbox{{\strut \ Step 2 \ }} 
Let us check that $\bar f_1D_i = *$ is the base point for all indexes $i$ except for at most a single one.
We will use Proposition \ref{Prop:Characterization of C(X) as a subspace of Map}.
Recall that for fixed $i$, the map $$\bar f_1D_i : \P(r)\to C(X)$$
evaluated at some operation $\mu\in \P(r)$ is the previously defined sequence $$\bar f_1D_i(\mu) = \left(g_1^{D_i(\mu)}, g_2^{D_i(\mu)},...\right).$$ 
First, observe that for any $\theta\in \P(1)$ and index $i$, with $1\leq i \leq r,$ 
we have 
$$\gamma\left(D_i(\mu);\theta\right) 
= D_i \left(\gamma(\mu;\id_\P,..,\theta,...,\id_\P)\right).$$
Therefore, the first component of the sequence  $\bar f_1D_1(\mu)$ can be written as $$g_1^{D_i(\mu)} = f_1\left( D_i\left(\gamma\left(\mu;-\right)\right)\right).$$
Since the sequence $\left(f_1,f_2,...\right)$ is an element of the space $C(X)$, 
it follows that $f_1D_i$ is the base point for all $i$ except for at most one, and therefore, 
the same holds for the family $\left\{g_1^{D_1(\mu)},...,g_1^{D_i(\mu)},...\right\}$, 
which implies that $\bar f_1 D_i$ is the base point for all $i$ except at most one.

\begin{remark}
In Proposition \ref{Prop:Characterization of C(X) as a subspace of Map},
we identified $C(X)$ as a certain subspace of $\operatorname{Map}\left(\P(1),X\right)$.
From this point of view,
the comultiplication $\Delta = \Delta_X :C(X) \to CC(X)$ is given as follows.
Let $f\in C(X)\subseteq \operatorname{Map}\left(\P(1),X\right)$. 
Then,
    \begin{center}
        \begin{tikzcd}[row sep = tiny]
\Delta\left(f\right):\P(1) \arrow[r] & C(X)                                      &                                        \\
c \arrow[r, maps to]                   & \Delta\left(f\right)(c):\P(1) \arrow[r] & X                                      \\
                                       & d \arrow[r, maps to]                      & {f\left(\gamma\left(c;d\right)\right)}.
\end{tikzcd}
    \end{center}
That is, given $f\in C(X)$, and $c,d\in \P(1)$, the map $\Delta\left(f\right)$ is explicitly given by 
$$\Delta\left(f\right)(c)(d) = f\left(\gamma\left(c;d\right)\right).$$
\end{remark}

\begin{proposition}
	\label{Prop: Comonad structure maps}
	With the notation before, the triple $(C,\varepsilon,\Delta)$ is a comonad in $\mathsf{Top_*}$.
\end{proposition}

\begin{proof}
	We prove the coassociativity and counit axioms object-wise. 
	For a pointed space $X$, these axioms are described by
	the following diagrams:
	\begin{center}
		\begin{tikzcd}
			C(X) \arrow[r, "\Delta_X"] \arrow[d, "\Delta_X"']                                       & C\left(C\left(X\right)\right) \arrow[d, "\Delta_{C(X)}"] &  &  & C(X) \arrow[r, "\Delta_X"] \arrow[rd, "\id"] \arrow[d, "\Delta_X"']               & C\left(C\left(X\right)\right) \arrow[d, "\varepsilon_{C(X)}"] \\
			C\left(C\left(X\right)\right) \arrow[r, "C\left(\Delta_X\right)"] & C\left(C\left(C\left(X\right)\right)\right)              &  &  & C\left(C\left(X\right)\right) \arrow[r, "C(\varepsilon_X)"] & C(X)                                   
		\end{tikzcd}
	\end{center}
The left diagram corresponds to the coassociativity condition, and the right diagram to the counit condition.
\medskip

\noindent Let $\alpha = \left(f_1,f_2,...\right) \in C(X)$. We check next that it satisfies the mentioned diagrams.

\medskip

\noindent $\triangleright$ Coassociativity.
We must check that 
\begin{equation}\label{ecu:Coassociativity}
	\left(C\left(\Delta_X\right) \circ \Delta_X\right)(\alpha) = \left(\Delta_{C(X)}\circ \Delta_X\right) (\alpha).
\end{equation}
We analyze $\Delta_X(\alpha)$ first, given that it appears on both sides of the equation above, 
and then look at each of the sides of the equation above.
By Lemma \ref{Lemma:Determined by 1st comp}, it suffices to check that the arity one term of the  sequences arising from both sides of Equation (\ref{ecu:Coassociativity}) agree.
This will ultimately follow from the associativity of the operadic composition $\gamma$ of the operad $\P.$ 

\smallskip

$\bullet$
Description of  $\Delta_X(\alpha)$.
	\begin{center}
\begin{tikzcd}[column sep=tiny,row sep=tiny]
	\Delta_X:C(X) \arrow[r] & C\left(C(X)\right)      \\
	\alpha \arrow[r, maps to]          & {\Delta_X(\alpha) = \left(\bar f_1, \bar f_2,...\right)}
\end{tikzcd}
\end{center}
By Lemma \ref{Lemma:Determined by 1st comp}, the sequence $\left(\bar f_1, \bar f_2,...\right)$ is determined by its first component $\bar f_1.$
It is given as follows:
\begin{center}
\begin{tikzcd}[column sep=tiny,row sep=tiny]
	\bar f_1:\P(1) \arrow[r] & C(X)                         &  & g_1^\mu:\P(1) \arrow[r]   & X                                                  \\
	\mu \arrow[r, maps to]   & {\bar f_1(\mu) = \left(g_1^\mu,g_2^\mu,...\right)} &  & \theta \arrow[r, maps to] & g_1^\mu(\theta)=f_1\left(\gamma(\mu;\theta)\right)
\end{tikzcd}
\end{center}
\item $\bullet$ The left hand side of Equation (\ref{ecu:Coassociativity}) reads:
$$\left(C\left(\Delta_X\right) \circ \Delta_X\right)(\alpha) = C(\Delta_X)\left(\Delta_X(\alpha)\right) = C\left(\bar f_1,\bar f_2,...\right) = \left(\Delta_X \circ \bar f_1, \{\Delta_X,\Delta_X\} \circ \bar f_2,...\right).$$
Here, given maps $\varphi_i:X_i\to Y$, we are denoting the induced map by $\{\varphi_1,...,\varphi_n\} :X_1\vee ... \vee X_n \to Y.$
We have:
\begin{center}
\begin{tikzcd}[column sep=tiny,row sep=tiny]
	\Delta_X\circ \bar f_1:\P(1) \arrow[r] & C(X) \arrow[r]                                              & C\left(C\left(X\right)\right) \\
	\mu \arrow[r, maps to]                 & {\bar f_1(\mu) = \left(g_1^\mu, g_2^\mu,...\right)} \arrow[r, maps to] & {\left(\bar g_1^\mu,\bar g_2^\mu,...\right)}       
\end{tikzcd}
\end{center}
The map $\bar g_1^\mu$ above is determined by:
\begin{center}
\begin{tikzcd}[column sep=tiny,row sep=tiny]
	\bar g_1^\mu:\P(1) \arrow[r] & C(X)                         &  & h_1:\P(1) \arrow[r]        & X                                                         \\
	\theta \arrow[r, maps to]    & {\bar g_1^\mu(\theta) := (h_1,h_2,...)} &  & \lambda \arrow[r, maps to] & h_1(\lambda) = g_1^\mu\left(\gamma(\theta;\lambda)\right)
\end{tikzcd}
\end{center}

\item $\bullet$ The right hand side of Equation (\ref{ecu:Coassociativity}) reads:
\begin{equation*}
 \left(\Delta_{C(X)}\circ \Delta_X\right) (\alpha) = \Delta_{C(X)} \left(\Delta_X(\alpha)\right) = \Delta_{C(X)} \left(\bar f_1,\bar f_2,...\right) = 
 \left(\bar{\bar{f_1}} \bar{\bar{f_2}},...\right).
\end{equation*}
Here, 
\begin{center}
	\begin{tikzcd}[column sep=tiny,row sep=tiny]
		\bar{\bar{f_1}}:\P(1) \arrow[r] & C\left(C\left(X\right)\right)       &  & l_1^\mu:\P(1) \arrow[r]   & C(X)                                           \\
		\mu \arrow[r, maps to]          & {\bar{\bar{f_1}}(\mu) = \left(l_1^\mu,l_2^\mu,...\right)} &  & \theta \arrow[r, maps to] & l_1^\mu(\theta) = \bar f_1\left(\gamma(\mu;\theta)\right)
	\end{tikzcd}
\end{center}

As mentioned, to check the coassociativity condition it suffices to check that $\bar{\bar{f_1}} = \Delta_X\circ \bar f_1.$
By Lemma \ref{Lemma:Determined by 1st comp} again, 
our problem reduces to checking that $\ell_1^\mu=\bar g_1^\mu.$
And once more, using the same lemma, this reduces to checking that the sequence $\bar f_1\left(\gamma\left(\mu;\theta\right)\right)$ has first term equal to $h_1(\lambda)$ described before.
The  first term is explicitly given by 
\begin{equation}
\label{ecu : Una}
    f_1\left(\gamma\left(\gamma\left(\mu;\theta\right);\lambda\right)\right).
\end{equation}
On the right hand side,  the first nested  term of 
$g_1^\mu\left(\gamma\left(\theta;\lambda\right)\right)$ is explicitly given by 
\begin{equation}
\label{ecu : Dos}
    f_1\left(\gamma\left(\mu;\gamma\left(\theta;\lambda\right)\right)\right).
\end{equation}
By the associativity of the operadic composition $\gamma$, 
the term inside $f_1$ in Equation (\ref{ecu : Una}) is the same as the term inside $f_1$ in Equation (\ref{ecu : Dos}).
Thus, these two maps are equal.
This proves the coassociativity of the comultiplication.

\smallskip

\noindent $\triangleright$ Counit. We must check two identities:

\begin{enumerate}
	\item $\left(C\left(\varepsilon_X\right)\circ \Delta_X\right) (\alpha) = \alpha.$
	
	Indeed,
	\begin{align*}
		\left(C\left(\varepsilon_X\right)\circ \Delta_X\right) (\alpha) &= C\left(\varepsilon_X\right) \left(\Delta_X(\alpha)\right) = C\left(\varepsilon_X\right) \left(\bar f_1,\bar f_2,...\right) = \left(\varepsilon_X \circ \bar f_1, \{\varepsilon_X,\varepsilon_X\}\circ \bar f_2,...\right).
	\end{align*} Let us check that $\varepsilon_X\circ \bar f_1 = f_1$ as maps $\P(1)\to X$. 
 If $\mu\in \P(1),$ then: 
$$\left(\varepsilon_X \circ \bar f_1\right) (\mu) = \varepsilon_X \left(\bar f_1(\mu)\right) = \varepsilon_X\left(g_1^\mu,g_2^\mu,...\right) = g_1^\mu(\id_\P) = f_1\left(\gamma(\mu;\id_\P)\right) = f_1(\mu).$$

	\item $\left(\varepsilon_{C(X)} \circ \Delta_X\right) (\alpha) = \alpha$. 
	
	In this case, $$\left(\varepsilon_{C(X)} \circ \Delta_X\right) (\alpha) = \varepsilon_{C(X)} \left(\Delta_X(\alpha)\right) = \varepsilon_{C(X)} \left(\bar f_1,\bar f_2,...\right) = \bar f_1(\id_\P).$$
	We must check that $\bar f_1(\id_\P)=f_1$ as maps $\P(1)\to X$. 
	Indeed, if $\theta \in \P(1)$, then $$\bar f_1(\id)(\theta) = g_1^1(\theta) = f_1\left(\gamma(\id;\theta)\right) = f_1(\theta).$$
\end{enumerate}
The proposition is therefore proven.
 \end{proof}

For the sake of completeness, 
we recall here the well-known fact that comonads explicitly create the cofree coalgebras of the underlying category 
(see for instance \cite[Corollary 5.4.23]{Per19}).

\begin{theorem} 
	\label{teo: Cofree coalgebras}
	Let $X$ be a pointed space. 
	Then, 
	$C(X)$ is the \emph{cofree $C$-coalgebra} on $X$.
	That is, for any $C$-coalgebra $A$ in pointed spaces, 
	there is a natural bijection 
	\begin{equation*}
		\Hom_{\mathsf{Top_*}}\left(A,X\right) \cong \Hom_{\mathsf{C-Coalg}}\left(A,C(X)\right).
	\end{equation*}
\end{theorem}

We will give a few explicit examples of how this comonad looks like in the case of the associative operad and the little $n$-cubes operad in Section \ref{sec: Comonad of little cubes}.

\subsection{Alternative definitions of a coalgebra over an operad}
\label{sec: Categories of coalgebras}

Let $\P$ be a unitary operad in $\mathsf{Top}$.
The comonad $C=C_\P$ constructed in Section \ref{sec: Topological comonads}  naturally  gives rise to a category of coalgebras in $\mathsf{Top_*}$. 
The objects in this category are pointed spaces $X$ together with a coalgebra structure map $c:X\to C(X).$
We call the objects of this category $\P$-coalgebras.
There is an equivalent way of defining a $\P$-coalgebra that does not require the explicit construction of the comonad $C$.
In this alternative definition, 
the objects are pointed spaces $X$ together with an operad map $\P\to \CoEnd_X$, 
where $\CoEnd_X$ is the coendomorphism operad asssociated to the pointed space $X$. 
In this section, we present the alternative definition of $\P$-coalgebra in terms of coendomorphisms,
and show that this is equivalent to the comonadic definition for unitary operads.
The definition of $\P$-coalgebras in terms of the coendomorphism operad is much more intuitive,
and defines the coalgebra structure in terms of explicit cooperations, i.e. maps $X \rightarrow X^{\vee r}$. 
On the other hand, the comonad definition has the benefit that it will be much easier to compare 
it to the $\Sigma^n \Omega^n$-comonad, 
making it more suitable for proving the approximation and  recognition theorems later in this paper.

\medskip

We start by defining the category of $\P$-coalgebras using the comonad $C_\P$.

\begin{definition} 
Let $\P$ be a unitary  operad in $\mathsf{Top}$.
The category $C_\P\mathsf{-Coalg}$ of coalgebras in $\mathsf{Top_*}$ associated to the
comonad $C_\P$ is called the \emph{category of (comonadic) $\P$-coalgebras.}
The objects in this category are triples $(X,c,\epsilon)$, 
where $c:X\to C(X)$, called the coalgebra structure map of $X$ and $\epsilon:C_\P(X) \rightarrow X$ its counit, 
are maps of pointed spaces satisfying counit and coassociativity axioms:
\begin{center}
\begin{tikzcd}
X \arrow[r, "c"] \arrow[rd, "\id"'] & C(X) \arrow[d, "\varepsilon_X"] &  & X \arrow[r, "c"] \arrow[d, "c"] & C(X) \arrow[d, "C(c)"] \\
                                    & X                               &  & C(X) \arrow[r, "\Delta_X"]      & C(C(X))               
\end{tikzcd}
\end{center} 
The morphisms between these objects are pointed maps $X\to Y$ that make the obvious squares commute.
\end{definition}

Before giving the alternative definition of  $\P$-coalgebras, 
we define the coendomorphism operad associated to a pointed space.

\begin{definition} 
\label{def:Coend operad}
Let $X$ be a pointed space. 
The \emph{coendomorphism operad} $\CoEnd_X$ in pointed topological spaces with the wedge sum
has arity $r$ component 
\begin{equation*}
	\CoEnd_X(r):=\Map_*\left(X,X^{\vee r}\right),
\end{equation*} 
the based mapping space from $X$ to the $r$-fold wedge sum of $X$ with itself. For $r=0$, set $\CoEnd_X(0)=*$. 
The operadic composition maps are defined as
\[
\gamma:\Map_*\left(X,X^{\vee n}\right)\times \Map_*\left(X,X^{\vee m_1}\right) \times \cdots \times \Map_*\left(X,X^{\vee m_n}\right) \rightarrow \Map_*\left(X,X^{\vee \sum m_i}\right),
\]
\[
\gamma\left( f,g_1,...,g_n\right):= \left(g_1\vee...\vee g_n\right) \circ f.
\]
The symmetric group action on $\CoEnd_X(r)$ permutes the wedge factors in the output of a map $f:X\to X^{\vee r}$. 
The unit $\eta : I \to \CoEnd_X$ is 
		determined by mapping the base point in $I(1) = \{*\}$ 
		to the identity map in $\CoEnd_X(1)=\Map_*\left(X,X\right)$.
\end{definition}

It is straightforward to check that $\CoEnd_X$ is an operad in pointed spaces,
and we leave this to the reader. 
The coendomorphism operad gives an alternative definition of  $\P$-coalgebras.
\begin{definition} \label{def:coendomorphism coalgebra}
Let $\P$ be a not necessarily unitary operad in $\mathsf{Top}$.
A $\P$-coalgebra is a pointed topological space $X$ together with an operad map $\P \to \CoEnd_X$. 
A morphism of $\P$-coalgebras is a pointed map $f:X\to Y$ such that the following diagram commutes 
 for all $n$:
\begin{center}
\begin{tikzcd}
	\P(n)\times X \arrow[d, "\id\times f"] \arrow[r, "\Delta_n"] & X\vee ... \vee X \arrow[d, "f \vee ... \vee f"] \\
	\P(n)\times Y \arrow[r, "\Delta'_n"]                         & Y \vee ... \vee Y                              
\end{tikzcd}	
\end{center} Here, $\Delta_n$ and $\Delta_n'$ are the coalgebra structure maps of $X$ and $Y$, respectively, which 
are written arity-wise by using the mapping space-product adjunctions 
$$\Map\left(\P(n)\times Z,Z^{\vee r}\right) \cong \Map\left(\P(n),\Map\left(Z,Z^{\vee r}\right)\right),$$ 
where $Z$ is any pointed topological space. 
Since we are mixing pointed and unpointed spaces,
we are viewing $\Map_*(X,X^{\vee r})$ as a subspace of the unpointed mapping space,
so that we are able to use the product-mapping space adjunction.
\end{definition}

\begin{remark}
Note that this definition of a $\P$-coalgebra is more general than the one using the comonad from the previous section. 
In particular, we do not require the operad to be unitary so these coalgebras are defined for a larger class of operads. 
\end{remark}

By using the product-mapping space adjunction for $S_r$-spaces, 
we see that  there are several equivalent ways of unpacking the definition of a coendomorphism $\P$-coalgebra. 
The definition of a coalgebra as a sequence of coproduct maps
\begin{equation*}
	\Delta_r:\P(r) \times X \rightarrow X^{\vee r}
\end{equation*}
is also equivalent to a sequence of maps
\begin{equation*}
	\Delta_r':X \rightarrow \Map\left(\P(r),X^{\vee r}\right)^{\S_r},
\end{equation*} 
satisfying certain conditions. Here $\Map(\P(r),X^{\vee r})^{\S_r}$ is the subspace of $\S_r$-invariant maps.

\medskip

Versions of the coendomorphism operad have been explicitly used before in for example \cite{Ber04} 
in the category of chain complexes. 
The notion of a coalgebra in the category of pointed spaces 
with the wedge product has also appeared before in \cite{Kle97}.
However, the authors do not use the coendomorphism operad or construct an explicit comonad.

\smallskip

The following result asserts that both definitions of $\P$-coalgebras are equivalent for unitary operads.

\begin{proposition}
\label{prop: Equivalent categories of coalgebras} 
Let $\P$ be a unitary operad in $\mathsf{Top}$.
Then the definition of a $\P$-coalgebra via the comonad from Section \ref{sec: Topological comonads} 
is equivalent to the definition via the coendomorphism operad 
from Definition \ref{def:coendomorphism coalgebra}.
\end{proposition}

\begin{proof}
Indeed, we can identify operad maps $\rho: \P \to \CoEnd_X$ with coalgebra structure maps $c:X\to C(X)$
by the following rule: for any $\theta \in \P(r)$ and $x\in X$,
$$\rho_r(\theta)(x) = f_r^x(\theta).$$
Here, $\rho_r$ is the arity $r$ component of $\rho$, and $f_r^x$ is the $r$th-term of the sequence $c(x)=\left(f_1^x,f_2^x,...\right)$.
The formula above turns a coendomorphism coalgebra into a comonad coalgebra and vice versa.
It is further straightforward to check that this identification commutes with morphisms.
\end{proof}

From now on, we always use the shorter notation $\P\mathsf{-Coalg}$ for the category of $\P$-coalgebras.

\begin{remark}\label{remark: Counital coalgebras}
The $\P$-coalgebras defined in this section are \emph{canonically counital}. 
That is, they come equipped with the unique map $\varepsilon : X\to *,$ 
and this map behaves as a counit with respect to the rest of the structure.
This explains the compatibility conditions of Equation (\ref{ecu:Compatibility with restriction operator and wedge collapse}).
Indeed, if $X$ is a $\P$-coalgebra, 
then the following diagram commutes:		
\begin{center}
	\begin{tikzcd}[column sep=2.5cm]
		P(n)\times X \arrow[r, "\Delta_r"] \arrow[d, "d_i \times \id"'] & X^{\vee r} \arrow[r, "\id \vee\cdots \vee \varepsilon \vee \cdots \vee \id"] & X^{\vee (r-1)} \arrow[d, "\id"] \\
		P(n-1)\times X \arrow[rr, "\Delta_{r-1}"]                    &                                                                           & X^{\vee(r-1)}                
	\end{tikzcd}
\end{center}
In the diagram above, $\Delta_r$ is the arity $r$ coalgebra structure map of $X$, 
and
$\id \vee\cdots \vee \varepsilon \vee \cdots \vee \id$ is precisely $\pi_i.$
Note that the counit of a coalgebra is \emph{unique}, i.e. since $*$ is the terminal object there is only one possible map from $X$ to $X^{\vee 0}=*$.
This is in high contrast with the (unpointed) algebra case,
in which there are many possibilities for a unit,
i.e. there are many maps from $X^{\times 0}=*$ to $X$,
since $*$ is not the initial object in unpointed spaces. 
\end{remark}

\subsection{The comonad associated to the little \texorpdfstring{$n$}{n}-cubes operad}
\label{sec: Comonad of little cubes}

In this section, we take a closer look at the comonad constructed 
in Section \ref{sec: Topological comonads},
in the particular case of $\P = \mathcal C_n$ being the little $n$-cubes operad.
Although we assume familiarity with this operad, there are a number of small variations in the literature. We give a brief summary below in order to carefully fix our conventions and establish the notation.
We will consistently denote by $C_n$ the comonad in pointed spaces associated to the little $n$-cubes operad $\mathcal C_n$. 
In Proposition \ref{Prop: Geometric characterization of C_n},
we give a geometric characterization of $C_n\left(X\right)$ as an explicit subspace of $\Map\left(\mathcal C_n(1),X\right)$.

\medskip

Denote by $I^n$ the unit $n$-cube of $\R^n$ and by $J^n$ its interior.
A \emph{little $n$-cube} is a rectilinear embedding $h: I^n \to I^n$ of the form $h = h_1 \times \cdots \times h_n$, where each component $h_i$ is given by
\begin{equation}
	\label{ecu: little cubes components}
h_i(t) = (y_i-x_i)t + x_i, \quad \textrm{ for } \quad  0\leq x_i < y_i \leq 1.
\end{equation} 
The image $h\left(J^n\right)$ of the interior of $I^n$ under a rectilinear embedding $h$ will be denoted $\mathring h$. So although the operad is called the little $n$-cubes operad it is technically the little $n$-rectangles operad.

\smallskip

For each $n\geq 1$, the \emph{little $n$-cubes operad} $\mathcal C_n$ is an operad in $\mathsf{Top}$.
It was introduced independently by Boardman--Vogt and May \cite{Boa73,May72} for studying iterated loop spaces. 
A comprehensive modern reference is \cite{Fre17}. 
We consider the unitary version of this operad, i.e., $\mathcal C_n(0)=*$ is the one-point space.
For each $r\geq 1$, the arity $r$ component $\mathcal C_n(r)$ of $\mathcal C_n$ is the subspace of the mapping space 
$$\mathcal C_n(r) \subseteq \Map\left(\coprod_{r} I^n, I^n\right)$$
given by those rectilinear embeddings for which the images of the interiors of different cubes are pairwise disjoint. 
That is, 
\begin{equation*}
	\mathcal C_n(r) = \left\{\left(c_1,...,c_r\right) \mid \textrm{each $c_i$ is a little $n$-cube, and $\mathring c_i\cap \mathring c_j\ = \emptyset$ for all $i\neq j$} \right\}.
\end{equation*}
The symmetric group $S_r$ acts on a configuration 
$c = \left(c_1,...,c_r\right)$ of little cubes  by permuting its components, 
$\left(c_1,...,c_r\right) \cdot \sigma = \left(c_{\sigma^{-1}(1)},...,c_{\sigma^{-1}(r)}\right)$.
The operadic unit $1\in \mathcal C_n(1)$ is the identity map $I^n \to I^n$, and the partial composition products are explicitly given by 
$$\left(c_1,...,c_r\right) \circ_ i \left(d_1,...,d_s\right) = \left(c_1,...,c_{i-1}, c_i\circ d_1,...,c_i \circ d_s,c_{i+1},...,c_r\right).$$
That is: we re-scale and insert the little $n$-cubes $d_1,...,d_s$ in place of the little $n$-cube $c_i$, 
which is removed, 
and then relabel accordingly.

\medskip

Recall from Proposition \ref{Prop:Characterization of C(X) as a subspace of Map} that the underlying space of the coalgebra 
$C_{\P}(X)$ 
associated to a unitary topological operad $\P$ and a pointed space $X$ 
is characterized as a certain subspace of $\Map\left(\P(1),X\right)$. 
In the particular case of the comonad $C_n$ associated to the little $n$-cubes operad,
there is a very geometrical characterization.
We need the following preliminary notation. 
First, recall that $$D_i = d_1\cdots \widehat {d_i} \cdots d_r : \mathcal C_n(r) \to \mathcal C_n(1)$$ 
denotes the composition of the restriction operators omitting the $i$-th term.
Evaluated at a configuration 
$c = \left(c_1,...,c_r\right)\in  \mathcal C_n(r)$, the map $D_i$ recovers the $i$-th little $n$-cube $c_i$.
Now, let  $X$ be a pointed space. 
Given  $f:\mathcal C_n(1) \to X$ any map, 
define for all $r\geq 2$ and $1\leq i \leq r$ the collection of maps 
\begin{equation}
    \label{ecu : maps f_r}
    f_r^i := f\circ D_i  : \P(r)\to X \quad \textrm{and} \quad f_r := \left(f_r^1,...,f_r^r\right) : \P(r)\to X^{\times r}.
\end{equation}
The mentioned characterization is the following.

\begin{proposition} 
	\label{Prop: Geometric characterization of C_n}
Let $X$ be a pointed space and $C_n$ be the comonad associated to the little $n$-cubes operad.
Then a map $f:\mathcal C_n(1)\to X$ belongs to $C_n(X)$ if, and only if, 
$f$ satisfies the following property:
	\begin{equation*}
		(D) \ \ \textit{If $c_1,c_2\in \mathcal C_n(1)$ are little $n$-cubes such that $\mathring c_1 \cap \mathring c_2 = \emptyset$, then $f(c_1)=*$ or $f(c_2)=*$.}
	\end{equation*}
\end{proposition}

It follows that taking 
$f=f_1$, each map $f_r$ in (\ref{ecu : maps f_r}) has its image in the $r$-fold wedge $X^{\vee r}$,
it is $S_r$-equivariant,
and the compatibility conditions $f_{r-1}d_i = \pi_if_r$ are satisfied for all $r\geq 2$ and $1\leq i \leq r$
if, and only if, condition $(D)$ is satisfied.

\begin{proof}
Assume $f = f_1 : \mathcal C_n(1) \to X$ satisfies  property $(D)$.
Fix an  arbitrary $r\geq 2$, and some $1 \leq i \leq r$.
Define $f_r : \mathcal C_n(r) \to X^{\times r}$ by 
$$f_r = \left(f_1D_1,...,f_1D_r\right).$$
Let us check that $f_r$ has its image in the wedge.
Indeed, for any $\theta = \left(c_1,...,c_r\right)\in \mathcal C_n(r)$, 
it follows from the definition of the space $\mathcal C_n(r)$ that $\mathring c_k \cap \mathring c_j= \emptyset$ for all $j\neq k$.
Furthermore, for each index $j$ between $1$ and $r$, we can write  
$$c_j = \left(d_1 \circ \cdots \widehat{d_j} \cdots \circ d_r\right)\left(\theta\right) = D_j(\theta).$$
Therefore, condition $(D)$ applied to each pair $(j,k)$ with $j\neq k$
implies that at most a single component $f_1(c_j)$ is not the base point. 
Said differently: 
$f_r$ has its image in the wedge.
The map $f_r$ is $S_r$-equivariant.
Indeed, for any $\sigma\in S_r$, one has 
$$f_r \cdot \sigma = \left\{f_1D_1,...,f_1D_r\right\} \cdot \sigma = 
\left\{f_1D_1 \cdot \sigma,...,f_1D_r\cdot \sigma\right\}  =
\left\{f_1D_\sigma(1) ,...,f_1D_\sigma(r)\right\} = \sigma \cdot \left\{f_1D_1,...,f_1D_r\right\}.$$
Since $\sigma$ permutes the coordinates of the wedge factors, the claim is proven.

For the converse, assume that $\left(f_1,f_2,...\right) \in C_n(X),$ and that 
$c_1,c_2 \in \mathcal{C}_n(1)$ are little $n$-cubes such that $\mathring{c}_1\cap \mathring{c}_2 = \emptyset$.
This is precisely the condition needed to ensure that $(c_1,c_2)$ is an element of $C_n(2)$.
Consider $f_2\left(c_1,c_2\right)\in X\vee X$. From the comonadic compatibility conditions, one has 
\[
f_1\left(c_1\right) = \pi_1 f_2\left(c_1,c_2\right) \quad \textrm{and} \quad f_1\left(c_2\right) = \pi_2 f_2\left(c_1,c_2\right).
\]
Therefore, one of $f_1\left(c_1\right)$ or $f_1\left(c_2\right)$ must be the base point. It follows that $f_1$ satisfies property $(D)$.
\end{proof}

In the next remark, 
we point out the obvious fact that non-trivial strictly coassociative coalgebras do not exist in pointed spaces.

\begin{remark}
\label{Remark: Trivial coassociative coalgebras}
Recall that a pointed space $X$ is a co-H-space
if it comes equipped with a map $c:X\to X\vee X$ that is a factorization up to homotopy of the identity map $X\to X$:
\begin{center}
	\begin{tikzcd}
		X \arrow[r, "c"] \arrow[rd, "\id"'] & X\vee X \arrow[d, "q_i"] \\
		& X                       
	\end{tikzcd}
\end{center}
That is, $q_1c \simeq \id \simeq q_2c$, where $q_i:X \vee X \to X$ is the projection onto the $i$-th factor of the wedge.
If we try to strictify this diagram,  considering $q_1c = \id = q_2c$, 
then for any $x\in X$ we would have the following situation. 
The coproduct $c(x)$ is either a point in the first wedge factor, $(x_1,*)$, 
or it is a point in the second wedge factor, $(*,x_2)$.
Without loss of generality, we may assume  that it is of the form $c(x) = (x_1,*)$,  
we would then have 
$$q_2c(x) = q_2(x_1,*) = *.$$
If $X$ has more than one point, 
we will not have $q_2c(x)=x$ for $x\neq *$.
Thus, the unique strictly coassociative counital coalgebra is the one point space.
This is a significant contrast with the algebra case, where for example,
the James construction \cite{Jam55} gives a strictly associative monoid in pointed spaces modeling $\Omega\Sigma  X$. 
The classical Moore loop space is another important example of a pointed space endowed with a strictly associative product.
We conclude that there is no possible "rectification" of 
a counital homotopy coassociative coalgebra into a counital strictly coassociative coalgebra.  
Aside from the elementary proof given here, 
the non-existence of strictly coassociative coalgebras in $\mathsf{Top_*}$ will also follow from  Proposition \ref{Prop: Reduced operads produce trivial comonads},
a more general statement asserting that reduced operads produce trivial comonads, 
leaving no place for non-trivial  counital coassociative coalgebras.
Remark that it is the counit that is causing all the problems in the discussion above.
Since there are non-trivial non-counital strictly coassociative coalgebras, the argument above does not apply. 
It is therefore not known whether strictly coassociative 
rectifications exist in the case of non-counital coalgebras,
but this is beyond the scope of this paper.
\end{remark}

The particular instance of Theorem \ref{teo: Cofree coalgebras} in this case gives the following important observation.

\begin{theorem} 
	Let $X$ be a pointed space.
	Then, $C_n\left(X\right)$ is the \emph{cofree $C_n$-coalgebra} on $X$.
	That is, for any $C_n$-coalgebra $A$, there is a natural bijection
	\begin{equation*}
		\Hom_{\mathsf{Top_*}}\left(A,X\right) \cong \Hom_{C_n\mathsf{-Coalg}}\left(A,C_n(X)\right).
	\end{equation*}
\end{theorem}

\subsubsection{Reduced  topological operads and weak equivalences}

In this section, 
we prove that for
\emph{reduced} unitary topological operads (i.e. $\P(1) = \{*\}$) ,
the comonad $C_\P$ is always trivial, in the sense that it is a one-point space when evaluated at any pointed space.
Therefore, 
the associated category of $\P$-coalgebras is trivial
(Proposition \ref{Prop: Reduced operads produce trivial comonads}).
This is a striking difference with the construction of $C_n$ in the case of the little $n$-cubes operad $\mathcal C_n$,
whose category of coalgebras is rich and interesting.
As a consequence,
we readily see that the comonad construction does not respect weak equivalences
in the Berger--Moerdijk model structure \cite{Ber03} on topological operads.
That is, 
if $\P \to \mathcal Q$ is a morphism of unitary operads in $\mathsf{Top_*}$
which is a weak equivalence in each arity, 
it does not necessarily follow that the induced map $C_{\P}(X) \to C_{\mathcal Q}(X)$ is a weak equivalence for each pointed space $X$.
For example, the associative operad $\mathsf{Ass}$ is reduced, 
producing a trivial category of coalgebras,
but there is a a well-known weak equivalence of operads $\mathcal C_1 \twoheadrightarrow \mathsf{Ass}$.
Said differently,
a weak equivalence of unitary operads does not imply an equivalence of categories of coalgebras
(even of up to homotopy coalgebras).

\begin{proposition}
\label{Prop: Reduced operads produce trivial comonads}
	If $\P$ is a reduced  unitary topological  operad, 
	then $C_{\P}$ is the trivial comonad. 
	That is, $C_{\P}(X)$ is the one-point space for all pointed spaces $X$.
	In particular, the comonads $C_{\mathsf{Ass}}$ and $C_{\mathsf{Com}}$ 
	produced respectively from the associative and commutative operads are trivial.
\end{proposition}

\begin{proof}
	Let $\P$ be an operad as in the statement. 
	Fix a pointed space $X$, and consider an arbitrary sequence $\alpha = \left(f_1,f_2,...\right) \in C_{\P}(X)$.
	Then, 
	$$f_1 : \P(1)\to X$$ 
	specifies some point $f_1(*)=x_0\in X$.
	Recall (Lemma \ref{Lemma:Determined by 1st comp}) 
	that the higher terms $f_r$ in the sequence $\alpha$ are determined by the recursive formula 
	\begin{equation}\label{Ecu:recur in proof}
		f_r = \left\{f_1D_1 , ..., f_1D_r\right\}.
	\end{equation}
	In particular, for any $\theta \in \P(2)$,
	$$f_2(\theta) = \left\{f_1d_2(\theta), f_1d_1(\theta)\right\} = \{x_0,x_0\}.$$
	Therefore, for $f_2$ to be well-defined (i.e., having its image in the wedge),
	the point $x_0$ must be the base point of $X$.
	It then follows from the recursive formula (\ref{Ecu:recur in proof}) that for all $r\geq 2$ and $\theta\in \P(r)$,
	we have
	$$f_r(\theta) = \left\{f_1D_1(\theta),...,f_1D_r(\theta)\right\} = \left\{x_0,...,x_0\right\}.$$
	That is, $\alpha$ is the trivial sequence.
\end{proof}

\subsection{Iterated suspensions are coalgebras over the little cubes operad}\label{sec:Iterated suspensions are coalgebras}

In this section, we show that the $n$-fold reduced suspension $\Sigma^n X$ of 
a pointed space $X$ is a coalgebra over the little $n$-cubes operad. 
These spaces are the paradigmatic examples of $\mathcal C_n$-coalgebras.
To show our results, we use the coendomorphism version of $\mathcal C_n$-coalgebras.
At the end of the section, we explain how the results in this paper allows us to swiftly recover the classical $\mathcal C_n$-algebra structure on $n$-fold loop spaces as a convolution structure.
The $\mathcal C_n$-coaction on $S^n$ that we describe in this section has previously appeared, in the context of the factorization homology, in \cite{ginot12}.

\begin{theorem} \label{teo: Iterated suspensions are coalgebras}
	The $n$-fold reduced suspension of a pointed space $X$ is a $\mathcal C_n$-coalgebra.
	More precisely, 
	there is a natural and explicit operad map
	\begin{equation*}
		\nabla: \mathcal{C}_n \rightarrow \coend_{\Sigma^n X}
	\end{equation*} 
that encodes the homotopy coassociativity and homotopy cocommutativity of the classical pinch map $\Sigma^nX \to \Sigma^nX \vee \Sigma^nX$.
In particular, the pinch map  is an operation associated to an element of $\mathcal C_n(2)$. 
Furthermore, for any based map $X\to Y$, the induced map $\Sigma^n X \to \Sigma^n  Y$ extends to a morphism of \ $\mathcal C_n$-coalgebras.
\end{theorem}

To prove the result above, we proceed in two steps. 
First, we prove it for connected spheres, which are particular cases of reduced suspensions.
That is, we show that the sphere $S^n$ is a coalgebra over the little $n$-cubes operad for every $n\geq1$.
This is Proposition \ref{prop: Spheres} below. 
Building on top of this preliminary result, 
we go on to prove Theorem \ref{teo: Iterated suspensions are coalgebras},
extending the result from connected spheres to any $n$-fold reduced suspension.

\begin{proposition}
\label{prop: Spheres}
	For every $n\geq 1$, there is a natural and explicit morphism of operads 
	$$\nabla:\mathcal{C}_n\rightarrow \coend_{S^n}$$ turning the $n$-sphere into a $\mathcal{C}_n$-coalgebra, 
	so that all properties of Theorem \ref{teo: Iterated suspensions are coalgebras} for $\Sigma^n X=S^n$ hold true. 
\end{proposition}
\begin{proof}

	 Let us define the arity $r$ component of $\nabla$. 
	 This is a map $$\nabla_r: \mathcal C_n(r)\to \CoEnd_{S^n}(r) = \Map_*\left(S^n, S^n \vee ... \vee S^n\right).$$
	 For $c = \left(c_1,...,c_r\right)\in \mathcal C_n(r)$ a configuration of little $n$-cubes, we
	 define the pointed map 
	 \begin{center}
	 	\begin{tikzcd}[row sep=tiny]
	 		\nabla_r(c) : S^n \arrow[r] & \left(S^n\right)^{\vee r} \\
	 		t \arrow[r, maps to]        & \nabla_r(c)(t)       	\end{tikzcd}
	 \end{center}
 as follows. 
Identify $S^n = I^n /\partial I^n$.
Then $t\in S^n$ is either the base point $t = \{\partial I^n\}$ or else it is an interior point of the $n$-cube $I^n$.
If $t$ is interior, 
then there is at most a single cube $c_i$ such that $t\in \mathring{c_i}$.
Define
$$\nabla_r(c)(t) = 
\begin{cases} \left[c_i^{-1}(t)\right] &\mbox{if } t\in \mathring{c_i}, \\
	 	* & \mbox{otherwise.}
\end{cases}$$
Here, $\left[c_i^{-1}(t)\right]$
denotes the class of $c_i^{-1}(t)$ as the corresponding 
point in the $i$-th factor of the wedge $S^n \vee ... \vee S^n$.
So defined, the maps $\nabla_r(c)$ are pointed, continuous and turn $\nabla$ into a morphism of operads. 
These last assertions are straightforward to check and left to the reader.
\end{proof}

We prove next that the little $n$-cubes coalgebra structure on the sphere $S^n$ just described 
induces the little $n$-cubes coalgebra structure on an arbitrary $n$-fold reduced suspension.

\medskip

{\noindent \it Proof of Theorem \ref{teo: Iterated suspensions are coalgebras}:} 
Let $\Sigma^nX$ be the $n$-fold reduced suspension of a pointed space $X$.
Write $\Sigma^nX=S^n\wedge X$,
and recall that for any three pointed spaces $A$, $Y$ and $Z$, 
the wedge and smash product distribute over each other \cite[S. 4.F ]{Hat02}, i.e.
\begin{equation*}
	A \wedge \left( Y \vee Z \right) \cong \left(A  \wedge Y \right) \vee \left( A  \wedge Z \right).
\end{equation*} 
In particular, when we take $A$ to be $S^n$,
\[
\Sigma^n\left(Y \vee Z\right) \cong \Sigma^n Y \vee \Sigma^n Z.
\]
Now, for $c\in \mathcal{C}_n(r)$, 
define the map 
$\Sigma^nX\to \left(\Sigma^nX\right)^{\vee r}$ 
as the composition
\begin{equation*}
	\Sigma^nX \cong S^n \wedge X \xrightarrow{\nabla_r(c) \wedge \id_X} \left( \left(S^n\right)^{\vee r}\right) \wedge X \xrightarrow{\cong} \left( S^n \wedge X \right)^{\vee r} \cong \left(\Sigma^nX\right)^{\vee r},
\end{equation*} where $\nabla_r$ is the arity $r$ component of the map $\nabla$ defined in Proposition \ref{prop: Spheres}. 
All these maps are continuous,
commute with the symmetric group actions and the operadic composition maps, producing a functorial construction.
Alternatively, one can define the operad map $$\CoEnd_{S^n} \to \CoEnd_{\Sigma^nX}$$  given (up to isomorphism) by $f \mapsto f\wedge \id_X$, and precompose it with the operad map of Proposition \ref{prop: Spheres}. 
Doing this, one ends up with the map we described before.
In this sense, the $\mathcal C_n$-coalgebra structure of an $n$-fold suspension always factors through the $\mathcal C_n$-coalgebra structure of $S^n.$
\hfill$\square$ \\

\begin{remark}
The defined operad map $\nabla: \mathcal C_n \to \CoEnd_{\Sigma^nX}$  
is determined by its arity $1$ component $\nabla_1: \mathcal C_n(1)\times \Sigma^nX \to \Sigma^n X$. 
Being more precise, 
it is a consequence of Proposition \ref{Lemma:Determined by 1st comp} that
the following formula holds for all $c \in \mathcal C_n(r)$ and $z\in \Sigma^n X$: 
	$$\pi_i\left( \nabla_r\left(c,z\right)\right) 
	= \nabla_{r-1}\left(d_i\left(c\right),z\right),$$ 
where $\pi_i$ and $d_i$ are the wedge collapse and restriction operators from Section \ref{sec: Topological comonads}.
\end{remark}

In the remainder of the section, 
we explain how the coalgebraic framework 
introduced in this work let us swiftly recover the classical result by May that iterated loop spaces are algebras over the little 
$n$-cubes operad. For this, we first need to define fold algebras in the category of pointed spaces  with the wedge product $\vee$.

\begin{definition}
Let $X$ be a pointed space.
The \emph{fold endomorphism operad} $\operatorname{End}^{\vee}_X$ is  the operad whose arity $r$ component is given by 
\[
\operatorname{End}^{\vee}_X(r)=\Map_*\left(X^{\vee r},X\right),
\]
with the composition map given by inserting the output of a map 
into the input,
and the symmetric group action is given by permuting the inputs.
If  $\P$ is an operad in unpointed spaces, 
then a fold $\P$-algebra is a pointed space $X$ together with a morphism of operads $\P \to \End^\vee_X$.
Equivalently: $\operatorname{End}^{\vee}_X$ is the endomorphism operad in the category of pointed spaces together 
with the wedge product as symmetric monoidal structure.
\end{definition}

We leave it to the reader to check that the definition above gives an operad. 
Every pointed space is canonically a commutative fold-algebra,
where the products are given by the canonical fold maps (which explains the name).

\smallskip

Let $\P$ and $\mathcal Q$ be operads in unpointed spaces.
We can now define a convolution algebra in pointed spaces between a $\P$-coalgebra and a fold $\mathcal Q$-algebra
by making use of the definition of a fold $\mathcal Q$-algebra.
Denote by $\P \times \mathcal Q$ the arity-wise product of $\P$ and $\mathcal Q$.

\begin{proposition}
Let $\P$ and $\mathcal Q$ be operads in unpointed spaces.
Let $X$ be a $\P$-coalgebra and $Y$ a fold $\mathcal Q$-algebra. 
Then the pointed mapping space
$\Map_*\left(X,Y\right)$ is a $\P\times \mathcal Q$-algebra.
The structure maps
$$\gamma: \P(r)\times \mathcal Q(r) \times \Map_*(X,Y)^{\times r} \to \Map_*(X,Y)$$ 
applied to pointed maps $f_1,...,f_r : X\to Y$ and operations $(\theta,\nu)\in \P(r) \times \mathcal Q(r)$ is
explicitly given by 
\[
\gamma\left((\theta,\nu);f_1,...,f_r\right) = \left(\nu \circ \left(f_1 \vee \cdots \vee f_r\right)\circ \Delta\right) (\theta).
\]
Here,  $\Delta:\P \to \operatorname{CoEnd}_X$ is the $\P$-coalgebra structure map of $X$.
\end{proposition}

\begin{proof}
This is similar to the construction in Section 1 of \cite{Ber03} and is left to the reader.
\end{proof}

In particular, $n$-fold loop spaces fall into the framework described in the previous result.
Since every pointed space is canonically a commutative fold algebra,
and the arity-wise product of $\mathcal C_n$ with the commutative operad is isomorphic to $\mathcal C_n$, 
we recover May's classical $\mathcal C_n$-algebra structure  on loop spaces as follows (see \cite{May72}).

\begin{corollary}
	Let $\Omega^n X$ be an $n$-fold loop space.
	Then, 
	the $\mathcal C_n$-algebra 
	structure on 
	$$\Omega^n X = \Map_*\left(S^n ,X\right)$$
	induced by the $\mathcal C_n$-coalgebra structure of $S^n$ and the fold $\mathsf{Com}$-algebra structure on $X$
	as a convolution algebra
 is exactly the  classical $\mathcal C_n$-algebra structure on loop spaces.
\end{corollary}

\begin{proof}
	By definition, each map $S^n \to S^n \vee \cdots \vee S^n$ arising from the $\mathcal C_n$-coalgebra structure of $S^n$ induces the following convolution product on an $n$-fold loop space $\Omega^n X$.
	Given $\alpha_1,...,\alpha_r:S^n\rightarrow X$ and $\theta \in \mathcal{C}_n(r)$,  
	define $\gamma(\alpha_1,...,\alpha_r)$ as 
	\[
	S^n \xrightarrow{\nabla(\theta)} \left( S^n \right)^{\vee r} \xrightarrow{\alpha_1 \vee ... \vee \alpha_r} X^{\vee r} \xrightarrow{\mu_r} X,
	\]
	where $\mu_r\in \mathsf{Com}(r)$ is the $r$th fold map. 
	Here, $\mathsf{Com}$ is the commutative operad.
	One checks that these maps are exactly the maps described in \cite[Section 5]{May72}.
\end{proof}

\section{The Approximation Theorem}\label{sec:approximation theorem}

To prove the recognition principle for $n$-fold loop spaces,
as well as to develop a unified theory of homology operations for them, 
May proved the \textit{approximation theorem} \cite[Theorem 6.1]{May72}.
The proof of this result consists of giving a morphism of monads from the monad $M_n$ 
associated to
the little $n$-cubes operad to the monad $\Omega^n \Sigma^n$, 
and proving that this natural transformation is a homotopy equivalence on  connected spaces.
In this section, 
we prove an Eckmann--Hilton dual result to approximate the comonad $\Sigma^n \Omega^n$. 

\begin{theorem} 
	\label{teo: Approx Theorem}
	For every $n\geq 1$, there is a natural morphism of comonads 
	$$\alpha_n : \Sigma^n \Omega^n \longrightarrow C_n.$$
	Furthermore, 
	for every pointed  space $X$, 
	there is an explicit natural deformation retract of pointed spaces
	\begin{center}
\begin{tikzcd}
\Sigma^n \Omega^n X \arrow[r, shift left] & C_n(X) \arrow[l, shift left] \arrow[loop, distance=2em, in=125, out=55]
\end{tikzcd}
	\end{center}
	In particular, $\alpha_n(X)$ is a homotopy equivalence.
\end{theorem}

The proof of the result above does not consist of a dualization of the corresponding proof of May's proof in the case of loop spaces.
We take a different route which has the 
advantages of giving explicit homotopies and not requiring auxiliary spaces as is needed in May's original approach. 
Furthermore: we produce a homotopy equivalence, not just a weak equivalence as in the case of loop spaces.
It is not  clear at the moment whether the methods employed in this paper can be used
to give an alternative proof of the loop space approximation theorem.

\medskip

Let $n\geq 1$ be a fixed integer.
The natural transformation $\alpha = \alpha_n:\Sigma^n \Omega^n \to C_n$
is defined object-wise as the composition
 \begin{equation*}
	\alpha_X : \Sigma^n \Omega^n X \xrightarrow{\gamma} C_n\left(\Sigma^n \Omega^n X\right) \xrightarrow{C_n\left(\eta_X\right)} C_n\left(X\right),
\end{equation*}
where $\gamma$ 
is the $\mathcal C_n$-coalgebra structure map of $\Sigma^n \Omega^nX$ (Theorem \ref{teo: Iterated suspensions are coalgebras}), 
and $\eta_X $ is the 
evaluation at $X$ of the counit  $\eta: \Sigma^n\Omega^n \to \id_{\mathsf{Top_*}}$ of the $\left(\Sigma^n,\Omega^n\right)$-adjunction.
Unraveling the definitions, we readily see that $\alpha = \alpha_X$
is explicitly given on a point $[t,\ell]\in \Sigma^n \Omega^n X = S^n\wedge \Map_*\left(S^n,X\right)$ as the map $\alpha[t,\ell] : \mathcal C_n(1)\to X$
that acts on a little $n$-cube $c\in \mathcal C_n(1)$ by
$$\alpha[t,\ell](c) =
\begin{cases} \ell\left(c^{-1}(t)\right) &\mbox{if } t\in \mathring{c}\\
*	& \mbox{otherwise} \end{cases}$$
See  Proposition \ref{Prop: Morphism of comonads} for more details on the definition of $\alpha$.

\medskip

The proof of Theorem \ref{teo: Approx Theorem} consists of the following two steps:

\smallskip

$(i)$  We must check that $\alpha$ defines a morphism of comonads. 
	This is not complicated, but it is lengthy.
	Because of this, we postponed this proof to Appendix \ref{sec: Appendix on morphism of comonads} (Proposition \ref{Prop: Morphism of comonads}).
	\medskip

$(ii)$ We must check that for a fixed pointed space $X$, 
the space $\Sigma^n \Omega^n X$ is a deformation retract of $C_n\left(X\right)$ in the category of pointed spaces.
To do so, we give a pointed map (of spaces, not comonads) $\Psi = \Psi_n: C_n\left(X\right) \to \Sigma^n \Omega^n X$ and a homotopy $\mathcal{H}=\mathcal H_n:C_n(X) \times I \to C_n(X)$  such that 
	\begin{equation}
	    \label{ecu: Identities for Psi}
	    \Psi \circ \alpha = \id_{\Sigma^n \Omega^n X} \quad \textrm{ and }\quad  \alpha \circ \Psi \simeq \operatorname{id}_{C_n\left(X\right)}.
	\end{equation}
To define $\Psi$ and the homotopy $\mathcal H :\alpha\circ \Psi \simeq \operatorname{id}_{C_n\left(X\right)}$, 
we introduce for each $f\in C_n\left(X\right)$ a certain subset of the $n$-cube $I^n$ 
which we name the \emph{cubical support of $f$} and denote  $\operatorname{CSupp}\left(f\right)$.
The cubical support of a map $f$ has a well-defined \emph{center},
which is a point 
\[
\operatorname{Cent}\left(f\right) \in  \operatorname{CSupp}\left(f\right) \subseteq I^n.
\]
The cubical support and its center will play an important role in the deformation retract.

\smallskip

Theorem \ref{teo: Approx Theorem} will then follow from the two items just described.
Since the first item is proved in the appendix, it remains to prove the second one.
To do so, we give the details of the auxiliary construction of the cubical support and its center 
in Section \ref{sec: Center and cubical support},
and then prove the assertions of item $(ii)$ in Section \ref{sec: proof of item (ii)}.

\subsection{The cubical support of a map and its center}
\label{sec: Center and cubical support}
For each pointed space $X = (X,*)$, 
the cubical support is a map 
\[
\operatorname{CSupp}:  C_n\left(X\right)\to \overline{\mathcal C_n\left(1\right)}
\]
defined on the complement of the constant map, denoted by $*$, and where $\overline{\mathcal C_n\left(1\right)}$
is the topological closure of $\mathcal C_n\left(1\right)$.
This closure is built by attaching the limits of shrinking cubes, which are rectangles that have been squeezed in some dimensions.
In particular, these limits can be singletons. 
See Figure \ref{fig: squeezed cubes}, where we represent four different elements of $\overline{\mathcal C_2(1)}$.
The space $\overline{\mathcal C_n\left(1\right)}$ is compact.
The map $\operatorname{CSupp}$ is crucial to this paper, playing a fundamental role in Theorem \ref{teo Intro 2}. 
We shall show a quotient of it is continuous, 
\[
\overline{\operatorname{CSupp}}:  C_n\left(X\right)/ L(X)\to \overline{\mathcal C_n\left(1\right)}/\overline{A}.
\]
The construction of the subspaces $L(X)$ and $\overline{A}$ is technical and will be explained along the section.
We carefully define this map and check its continuity. 
Then, we compute some examples and prove some necessary technical results.

\begin{figure}[h!]\centering
	\includegraphics[scale=0.3]{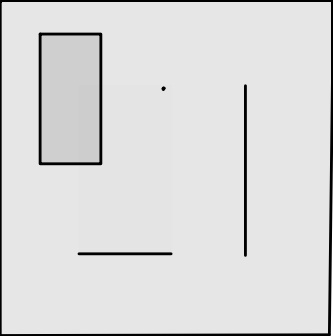}
        \caption{{\footnotesize{A little 2-cube, a singleton, and two little 2-cubes  squeezed at different dimensions in $\overline{\mathcal C_2(1)}$.}}}
        \label{fig: squeezed cubes}
\end{figure}

We first lay out some notation and technical results needed to define the cubical support maps.
Recall that any pointed space $\left(X, \ast\right)$ in our category of spaces is a neighbourhood deformation retract 
(NDR), see \cite[Appendix A]{May72}. 
This means there is a continuous function $u: X\to [0,1]$ such that $u^{-1}\left(\{0\}\right) = \{\ast\}$,
and a homotopy $h:X\times I \to X$  that retracts $u^{-1}\left([0,1)\right)$ onto $\{\ast\}$. 
In particular, the unit interval $I = (I,0)$ pointed at 0 is a NDR. 
The cubical support function on the interval will play a prominent role,
since the proof of the continuity of the cubical support map for a general space will be reduced to the continuity of 
the cubical support map for the interval. 
To that end,
we fix once and for all 
the identity map $v:I\to I$ and the homotopy $K:I\times I \to I$ given by $K(x,t) = (1-t)x$ as NDR maps on the interval.

\smallskip

We define the following maps on $\mathcal C_n\left(1\right)$ and on its compactification. 
\begin{itemize}
    \item The \emph{radius} map $\operatorname{rad}: \overline{\mathcal C_n\left(1\right)} \to [0,1]$, 
    whose value at a (possibly squeezed) little $n$-cube $c$ is the smallest side length of $\operatorname{Im}(c)$,
    that is, $\operatorname{rad}(c) = \min_i\left\{y_i- x_i\right\}$.
    The restriction of this map to $\mathcal C_n\left(1\right)$ is denoted in the same way, and it is a non-zero function.
    That is, $\operatorname{rad}(c)> 0$ for all little $n$-cubes $c\in \mathcal C_n(1)$.
    This is because $c$ is a linear embedding with non-empty interior at each of its coordinates.

\item The \emph{distance to the boundary} map $d_b: \overline{\mathcal C_n\left(1\right)} \to [0,1]$ 
on each (possibly squeezed) little $n$-cube $c$ as the smallest distance 
between any edge of $\operatorname{Im}(c)$ and any edge of $I^n$, that is, $d_b(c) = \min\left\{x_i, 1-y_i\right\}$.

\item The \emph{opposite side distance} map $d: \mathcal C_n\left(1\right) \to (0,1]$ by the formula $d=\operatorname{rad} + d_b$.
Since $\operatorname{rad}(c)>0$ for every little $n$-cube $c$, it follows that $d(c) > 0$ for every little $n$-cube $c$ too.
On the other hand, $d_b(c)$ can be at most $1-\operatorname{rad}(c)$ by definition.
Therefore, 
\[
d(c) = \operatorname{rad}(c) + d_b(c)  \leq \operatorname{rad}(c) + \left(1-\operatorname{rad}(c)\right) = 1.
\]
We shall denote by $d'$ the extension of this function to the
compactification by the same formula, $d':\overline{\mathcal C_n\left(1\right)}\to [0,1]$.
Although $d$ is a non-zero function, the extension $d'$ vanishes on cubes that have been shrunk to a point
or on little cubes lying on the boundary that have been squeezed in some of its dimensions.

\item The map $q: \mathcal C_n\left(1\right) \to (0,1]$ given on each little $n$-cube $c$ as the ratio
\[
q(c) = \frac{\operatorname{rad}(c)}{d(c)}.
\]
Since $d(c)>0$, the quotient is well-defined. 
Furthermore, as $\operatorname{rad}$ is a positive function on $\mathcal C_n(1)$ 
and $d=\operatorname{rad}+d_b$ with $d_b$ non-negative, it follows that $q(c)$ takes values in $(0,1]$.
\end{itemize}

The following two subsets will play an important role in our proof. 
First, the set 
\[
A = \left\{c\in  \mathcal C_n\left(1\right)\ \mid\ \operatorname{rad}(c) = d(c) \right\} 
\]
of the little $n$-cubes where at least one edge touches some side of the large ambient cube,
that is, cubes $c$ such that $\operatorname{Im}(c) \cap \partial I^n \neq \emptyset$.
The expression for $A$ above occurs because $\operatorname{rad}(c) = d(c)$ if, and only if, $d_b(c)=0$,
and we have chosen it because it will be useful later on to express this fact using the maps $\operatorname{rad}$ and $d$.
Second, the set
\[
\overline{A} = \left\{c\in  \overline{\mathcal C_n\left(1\right)}\ \mid\ \operatorname{rad}(c) = d'(c) \right\}
\]
of the (possibly squeezed) little $n$-cubes in the closure with the same property, i.e., 
such that at least one edge touches (or is contained in) some side of the large ambient cube $I^n$.

Remark that $q(c) = 1$ is equivalent to $c \in A$,
and that $\operatorname{rad}, d,$ and $q$ are non-negative functions, while $d'$ is not.
This will be used in the sequel. 
\smallskip

The space $\mathcal C_n\left(1\right)$ is contractible.
Fix $H: \mathcal C_n\left(1\right)\times I \to \mathcal C_n\left(1\right)$ to be any choice of homotopy 
such that $H\left(-, 0\right)$ is the identity map,
$H\left(-, 1\right)$ is the constant map sending everything to $\id \in \mathcal C_n\left(1\right)$,
and furthermore $\operatorname{Im}(H(c, t))\subseteq \operatorname{Im}(H(c, s))$ if $t<s$.
Finally, we fix a \textit{compactification function} $G: \mathcal C_n\left(1\right) \to \mathcal C_n\left(1\right)$, 
which is defined as any continuous map such that:
\begin{enumerate}
    \item 
    The map $G$ is the identity when restricted to $A$.

    \item For each sequence of cubes outside of $A$ on which the function $q$ tends to zero,
    the image of the map $G$ on this sequence tends to the identity cube.

    \item There is an inclusion on images $\operatorname{Im}(c)\subseteq \operatorname{Im}\left(G(c)\right)$ 
    for every little $n$-cube $c\in \mathcal C_n\left(1\right)$.
\end{enumerate}
The geometric interpretation of the three conditions above are explained in Appendix \ref{sec: The explicit description of the map G}.
There, we further construct an explicit map $G$. 
The precise details of this construction are irrelevant to the remainder of this section,
as only the abstract properties of $G$ given above are needed to define the cubical support map.

\smallskip

The \emph{cubical support} of any map  $f:\mathcal C_n\left(1\right)\to X$ is defined as
\begin{equation}
\label{ecu: New cubical support}
	\operatorname{CSupp}\left(f\right) 
 = \bigcap_{\substack{c\in \mathcal C_n\left(1\right)}} 
 \operatorname{Im}\left( H\left( G(c), 1 - \max \left( u\left( f\left( G(c) \right) \right), q(c) \right) \right) \right)
\subseteq I^n.
\end{equation} 
Since $\operatorname{CSupp}\left(f\right)$ is defined as the intersection of the images of a family of cubes,
it may degenerate to a singleton or to a rectangle with at least one side of length zero; 
and thus it can be canonically identified with an element of $\overline{\mathcal C_n(1)}$, see Proposition~\ref{claim}.
For our purposes, it suffices to consider the cubical support of maps in the comonad $C_n(X)$. 
Consequently, the natural codomain of the (naive) cubical support map
is the topological closure $\overline{\mathcal C_n(1)}$ rather than $\mathcal C_n(1)$,
\begin{equation*}
    \operatorname{CSupp}:  C_n\left(X\right)\to \overline{\mathcal C_n\left(1\right)}.
\end{equation*}
Our next goal is to prove the continuity of a quotient of this map, which we define as the (genuine) cubical support map,
\[
\overline{\operatorname{CSupp}}: C_n\left(X\right)/ L(X)
\xrightarrow{} \overline{\mathcal C_n\left(1\right)}/ \overline{A}.
\]
The set $\overline{A}$ has already been defined. 
The subset of the domain we will quotient by is 
\[
L(X) = \left\{f \in C_n(X) \mid  \operatorname{CSupp}\left(f\right) \in \overline{A}\right\} \cup \left\{*\right\},
\]
where $*$ denotes the constant map. 
The first step to prove the continuity of the genuine cubical support map is contained in the following observation.
There, we use a subindex in the naive cubical support function because there are two different spaces involved.

\begin{lemma}
\label{lemma: Cubical support factors}
    For any pointed space $(X,*)$ and choice of NDR map $u:X\to I$, 
    the naive cubical support function on $X$ factors through the cubical support function on the interval $I$.
    That is, there is a commutative diagram 
    \[
    \begin{tikzcd}
    C_n\left(X\right) \arrow[rd, "\operatorname{CSupp}_X", swap] \rar{C_n\left(u\right)}&   C_n(I)
    \arrow[d, "\operatorname{CSupp}_I"] \\
     \ &  \overline{\mathcal C_n\left(1\right)}.
\end{tikzcd}
    \]
Therefore, the map $\operatorname{CSupp}_X$ is continuous if, and only if, 
the map $\operatorname{CSupp}_I$ is continuous.
\end{lemma}

\begin{proof}
Recall from Prop. \ref{Prop:Characterization of C(X) as a subspace of Map} that $C_n\left(u\right)\left(f\right) = u \circ f$.
With the choice made of NDR maps for the interval, 
the factorization follows at once, because for every little $n$-cube $c$,
we have 
\[
u\left(f\left(G(c)\right)\right) = v\left(u\circ f\left(G(c)\right)\right). \vspace{-5mm}
\]
\end{proof}
By definition, $\operatorname{CSupp}\left(L(X)\right)\subseteq  \overline{A}$.
Therefore, the diagram in Lemma \ref{lemma: Cubical support factors} above factors to produce a diagram 
    \[
    \begin{tikzcd}
    C_n\left(X\right)/ L(X) \arrow[rd, "\overline{\operatorname{CSupp}}_X", swap] \rar{C_n\left(u\right)}&   C_n(I)/ L(I)
    \arrow[d, "\overline{\operatorname{CSupp}}_I"] \\
     \ &  \overline{\mathcal C_n\left(1\right)} / \overline{A}.
\end{tikzcd}
    \]    
We define the \emph{cubical support map} of the pointed space $X$ as the quotient map above,
\[
\overline{\operatorname{CSupp}}: C_n\left(X\right) / L(X)
\xrightarrow{} \overline{\mathcal C_n\left(1\right)}/ \overline{A}.
\]
We are ready to prove the continuity of this map.

\begin{proposition}
\label{prop: The cubical support is continuous}
For every pointed space $X$, the cubical support map $\overline{\operatorname{CSupp}}$ is continuous.
\end{proposition}

\begin{proof}
As in the Lemma \ref{lemma: Cubical support factors}, the map $\overline{\operatorname{CSupp}}_X$ is continuous if, and only if, 
the map $\overline{\operatorname{CSupp}}_I$ is continuous.
Thus, it suffices to prove the case of $X = I$.
Recall that
the space $\overline{\mathcal C_n\left(1\right)}$ admits a metric induced by that on $I^n$, see Equation \eqref{ecu: distance on C_1(1)} 
for the 1-dimensional case.  
Choose any $f \in C_n(I)$ such that $\operatorname{CSupp}\left(f\right) \notin \overline{A}$. To prove continuity, 
we will show that for all $\epsilon > 0$ there exists an open set $U \subseteq C_n(I)$ such that $f\in U$ and 
$\operatorname{CSupp} \left(U\right) \subseteq B_{\epsilon}\left(\operatorname{CSupp}\left(f\right)\right)$,
where $B_{\epsilon}(-)$ is an open ball of radius $\epsilon$. 
Define $U$ as follows.
From the definition of $G$, there exists $\delta$ such that, 
for any cube $c$ such that $q(c)< \delta $, we have $G(c) \in B_{\epsilon}\left(\id\right).$ 
Therefore, 
$H\left(G(c), 1- \operatorname{max}\left( u\left(f\left(G(c)\right), q(c) \right)\right)\right) \in B_{\epsilon}\left(\id\right)$.
Now, define $\varepsilon = \min(\delta, d_b(\operatorname{CSupp}\left(f\right)))$. 
This is nonzero because, by our prior assumption, $\operatorname{CSupp}\left(f\right) \notin \overline{A}$. 
Let $\overline{\mathcal C_n\left(1\right)}_{\varepsilon}\subseteq \overline{\mathcal C_n\left(1\right)}$ 
be the subspace formed by cubes $c$ such that $ q(c)\geq \varepsilon$ 
and such that $d_b(\operatorname{CSupp}\left(f\right)) \geq \epsilon$. 
This subspace is compact.
Since $f$ is continuous,  there is a collection of open sets 
\[
\left\{ U_c: f\left(U_c\right) \subseteq B_{\epsilon}\left(f(c)\right)\right\}_{c\in \overline{\mathcal C_n\left(1\right)_{\varepsilon}}}.
\]
This is an open cover of $\overline{\mathcal C_n\left(1\right)}_{\varepsilon}$,
and therefore it admits a finite subcover $\left\{U_{c_i}\right\}_{i=0}^N$. 
Let $K_{c_i} = \overline{U_{c_i}}$ be the closure of $U_{c_i}$. 
This produces a collection of compact sets $\left\{K_{c_i}\right\}_{i=0}^N$.
Define
$$U := \bigcap_{i = 0}^N \Map\left(K_{c_i}, \left(\operatorname{max}\left(0, f(c_i) - \frac{\epsilon}{2}\right), \operatorname{min}\left(f(c_i)+ \frac{\epsilon}{2}, 1 \right)\right)\right).$$
Let $g\in U.$ Then 
$$\operatorname{CSupp} \left(g\right) 
=  \bigcap_{\substack{c\in \mathcal C_n\left(1\right)}} 
\operatorname{Im}\left(H\left(G(c), 1 - \max\left(u\left(g\left(G(c)\right)\right), q(c)\right)\right)\right).$$
By construction, 
$g$ differs from $f$ at most $\epsilon$ at each coordinate. 
Being precise, for each cube $c\in \mathcal C_n(1)$,
the rectangles 
\[
\operatorname{Im}\left(H\left(G(c), 1 - \max\left(u\left(g\left(G(c)\right)\right), q(c)\right)\right)\right) 
\quad \textrm{and} \quad 
\operatorname{Im}\left(H\left(G(c), 1 - \max\left(u\left(f\left(G(c)\right)\right), q(c)\right)\right)\right)
\]
differ by at most $\epsilon$ at each coordinate. 
This means that each factor of the intersection above 
belongs to the open ball $B_{\epsilon}\left(\operatorname{CSupp}\left(f\right)\right)$.
Therefore, the intersection of all of them belongs to the same open ball, 
because the distance at each coordinate is bounded by the distance of any factor of the intersection,
which is always strictly smaller than $\epsilon$.
This proves that 
$\operatorname{CSupp} \left(g\right) \in B_{\epsilon}\left(\operatorname{CSupp}\left(f\right)\right)$, as we wanted to show. 
\end{proof}

Our next goal is to calculate some explicit examples and prove some technical results about the 
cubical support. 
The following lemma will make these calculations easier.

\begin{lemma} 
\label{lemma: old cubical support contained in new cubical support}
Let $X$ be a pointed space.
For each map $f\in C_n(X)$ such that $\operatorname{CSupp}\left(f\right)\notin L(X)$, there is an inclusion of sets 
\[
\bigcap_{\substack{c\in \mathcal C_n(1) \\ f(c)\neq *}} \operatorname{Im}(c) \subseteq \operatorname{CSupp}\left(f\right).
\]
\end{lemma}

\begin{proof}
    The intersection $\bigcap_{\substack{c\in \mathcal C_n(1) \\ f(c)\neq *}}\operatorname{Im}(c)$ can be written as 
    $
    \bigcap_{c\in \mathcal C_n(1)} D_c,
    $
    where 
    $$
    D_c = \begin{cases}
        \operatorname{Im}(c) & \mbox{if } f(c)\neq *
        \\
        I^n & \mbox{if } f(c) = *.
    \end{cases}
    $$
    The cubical support $\operatorname{CSupp}\left(f\right)$ is also defined as an intersection over all 
    $c\in \mathcal C_n(1)$,
    with each factor given 
    as $\operatorname{Im}\left( H\left( G(c), 1 - u\left( f\left( G(c) \right) \right) \right) \right).$ 
    Moreover, one has
    $$
    D_c \subseteq \operatorname{Im}\left( H\left( G(c), 1 - \max\left(u\left(f\left(G(c)\right)\right), q(c)\right) \right) \right),
    $$
    as $\operatorname{Im} (c) \subseteq \operatorname{Im} \left(G(c)\right)$ and $\operatorname{Im} (c) \subseteq \operatorname{Im} (H(c,t))$ 
    for all $t\in I$. 
    The conclusion follows.
\end{proof}

From now on, 
denote 
\[
\operatorname{CSupp}'\left(f\right) =\bigcap_{\substack{f\in C_n(X) \\ c\in \mathcal C_n(1) \\ f(c)\neq *}} \operatorname{Im}(c).
\]
By the lemma before, 
$\operatorname{CSupp}'\left(f\right) \subseteq \operatorname{CSupp}\left(f\right)$.
This inclusion is an equality in some situations, but there are exceptions
(see Examples \ref{exa: cubical support computations}). 
The reason to introduce the set $\operatorname{CSupp}'\left(f\right)$ is that it is easier to compute in practice than
$\operatorname{CSupp}\left(f\right)$.
It is also worth noting that,
whenever $\operatorname{CSupp}\left(f\right)$
is a singleton, the set $\operatorname{CSupp}'\left(f\right)$ is so as well.

\smallskip

The following observation is essential.
Recall that $A\subseteq \mathcal C_n(1)$ is the subset of little $n$-cubes $c$ whose image $\operatorname{Im}(c)$
intersects the boundary $\partial I^n$, and that $G:\mathcal C_n(1)\to \mathcal C_n(1)$ is a fixed compactification function.

\begin{proposition}
For each map $f\in C_n(X)$ such that $\operatorname{CSupp}\left(f\right)\notin L(X)$,
there is an inclusion
\[
\operatorname{CSupp}\left(f\right) \subseteq \bigcap_{\substack{c\in A \\ f(c)\neq *}} 
\operatorname{Im}(c).
\]
\end{proposition}

\begin{proof}
For each little $n$-cube $c$, we have:
\[
\operatorname{Im}(c) \subseteq \operatorname{Im}\left(G(c)\right) \subseteq
\operatorname{Im} \left(H\left(G(c),t \right)\right)  \quad \text{for all } t.
\]
Hence,
\[
\operatorname{Im}(c) \subseteq \operatorname{Im}
\left( H\left( G(c), 1 - \max\left( u\left(f\left(G(c)\right)\right),\ q(c) \right) \right) \right).
\]
Therefore,
\[
\begin{aligned}
\operatorname{CSupp}\left(f\right) 
&= \bigcap_{\substack{c \in \mathcal{C}_n(1)\\ f(c)\neq *}} 
\operatorname{Im} \left( H\left( G(c), 1 - \max\left( u\left(f\left(G(c)\right)\right),\ q(c) \right) \right) \right) \\
&\subseteq \bigcap_{\substack{c \in A \\ f(c) \neq * }}
\operatorname{Im} \left( H\left( G(c), 1 - \max\left( u\left(f\left(G(c)\right)\right),\ q(c) \right) \right) \right) \\
&= \bigcap_{\substack{c \in A \\ f(c) \neq * }}
\operatorname{Im}H\left( c, 1 - \max\left( f(u(c)),\ q(c) \right) \right).
\end{aligned}
\]
The final inclusion follows from the fact $G(c) = c$ for $c\in A.$ Finally, as $q(c) = 1$ for $c\in A,$ we have 
\[
= \bigcap_{\substack{c \in A \\ f(c) \neq * }}
\operatorname{Im} \left(H\left( c, 1 - \max\left( u\left(f\left(c\right)\right),1 \right) \right)\right) 
= \bigcap_{\substack{c \in A \\ f(c) \neq * }}
\operatorname{Im}\left(H\left( c, 0 \right)\right) =  \bigcap_{\substack{c \in A \\ f(c) \neq * }} \operatorname{Im}(c). \vspace{-2em}
\] 
\end{proof}

Our next goal is to be more explicit about the shape of $\operatorname{CSupp}'\left(f\right)$,
and hence of $\operatorname{CSupp}\left(f\right)$.
Recall that 
an \emph{$n$-rectangle} is a subspace of $\mathbb{R}^n$ which is rectilinearly homeomorphic to $I^n$ or a singleton.
An $n$-rectangle that does not reduce to a single point is determined by the set of its $2^n$ vertices, 
but also more efficiently by $2n$ numbers that describe the length of the sides and their position.
In other words, an $n$-rectangle $R$ is simply a cartesian product of closed intervals:
$$R= \left\{\left(x_1,...,x_n\right)\in \mathbb{R}^n \mid a_i \leq x_i \leq b_i \textrm{ for all } i=1,...,n\right\}=\left[a_1,b_1\right]\times \cdots \times \left[a_n,b_n\right],$$
for certain $a_i,b_i\in \mathbb{R}$ satisfying $a_i\leq b_i$.

\begin{proposition} 
\label{claim}
Let $f\in C_n\left(X\right)$.
The set $\operatorname{CSupp}'\left(f\right)$ is nonempty and an $n$-rectangle.
\end{proposition}

\begin{proof}

We start by proving the furthermore assertion, since we will explicitly use the description 
of $\operatorname{CSupp}'\left(f\right)$ to prove the first claim.
To do so,
let $c\in \mathcal C_n(1)$ be a little $n$-cube, 
and write $c = \left(g_1,...,g_n\right)$ in terms of its coordinate functions $g_i : I \to I$.
Then, the image of the cube $c$ is the $n$-rectangle
$$\operatorname{Im}(c) = \left[g_1(0),g_1(1)\right] \times \cdots \times \left[g_n(0),g_n(1)\right] \subseteq I^n.$$
There is an identification between little $n$-cubes and $n$-rectangles contained in $I^n$ that do not either reduce to a single point or alternatively have at least one side-length equal to 0. 
For a fixed map $f:\mathcal C_n(1)\to X$, we have the $n$-rectangle
$$\operatorname{CSupp}'\left(f\right) = \left[a_1,b_1\right] \times \cdots \times \left[a_n,b_n\right],$$
where for each $i=1,...,n$
\begin{align*}
	a_i &:= \sup \left\{g_i(0) \mid c = \left(g_1,...,g_n\right)\in \mathcal C_n(1) \textrm{ and } f(c)\neq *\right\},\\[0.2cm]
		b_i &:= \inf \ \  \left\{g_i(1) \mid c = \left(g_1,...,g_n\right)\in \mathcal C_n(1) \textrm{ and } f(c)\neq *\right\}.
\end{align*}
Now, we proceed to prove the main characterization in the proposition.
Let $f\in C_n(X)$ be any map.
If $\operatorname{CSupp}'\left(f\right) \neq \emptyset$, then obviously $f\neq *$.
Let us check the converse.
Assume therefore that $f\neq *$, and let us check that $\operatorname{CSupp}'\left(f\right) \neq \emptyset$. 
By the explanation given before of how the cubical support is constructed, 
we have at each coordinate $i$ that $\sup_j a^j_i \leq \inf_j b^j_i$ for all $j=1,...n$.
Therefore, there is a point $t=(t_1,...,t_n)$ such that $a^j_i \leq t_i \leq  b^j_i$  for all $j=1,...n$.
This finishes the proof.
\end{proof}

Next, we define the \emph{(rectangular) center map} 
$$
\operatorname{rCent}: \overline{\mathcal C_n(1)} \backslash   \overline{B} \to I^n/\partial I^n,
$$
where 
\begin{equation}
    \label{ecu: Subspace Q}
    \overline{B} = \left\{\left[a_1,b_1\right]\times \cdots \times \left[a_n,b_n\right] \in \overline{\mathcal C_n(1)}
\ \mid \ \exists i \textrm{ with } [a_i,b_i] = [0,1]\right\}. 
\end{equation}
Explicitly, the \emph{center} $\operatorname{rCent}(R)$ is the point 
\begin{equation}
    \label{ecu: rcenter}
    \operatorname{rCent}(R) 
= \left(a_1 + \frac{2}{\pi}\tan^{-1}\left( \frac{a_1}{1-b_1} \right)(b_1-a_1),...,a_n +\frac{2}{\pi}\tan^{-1}\left( \frac{a_n}{1-b_n} \right)(b_n-a_n) \right).
\end{equation}
Here, we adopt the convention of defining $\tan^{-1}\left(\frac{a_i}{0}\right) = \frac{\pi}{2}$ for $a_i\neq 0$.
We define the cubical center of such points to be $\ast$.
In particular,
if 
$R=(x_1,...,x_n)$ is a singleton, 
then $\operatorname{rCent}(R) = (x_1,...,x_n)$.
Assuming furthermore that  $R = \operatorname{CSupp}\left(f\right)$ for some $f$, 
then we define $\operatorname{Cent}\left(f\right),$ the center of $f$, as 
$$\operatorname{Cent}\left(f\right) := \operatorname{rCent}\left(\operatorname{CSupp}\left(f\right)\right)=\operatorname{rCent}(R).$$

\begin{remark}
\label{rem:discontinuity}
The map $\operatorname{rCent}$ is continuous on $\overline{\mathcal C_n(1)} \backslash  \overline{B}$,
but its extension to all of $\overline{\mathcal C_n(1)}$ by the same formula will not generally 
be continuous at $\overline{B}$ itself.
This turns out not to matter to prove the continuity of the map $\Psi : C_n\left(X\right) \to \Sigma^n \Omega^n X$ in Proposition
\ref{prop:psi}, for which $\operatorname{rCent}$ is an auxiliary map.
This is because, if $\overline{\operatorname{CSupp}}\left(f\right) \in \overline{B},$
then $\Psi\left(f\right)$ is the base loop in $\Omega^n X$, 
which is quotiented in the smash product in Proposition \ref{prop:psi}.
\end{remark}

\begin{examples}
\label{exa: cubical support computations}
Let us compute $\operatorname{CSupp}'\left(f\right)$ in several cases. 
	\begin{enumerate}
		
		\item Let $C_n(*)$ be the cofree counital $C_n$-coalgebra on a single point. 
		Then, $C_n\left(*\right)=*$  reduces to the trivial one-point space. 
		Thus, the unique map $f:\mathcal C_n(1)\to *$ collapses all little $n$-cubes to the base point, 
		and therefore, $\operatorname{CSupp}'\left(f\right) = \emptyset$.
            On the other extreme, 
            $\operatorname{CSupp}\left(f\right) = I^n$ is the unit $n$-cube.
		
		\item 
		Consider the map $f: \mathcal C_1(1)\to I$ given by 
		$$f(c) = 
		\begin{cases} 0 &\mbox{if } r \leq 1/2 \\
			r-1/2 & \mbox{if } r \geq 1/2 \end{cases}$$
		Here, $r = c(1)-c(0)$ is the size of the little $1$-cube $c$. 
		By Proposition \ref{Prop: Geometric characterization of C_n}, 
		$f$ defines an element in $C_1\left(I\right)$,
		and one readily checks that $\operatorname{Cent}\left(f\right) 
        = \operatorname{CSupp}'\left(f\right) = \left\{\frac{1}{2}\right\}$.
		
		\item Define $f: \mathcal C_1(1)\to I$ as in the example above replacing $1/2$ by any real number $a \in \left[\frac{1}{2},1\right).$  
		By Proposition \ref{Prop: Geometric characterization of C_n},
		$f$ defines a map in $C_1\left(I\right)$.
		It can be seen that $\operatorname{CSupp}'\left(f\right)=[1-a,a]$.
	\end{enumerate}
The examples above can be generalized to higher-dimensional cubes.
\end{examples}

An important example of cubical support is that of $n$-fold suspensions.


\begin{proposition}
\label{prop: suspensions have a single point cubical support}
Let $\Sigma^nX$ be the $n$-fold reduced suspension of a pointed space $X$,
and let $\gamma: \Sigma^nX \to C_n\left(\Sigma^nX\right)$ be its $\mathcal C_n$-coalgebra structure map.
Then, for every non-base point $[t,x] \in \Sigma^n X$, we have that 
\[
\operatorname{CSupp}\left(\gamma[t,x]\right) 
= \left\{ t\right\}.
\]
\end{proposition}

\begin{proof}
First, 
we prove the result  for spheres.
If $\gamma : S^n \to C_n\left(S^n\right)$ is the $\mathcal C_n$-coalgebra structure map, 
we explicitly have
\begin{equation*}
	\gamma(t)(c) =
	\begin{cases} c^{-1}(t) &\mbox{if } t\in \mathring c\\
		*	& \mbox{otherwise,} \end{cases}
\end{equation*}
where $t\in S^n$ and we identify $S^n$ with $I^n/\partial I^n$, the ambient cube of $c$ modulo its boundary. 
As observed earlier,  
$\operatorname{CSupp}\left(\gamma(t)\right)$
is, at most, the intersection of  the family 
$$
\operatorname{CSupp}\left(f\right) \subseteq \bigcap_{\substack{c\in A \\ f(c)\neq *}} \operatorname{Im} c.
$$
The image $\operatorname{Im}(c)$ of a little $n$-cube is non-trivial if, and only if, $t\in \operatorname{Im}(c)$.
Thus, the cubical support $\operatorname{CSupp}\left(\gamma(t)\right)$ is, 
at most, the intersection of all non-trivial cubes containing $t$ where $\operatorname{Im}(c)$ \emph{of maximal radius}. 
Since it is nonempty, it is precisely $\{t\}$.

Now, for an arbitrary $n$-fold reduced suspension $\Sigma^n X$, factorize its coalgebra structure map as follows:
\begin{center}
	\begin{tikzcd}
		\Sigma^nX = S^n \wedge X \arrow[r, "\ \ \gamma_{S^n}\wedge \id_X \ \  "] & C_n\left(S^n\right)\wedge X \arrow[r, "F"] & C_n\left(S^n\wedge X\right).
	\end{tikzcd}
\end{center}
The second map $F$ above is given by
$$
F\left(f,x\right) = \left[f(-),x\right], \quad \textrm{ for } f : \mathcal C_n(1)\to S^n \textrm{ and } x\in X.
$$ 
The final composition is therefore explicitly given by 
\begin{center}
	\begin{tikzcd}[row sep=small]
		{\gamma[t,x] : \mathcal C_n(1)} \arrow[r] & S^n \wedge X                  \\
		c \arrow[r, maps to]                                   & {\left[\gamma(t)(c),x\right]}.
	\end{tikzcd}
\end{center}
Here, the cubical support $\operatorname{CSupp}\left(\gamma[t,x]\right)$ is, at most,  the intersection of  the family 
\[
\left\{\operatorname{Im}(c) \mid c\in  \mathcal C_n(1) \ \textrm{and } \ 
\left[\gamma(t)(c),x\right]\neq *\ \textrm{and } \operatorname{rad}(c) = r_0\right\}.
\]
Similar to the case of the spheres, we have 
\begin{equation*}
	\left[\gamma(t)(c),x\right] =
	\begin{cases} \left[c^{-1}(t),x\right] &\mbox{if } t\in \mathring c\\
		*	& \mbox{otherwise} \end{cases}
\end{equation*}
We readily see from here that a little $n$-cube $c$ has non-trivial image if, 
and only if, 
$\mathring{c}$ contains the component $t$ of the sphere.
Thus, the intersection of them all yields the singleton $\{t\}$.  
\end{proof}

\subsection{Proof of Theorem \ref{teo: Approx Theorem}}
\label{sec: proof of item (ii)}

In this section, we prove the remaining details of Theorem \ref{teo: Approx Theorem}.
This includes the definitions of the maps, their continuity, and the identities involved in the deformation retract.

\medskip

\noindent \fbox{\strut \ Definition of $\Psi$ \ }

\medskip

\noindent The pointed map $\Psi$ is defined as follows: 
\begin{center}
	\begin{tikzcd}[row sep=small]
		\Psi : C_n\left(X\right) \arrow[r] & \Sigma^n \Omega^n X \\
		f \arrow[r, maps to]    & {\Psi\left(f\right)=\left[\operatorname{Cent}\left(f\right),\ell\right]}.
	\end{tikzcd}
\end{center}
Note that for any $f \in L(X)$, we have $\operatorname{Cent}\left(f\right) = \ast$, and thus $\Psi\left(f\right)$ is the base point of $\Sigma^n \Omega^n X$.
Under the identification $\Sigma^n \Omega^n X = S^n\wedge\Map_*\left(S^n,X\right)$,
the two components above are
$$\operatorname{Cent}\left(f\right)\in S^n \quad \textrm{ and } \quad \ell:S^n \to X, \quad s\mapsto \ell(s):=f\left(c_{s,\operatorname{Cent}\left(f\right)}\right).$$
Abusing the notation, $\operatorname{Cent}\left(f\right)$ above denotes the corresponding 
point in the quotient $S^n = I^n/\partial I^n$.
This will be further developed below.
On the other hand, 
the little $n$-cube $c_{s,\operatorname{Cent}\left(f\right)}$ that depends on both $f$ and $s$,
follows a certain construction to be explained below too.

\smallskip

\begin{proposition}
\label{prop:psi}
    The map $\Psi$ is continuous.
\end{proposition}
\begin{proof}
We define the subspace $K \subseteq C_n(X)$ as the set of maps $f$ such that $\operatorname{CSupp}\left(f\right)\notin \overline{B}$.
On this subspace, the center map $\operatorname{Cent}$ is continuous by construction,  
see the explicit formulae in Equation \eqref{ecu: rcenter}. 
Consequently, the restriction of $\Psi$ to $K$ is continuous, as it is the composition of continuous maps:
    \[
    K \xrightarrow{ \ \operatorname{Cent} \times \id } S^n\times K \xrightarrow{\id \times \ell_{(-)}} S^n\times \Omega^n X \longrightarrow S^n \wedge \Omega^n X.
    \]
    Furthermore, the center map $\operatorname{Cent}: C_n(X) \to S^n$ factors as the composition
    \[
    C_n\left(X\right) \to C_n\left(X\right) / L(X)
    \xrightarrow{\ \overline{\operatorname{CSupp}}\ } \overline{\mathcal C_n\left(1\right)}/ \overline{A} 
    \xrightarrow{\ \operatorname{rCent} \ } I^n/\partial I^n \cong S^n,
    \]
    which are continuous on $K.$
    It remains to check continuity for $f$ with $\operatorname{CSupp}\left(f\right) \in \overline{B}$. 
	If this is the case, then $f \in L(X)$ and so by definition, $\operatorname{Cent}\left(f\right) = \ast$.
	We claim that $f$ is identically the constant map to the basepoint.
	Indeed, since $\overline{\operatorname{CSupp}}\left(f\right) \in \overline{B}$, 
    there exists at least one dimension $i \in \{1, \dots, n\}$ such that the $i$-th projection 
    of $\operatorname{CSupp}\left(f\right)$ is the full interval $[0,1]$.
	Recall from Lemma \ref{lemma: old cubical support contained in new cubical support} that 
    $\operatorname{CSupp}'\left(f\right) \subseteq \operatorname{CSupp}\left(f\right)$.
	Therefore, for any cube $c \in \mathcal{C}_n(1)$ such that $f(c) \neq \ast$,
    the image $\operatorname{Im}(c)$ must span the entire $i$-th axis.
	However, any cube containing a whole axis can be approximated by a sequence of 
    interior cubes $\left\{c_k\right\}$ which do not contain the full axis.
	For such interior cubes, $f(c_k) = \ast$.
	By the continuity of $f$, it follows that $f(c) = \lim_{k \to \infty} f(c_k) = \ast$.
	Consequently, $f$ is the constant map, and the loop associated to $f$
    by $\Psi$ is the constant loop $\ell_c$.
    Next, we must show that for any sequence $g_n \to f \equiv \ast$ in $C_n(X)$, 
    the sequence $\Psi\left(g_n\right)$ converges to the basepoint in $\Sigma^n \Omega^n X$.
	By Proposition \ref{prop: The cubical support is continuous},
    the cubical support map is continuous, so $\operatorname{CSupp}\left(g_n\right)$ becomes arbitrarily close to $\overline{B}$.
	This implies that for the loop $\ell_{g_n}(s) = g\left(c_{s, \operatorname{Cent}\left(g\right)}\right)$
    to evaluate to a non-basepoint value, the chosen cube $c_{s, \operatorname{Cent}\left(g\right)}$ 
    must have an $i$-th projection arbitrarily close to $[0,1]$.
	By the explicit construction of $c_{s,t}$ in Lemma \ref{lemma: The cube c_{s,t}}, 
    this geometric constraint forces the coordinate $s_i$ to be arbitrarily close
    to the $i$-th coordinate of $\operatorname{Cent}\left(g\right)$.
	Therefore, the non-trivial domain of the loop $\ell_{g_n}$ is confined to an arbitrarily small
    neighbourhood of a hyperplane.
	Simultaneously, since $g_n$ converges to the constant map $\ast$ in the compact-open topology, 
    the values of $\ell_{g_n}(s)$ uniformly approach the basepoint $\ast$.
	It follows that the loop $\ell_{g_n}$ converges to the constant loop in $\Omega^n X$.
	Since the target space is the smash product $\Sigma^n \Omega^n X = S^n \wedge \Omega^n X$,
    where $S^n \times \{\ell_c\}$ is collapsed to the basepoint, we conclude that $\Psi\left(g_n\right)$ converges to the basepoint.
	This establishes the desired continuity.
\end{proof}

We also need the following auxiliary result.
It explicitly describes the little $n$-cube $c_{s,\operatorname{Cent}\left(f\right)}$
that appears in the loop $\ell: S^n \to X$ of the second component of $\Psi$.

\begin{lemma}
\label{lemma: The cube c_{s,t}}
    For each pair of points $s,t\in I^n \backslash \partial I^n$,
there is a unique little $n$-cube $c=c_{s,t}: I^n \rightarrow I^n$, depending continuously on $(s,t)$, such that: 
\begin{enumerate}
	\item $c(s)=t$,
	\item $\operatorname{Im}(c)$ is the largest $n$-rectangle contained in $I^n$ 
    requiring that for each coordinate, 
    at least one side of the embedded rectangle touches the corresponding side of the ambient cube.
    \end{enumerate}
\end{lemma}
If $s$ or $t$ lies in the boundary $\partial I^n$, we will not need to construct the cube $c_{s,t}$.
Indeed, in this case $\Psi$ will map the pair $[t,\ell]$ to the base point of $C_n(X)$.

\begin{proof}[Proof of Lemma \ref{lemma: The cube c_{s,t}}]
Let us explicitly construct $c$.
Recall from Equation (\ref{ecu: little cubes components})
that the rectilinear embedding $c$ is of the form 
\begin{equation*}
	c\left(x_1,...,x_n\right) = \big(\left(b_1-a_1\right)x_1 + a_1 , ... , \left(b_n-a_n\right)x_n + a_n\big),
\end{equation*}
where $0\leq a_i < b_i \leq 1$ for all $i$.
Thus, each component $c_i$ of $c$ is determined by the numbers $a_i$ and $b_i$.
Imposing that $c(s)=t$, we get the relations
$$(b_i-a_i)s_i + a_i = t_i \quad \textrm{for each $i$}.$$
A second constraint on each component $i$ determines the numbers $a_i,b_i$ uniquely.
Since $c$ touches each face of $\partial I^n$,  
at each component $c_i$ we must have one of the following two options:
\begin{enumerate}
\item $c_i(0) = 0$, and then we deduce that $$ c_i(x_i)=\frac{t_i}{s_i}\cdot x_i,$$ or else  
\item $c_i(1) = 1$, and then we deduce that 
$$c_i(x_i)= x_i+(1-x_i)\left(\frac{s_i-t_i}{s_i-1}\right).$$ 
\end{enumerate}
Now, 
there is no choice to be made here.
Rather, the case is determined by the relationship between $s$ and $t$.
That is, we are considering the separate cases where $s_i > t_i$ or $s_i < t_i$.
More precisely,
if for a fixed $i$, we have that $0< t_i/s_i<1$, 
then the first formula gives a well-defined affine linear map onto the interval, 
but the second formula does not (because its image lands outside the unit interval).
If on the contrary the inequality $0< t_i/s_i<1$ does not hold,
then the first formula does not work, while the second does.
To finish, 
observe that the formulae agree when $s_i=t_i$, 
which makes the construction of $c$ a continuous function of $s$ and $t$.
Of course, in the case $s_i=t_i$,
we are taking the identity map at the $i$-th coordinate.
This finishes the proof.    
\end{proof}
Our arguments so far show that the resulting function is a pointed continuous function of $f$.

\medskip

\noindent  \fbox{\strut \ Definition of the homotopy $\mathcal H$ \ }

\medskip

\noindent The next step in the proof of the approximation theorem is to construct a homotopy $\mathcal H:C_n\left(X\right)\times I \to C_n\left(X\right)$
such that
\begin{equation}\label{ecu: identities for H}
	\mathcal H_0 = \id_{C_n\left(X\right)}, \qquad \mathcal H_1 = \alpha \circ \Psi, \qquad \textrm{and} \qquad  \mathcal H(*,t) = * \ \forall t \in I.
\end{equation}

\medskip

The following auxiliary construction is a key ingredient for the homotopy $\mathcal H$. 
Intuitively speaking, the idea is to construct a homotopy from maps whose cubical support is more than a point to maps whose cubical support is exactly a point. We construct this homotopy by enlarging the cubes in $C_n(1)$ until they hit the boundary while also preserving the center. This is made precise in the following auxiliary construction.

\medskip

\noindent \emph{Auxiliary construction: The rectilinear expansion 
of a little $n$-cube $c\in \mathcal C_n(1)$ 
induced by a map $f\in C_n\left(X\right)$ whose center $\operatorname{Cent}\left(f\right)$ belongs to $\mathring{c}$.}

\medskip

\noindent \emph{Proof and explanations
for the auxiliary construction:} 
Let us explain the construction for a little $1$-interval $c\in \mathcal C_1(1)$; 
the general case is an application of this construction at each coordinate of a little $n$-cube.
Let $c\in \mathcal C_1(1)$, so that 
$$c(t) = (b-a)t+a$$ for some $a,b$ with $0\leq a< b\leq 1$.
Let
$$\operatorname{dist}\left(\Img(c),\partial I\right) = \min\{a,1-b\}$$
be the distance from $\Img(c)$ to the boundary of the interval.
\begin{definition}
\label{def: rectilinear expansion}
Let $c\in \mathcal C_1(1)$.
The \emph{rectilinear expansion} of $c$ induced by a map $f\in C_1(X)$ whose center $\operatorname{Cent}\left(f\right)$ 
belongs to $\mathring{c}$ is the unique path $\gamma = \gamma_c^f : I \to \mathcal C_1(1)$ 
satisfying:
\begin{itemize}
\item $\gamma(0) = c$,
\item for every $s\in (0,1]$, 
\begin{itemize}
    \item the cube $\gamma(s)$ is the rectilinear embedding that increases 
the diameter of $c$ by $\min\left\{s,\operatorname{dist}\left(\Img(c),\partial I\right)\right\}$ 
while keeping the ratios between the sides equal, and
\item the center $\operatorname{Cent}\left(f\right)$ is preserved by $\gamma(s)$, 
in the sense that if $t_0\in I$ is the unique point such that $c(t_0)=\operatorname{Cent}\left(f\right)$,
then $\gamma(s)(t_0) = \operatorname{Cent}\left(f\right)$.
\end{itemize}
\end{itemize}
\end{definition}
Let us explicitly describe the path $\gamma$ above. 
For each $s\in I$, we have $\gamma(s)\in \mathcal C_1(1)$ of the form 
\begin{equation*}
	\gamma(s)(t) = \left(b(s)-a(s)\right)t + a(s) \quad \forall \ t\in I.
\end{equation*}
For a fixed $s\in I$, 
the two mentioned conditions on $a(s)$ and $b(s)$ determine $\gamma(s)$ uniquely.
These conditions are the following. 
First, that
$$\gamma(s)\left(\frac{p-a}{b-a}\right) = p,$$
where for simplicity we denote $p=\operatorname{Cent}\left(f\right)$.
Secondly,
that the radius of $\gamma(s)$ is that of $c$ increased by 
$\min \left\{s,a,1-b\right\}$:
\begin{equation*}
	\left(b(s)-a(s)\right) - (b-a) = \min\left\{s,a,1-b\right\}.
\end{equation*}
These conditions produce the linear system of equations 
\begin{align*}[left=\empheqlbrace]
\left(b-p\right)a(s) + \left(p-a\right)b(s)     &= p(b-a)\\
	  - a(s) +b(s)                            &= \alpha(s)+b-a
\end{align*}
where $\alpha(s) = \min\left\{s,a,1-b\right\}$.
The unique solution to the system above is 
\[
a(s) = \frac{a^2-ab-a\alpha(s) + \alpha(s)p}{a-b} \qquad \textrm{and} \qquad b(s) = \frac{ab-b^2-b\alpha(s)+\alpha(s)p}{a-b}.
\]
Therefore, for a fixed $s\in I$,  the little $1$ interval $\gamma(s)$ is given by
\[
\gamma(s)(t) = \frac{\alpha(s)(p-a)}{a-b} + \left(b-a+\alpha\left(s\right)\right)t+a\quad \forall \ t\in I.
\]
This finishes the construction for a little $1$-interval.
In the general case, given 
$c\in \mathcal C_n(1)$ of the form 
\begin{equation*}
	c\left(t_1,...,t_n\right) = \big(\left(b_1-a_1\right)t_1 + a_1 , ... , \left(b_n-a_n\right)t_n + a_n\big)
\end{equation*}
 and $f\in C_n\left(X\right)$,  define $\gamma = \gamma^f_c:I \to \mathcal C_n(1)$ to be the path such that
$$\gamma(s)\left(t_1,...,t_n\right) 
= \left(\alpha_1(s)t_1 + p_1-p_1\alpha_1(s) , ... , \alpha_n(s)t_n + p_n-p_n\alpha_n(s)\right)  \quad \forall \ \left(t_1,...,t_n\right)\in I^n.$$
This finishes the construction of the auxiliary path $\gamma^f_c : I \to \mathcal C_n(1)$, 
and therefore the proof and explanations for the auxiliary construction.\hfill$\square$ \\

Now, we are ready to define the homotopy $\mathcal H : C_n\left(X\right)\times I \to C_n\left(X\right)$. 
For each $\left(f,t\right)\in C_n\left(X\right)\times I$,
the image of the homotopy is the map
\begin{center}
\begin{tikzcd}[row sep = small]
{\mathcal H\left(f,t\right) : \mathcal C_n(1)} \arrow[r] & X                           \\
\qquad \qquad c \arrow[r, maps to]                                     & f\left(\gamma^f_c(t)\right).
\end{tikzcd}
\end{center}
Here, $\gamma^f_c$ is the rectilinear expansion of $c$ induced by $f$. 
This rectilinear expansion shrinks the cubical support of $f$ to a point, 
in the sense that 
\begin{equation*}
\label{ecu: cubical support shrinks}
\operatorname{CSupp}\left(\gamma^f_c(1)\right) = \operatorname{Cent}\left(f\right)
\end{equation*}
for every non-constant $f$.
We must check that $\mathcal H$ is well-defined, continuous, 
and satisfies the requirements for being a pointed homotopy from $\id_{C_n\left(X\right)}$ to $\alpha\Psi.$
To check that $\mathcal H$ is well-defined, 
we must corroborate that for each $\left(f,t\right)$, 
the map $\mathcal H\left(f,t\right)$ indeed defines an element in $C_n\left(X\right)$. 
Recall from Proposition \ref{Prop: Geometric characterization of C_n} 
that given $c_1,c_2\in \mathcal C_n(1)$ with $\mathring c_1 \cap \mathring c_2 = \emptyset$, 
this amounts to checking that either
\begin{center}
	$\mathcal H\left(f,t\right)(c_1) = *$ \quad or \quad $\mathcal H\left(f,t\right)(c_2) = *$.
\end{center}
But this is immediate: 
if $\mathring c_1 \cap \mathring c_2 = \emptyset$, 
then $\operatorname{Cent}\left(f\right)$ cannot be in both $c_1$ and $c_2$ at the same time.
Therefore,  
by definition, $\mathcal H\left(f,t\right)$ maps every little cube $c_i$
not having $\operatorname{Cent}\left(f\right)$ in its image to the base point.
We conclude that $\mathcal H$ is well-defined.
The proof that $\mathcal H$ is continuous is done in Appendix \ref{sec: Appendix on continuity of homotopy}.
To finish, it follows directly from the definitions  
that the identities of Equations (\ref{ecu: identities for H}) hold:
\begin{enumerate}
    \item To check that $\mathcal H_0 = \id_{C_n(X)}$, note that 
    $\mathcal H_0 = \mathcal H\left(-,0\right) : C_n(X) \to C_n(X)$ is explicitly given on 
    some $f\in C_n(X)$ as the function $\mathcal C_n(1) \to X$ acting on a given little $n$-cube $c$ as 
    \[
    \mathcal H_0\left(f\right)(c) 
    = \mathcal H\left(f,0\right)(c)  = f\left(\gamma^f_c(0)\right).
    \]
    Now, the rectilinear expansion of $c$ along $f$, which is the path $\gamma^f_c$,
    recovers $c$ when evaluated at $0$.
    This is the second item in Definition \ref{def: rectilinear expansion}.
    Thus, the expression above is exactly $f(c)$, as we claimed.
    \item To check that $\mathcal H_1 = \alpha \circ \Psi$, note that 
    $\mathcal H_1 = \mathcal H\left(-,1\right) : C_n(X) \to C_n(X)$ is explicitly given on 
    some $f\in C_n(X)$ as the function $\mathcal C_n(1) \to X$ acting on a given little $n$-cube $c$ as 
    \[
    \mathcal H_1\left(f\right)(c) 
    = \mathcal H\left(f,1\right)(c)  = f\left(\gamma^f_c(1)\right).
    \]
    On the other hand, evaluated at this same function $f$ and little $n$-cube $c$, the map $\alpha \circ \Psi$ gives
    \[
    \alpha\left[\operatorname{Cent}\left(f\right),L\right](c) = L\left(c^{-1}\left(\operatorname{Cent}\left(f\right)\right)\right),
    \]
    where $L:S^n\to X$ is given by $L(s) = f\left(d_{s,\operatorname{Cent}\left(f\right)}\right)$.
    Here, $d_{s,\operatorname{Cent}\left(f\right)}$ is the auxiliary cube constructed in Lemma \ref{lemma: The cube c_{s,t}}.
    Denote by $t_0\in S^n$ the unique point such that $c(t_0)=\operatorname{Cent}\left(f\right)$.    
    It remains to check that 
    \[
    f\left(\gamma^f_c(1)\right) = L\left(c^{-1}\left(\operatorname{Cent}\left(f\right)\right)\right) = L\left(t_0\right) 
    = f\left(d_{t_0,\operatorname{Cent}\left(f\right)}\right).
    \]
    Here the first equality requires proof, while the other two hold by definition.
    It follows from the uniqueness in the second condition of Lemma \ref{lemma: The cube c_{s,t}}, that
    the required equality follows.
    \item That the homotopy is pointed, i.e., that $\mathcal H(*,t) = *$ for every $t \in I$, is immediate.
    This is because the definition of $\mathcal H(f,t)$ is given by evaluating its first variable $f$ at some cube. 
    If this first variable is the trivial map, the result follows.
\end{enumerate}

We have therefore explained in full detail the definition of $\mathcal H$, 
and checked it gives a pointed homotopy between $\id_{C_n(X)}$ and $\alpha \circ \Psi.$

\medskip

\fbox{\strut \ Proving the equality $\Psi \circ \alpha = \id_{\Sigma^n \Omega^n X}$ \ }

\medskip

\noindent Let  $[t,\ell]\in \Sigma^n \Omega^n X$.
By definition,
\begin{equation}
    \label{ecu : Two components check}
    \Psi \left(\alpha [t,\ell]\right)  = \left[\operatorname{Cent}\left(\alpha[t,\ell]\right), L\right],
\end{equation}
where $L:S^n\to X$ is the loop 
$$
s\mapsto L(s) = \alpha[t,\ell]\left(c_{s,\operatorname{Cent}\left(\alpha[t,\ell]\right)}\right).
$$
Assume that $X$ is not the one-point space and that $\ell$ is not the constant loop; 
otherwise the result is trivial. 
We must check that the two components in the right hand side of Equation (\ref{ecu : Two components check}) are, 
respectively, $t$ and $\ell$.
\begin{enumerate}
    \item  Let us check that $\operatorname{Cent}\left(\alpha[t,\ell]\right) = t$. 
    To do so, it suffices to check that $\operatorname{CSupp}\left(\alpha[t,\ell]\right)$ reduces to the single point $\{t\}$.
Indeed: 
if $c\in \mathcal C_n(1)$ is such that $\alpha[t,\ell](c)\neq *$, 
it follows from the definition of $\alpha[t,\ell]$ that $t\in \mathring{c}$ (recall Equation (\ref{Ecu: Explicit alpha X})).
 Thus, $t\in \Img(c)$ for all little $n$-cubes $c$ such that $\alpha[t,\ell](c)\neq *$. 
Therefore, $t$ is in the intersection of all such images, namely $\operatorname{CSupp}'\left(\alpha[t,\ell]\right)$.
Now, if $t_0\neq t$, then we can always find a little $n$-cube $\widetilde c$ 
with radius greater than $r_0$ 
such that $t_0\notin \Img(\widetilde c)$ and $t\in \Img(\widetilde{c})$, 
and furthermore $\ell\left(\left(\widetilde{c}\right)^{-1}(t)\right)\neq *$ 
(possibly after reparametrization: it might be the case that the loop $\ell$ passes through the base point of $X$, 
but we are assuming $\ell$ is not the constant loop).
\item Let us check that $L(s)=\ell(s)$ for all $s\in S^n$. 
Indeed: 
for $t = \operatorname{Cent}\left(\alpha[t,\ell]\right)$, the little $n$-cube 
$c = c_{s, \operatorname{Cent}\left(\alpha[t,\ell]\right)}$ is such that $c(s) = t$. 
Said differently,  $c^{-1}(t)=s$. 
Therefore, by definition:
$$L(s) = \alpha[t,\ell]\left(c\right) =  
\begin{cases} * &\mbox{if } t\notin \mathring{c} \\
	\ell\left(c^{-1}(t)\right) & \mbox{otherwise}\end{cases}  \ = \ell(s).$$
\end{enumerate}

To summarize: we have explained the definition of the map $\Psi$ and the homotopy $\mathcal H$, 
and have shown the deformation retract requirements of Equation (\ref{ecu: Identities for Psi}) hold. 
Thus, the proof of Theorem \ref{teo: Approx Theorem} is now complete.

\begin{remark}
In this section, we prove the approximation theorem for the little $n$-cubes (rectangles) operad, 
but the ideas could easily be modified to other little convex bodies operads,
like the little $n$-disks operad.
If that were the case, some modifications
would be needed to explain 
what exactly is meant by the center and how the expansion is defined. 
For simplicity, we have decided to look at the little cubes operads only.
\end{remark}

\section{The Recognition Principle for \texorpdfstring{$n$}{n}-fold reduced suspensions}
\label{sec : Recognition principle}

In this section, we  prove the recognition principle for $n$-fold reduced suspensions. 
The precise statement
is the following.

\begin{theorem} 
\label{teo : Recognition Principle}
	Let $X$ be a $\mathcal C_n$-coalgebra. 
	Then there is a pointed space $\Gamma^n(X)$, naturally associated to $X$, 
	together with a weak equivalence of $\mathcal C_n$-coalgebras
	\begin{center}
		\begin{tikzcd}
 \Sigma^n\Gamma^n(X) \arrow[r, "\simeq"] &X,
		\end{tikzcd}
	\end{center}
	which is a deformation retract in the category of pointed spaces. 
    Therefore, every $\mathcal C_n$-coalgebra has the homotopy type of an $n$-fold reduced suspension.
\end{theorem}

The result above is the converse of Theorem \ref{teo: Iterated suspensions are coalgebras},
where it was proven that $n$-fold reduced suspensions are $\mathcal C_n$-coalgebras.
Summarizing, we are providing the following intrinsic characterization of $n$-fold reduced suspensions as $\mathcal C_n$-coalgebras.

\begin{corollary} 
Every $n$-fold suspension is a $\mathcal C_n$-coalgebra,
and if a pointed space is a $\mathcal C_n$-coalgebra then it is homotopy equivalent to an $n$-fold suspension.
\end{corollary}

\begin{remark}
Our Theorem \ref{teo : Recognition Principle} does not require any connectivity assumptions on the spaces,
unlike similar statements in the literature (see for example \cite{Kle97,Blo22}).
Therefore, Theorem \ref{teo : Recognition Principle} is the sharpest possible result.
This follows from the fact that every $\mathcal C_n$-coalgebra is $(n-1)$-connected. 
Indeed, let $X$ be a $\mathcal C_n$-coalgebra with structure map $c:X\to C_n(X)$.
By the approximation theorem, the space $C_n(X)$ is 
homotopic to $\Sigma^n \Omega^n X$, and thus
$(n-1)$-connected. 
Since the composition $X \xrightarrow{c} C_n(X) \xrightarrow{\varepsilon_X} X$ is the identity on $X$ by the counit axiom, 
it follows that $X$ is $(n-1)$-connected.
\end{remark}

For readability, we shall give the proof of Theorem \ref{teo : Recognition Principle} straightaway, 
making reference to the results and notation of the following two subsections.

\begin{proof}[Proof of Theorem \ref{teo : Recognition Principle}]
By Theorem \ref{teo: Approx Theorem}, there is a natural morphism of comonads 
$\alpha_n : \Sigma^n \Omega^n \longrightarrow C_n$, and $\Sigma^n \Omega^n X$ is a deformation retract of $C_n\left(X\right)$.  
Since $\Sigma^n\Omega^n$ preserves equalizers (Proposition \ref{Prop: suspension commutes equalizers}), 
it follows from Lemma \ref{lemma: catcomonad} that the counit map $(\alpha_n)_\ast \alpha_n^!(X) \to X$ 
is a $C_n$-algebra morphism which is a deformation retract of pointed spaces. 
Since $(\alpha_n)_\ast$ preserves the underlying topological space, 
it follows that the $\Sigma^n \Omega^n$-coalgebra $\alpha_n^!(X)$  is a deformation retract of $X$ as a pointed space.  
It then follows from Theorem \ref{teo: SuspensionLoops coalgebras} together with 
the approximation theorem that $\alpha_n^!(X)$ is naturally homeomorphic to
an $n$-fold suspension, and so the counit map $(\alpha_n)_\ast \alpha_n^!(X) \to X$ is an
$C_n$-coalgebra map from a $n$-fold reduced suspension to $X$.  
In particular, $\Gamma^n$ is the functor $ P_n(\alpha_n)_\ast\alpha_n^!.$
\end{proof}

We give a second proof of Theorem \ref{teo : Recognition Principle} 
in Section \ref{sec : The one point subcoalgebra} using explicit formulae very similar to those appearing in the approximation theorem.
This alternative proof is more concrete, and has the further benefit of giving a characterization 
in terms of a certain $C_n$-subcoalgebra.

\subsection{The change of coalgebra structure induced by a comonad morphism}

In this section, 
we explain how
a morphism of comonads $\alpha : C_1 \to C_2$ induces an adjoint pair 
$$\alpha_* :C_1-\mathsf{Coalg} \leftrightarrows C_2-\mathsf{Coalg} : \alpha^!$$
between the corresponding categories of coalgebras (under reasonable hypotheses on the underlying ambient category).
The final goal is to prove the technical Lemma \ref{lemma: catcomonad}, 
which is an essential ingredient for proving Theorem \ref{teo : Recognition Principle}.

\medskip

Suppose that $C_1$ and $C_2$ are two comonads over a category $\mathcal M$ which admits finite limits, and that $\alpha: C_1 \to C_2$ is a morphism of comonads. 
The \textit{change of coalgebra functor}
$$\alpha_\ast : C_1-\mathsf{Coalg} \longrightarrow  C_2-\mathsf{Coalg}$$
is given by mapping a $C_1$-coalgebra $X$ to the same underlying object of $\mathcal M$
equipped with the $C_2$-coalgebra structure map given by the composition
$$X \xrightarrow{\gamma_X} C_1(X) \xrightarrow{\alpha_X} C_2(X).$$
On morphisms, $\alpha_*$ is the identity.

\smallskip

Since $\mathcal M$ has finite limits, by the dual of Dubuc's adjoint triangle theorem \cite{dubuc68},
the change of coalgebra functor $\alpha_*$ has a right adjoint $\alpha^!$
which we call the \textit{enveloping coalgebra functor.} 
If $X$ is a $C_2$-coalgebra, then the $C_1$-coalgebra $\alpha^!\left(X\right)$ is explicitly given as the equalizer 
in $C_1-\mathsf{Coalg}$ of the following pair of morphisms:
$$
\begin{tikzcd}
C_1(X) \arrow[rd, "\triangle^1_X", swap] \arrow[rr, "C_1\left(\delta_X\right)"] &                                                  
& C_1C_2(X) \\
                                                                & C_1 C_1(X) \arrow[ru, "C_1\left(\alpha_X\right)", swap] &                
\end{tikzcd}
$$
Above, $\delta_X$ is the structure map of $X$ as a $C_2$-coalgebra, 
and $\triangle^1_X$ is the 
comultiplication of the comonad $C_1$ evaluated at $X$.
The following proposition, which is the dual of \cite[Prop. 4.3.2]{Borceux}, 
gives conditions for this equalizer to be preserved by the forgetful functor to $\mathcal M$.

\begin{proposition}
\label{prop : Preserve limits}
Let $C$ be a comonad on $\mathcal M$ and let $U:C-\mathsf{Coalg}\to \mathcal M$ be the forgetful functor. Let $G: D\to C-\mathsf{Coalg}$ be a diagram such that $UG$ has a limit in $\mathcal M$ that is preserved by $C$ and $C\circ C.$ Then $G$ has a limit in $C-\mathsf{Coalg}$ that is preserved by $U.$
\end{proposition}

\begin{proof}
 The proof of this result is dual to that of \cite[Prop. 4.3.2]{Borceux}, and it is left to the reader.
\end{proof}

We will need two auxiliary definitions related to equalizers.
    
\begin{definition}
\label{def : cosplit equalizer}
    A \emph{cosplit equalizer} in a category is a diagram 
\begin{center}
\begin{tikzcd}
A \arrow[r, "p"] & B \arrow[r, "f", shift left] \arrow[r, "g"', shift right] \arrow[l, "h", dashed, bend left=60] & C \arrow[l, "s", dashed, bend left=60]
\end{tikzcd}
\end{center}
where 
\begin{equation}
\label{ecu: Split coequalizer}
    sg = \id_B, \quad hp = \id_A \quad \textrm{and } \quad sf = ph.
\end{equation}
\end{definition}

The notion of a cosplit equalizer above is dual to that of split coequalizer, 
and it plays in comonad theory the analog role of split coequalizers in the theory of monads (see \cite[VI. 6]{Mac71}).
The following result is elementary but important.

\begin{proposition}
\label{prop: Cosplit equalizer}
The cosplit equalizer of two morphisms is always an equalizer of the two morphisms; 
and any functor preserves cosplit equalizers.
\end{proposition}

\begin{proof}
    Assume we have a cosplit equalizer with the notation from Definition \ref{def : cosplit equalizer}.
    To prove the first assertion, 
    assume that $\varphi$ is any map such that $f\varphi=g \varphi$.
    Then, 
    $$\varphi = hp \varphi = sg\varphi =  ph \varphi$$ factors through $p$. 
Since $hp=\id_A$, this factorization is unique.
The second assertion is a straightforward consequence of the fact that functors preserve the associativity of the composition and the identity on objects.
\end{proof}

The second definition we need in relation to equalizers is the following.

\begin{definition}
A pair of parallel arrows in a category
\begin{tikzcd}
A \arrow[r, "f", shift left] \arrow[r, "g"', shift right] & B
\end{tikzcd}
is \emph{coreflexive} if there is a common retraction, 
that is, a map $s:B\to A$ such that $sf = \id_A = sg$.
\end{definition}

Next, we relate cosplit and coreflexive equalizers with coalgebra structures.

\begin{proposition}
\label{prop: coalgebra structure map is cosplit equalizer}
Let $C = \left(C,\Delta,\varepsilon\right)$ be a comonad in an arbitrary category, and let $X$ be a $C$-coalgebra.
Then, the coalgebra structure map $\gamma : X \to C(X)$ fits into an equalizer diagram in the category of $C$-coalgebras
$$
\begin{tikzcd}
X \arrow[r, "\gamma"] & C(X) \arrow[r, "C\left(\gamma\right)", shift left] \arrow[r, "\Delta_X"', shift right] & CC(X).
\end{tikzcd}
$$
This equalizer is coreflexive, and it is furthermore cosplit after applying the forgetful functor to the underlying ambient category.
\end{proposition}

\begin{proof}
Let $X$ be a $C$-coalgebra with structure map $\gamma$.
As a consequence of the coassociativity axiom for $\gamma$,
we have the fork 
in the statement.
First, we prove that the diagram of the statemet is a cosplit equalizer 
after applying the forgetful functor to the underlying category.
By Proposition \ref{prop: Cosplit equalizer},
we are done as soon as we give cosplitting maps $h,s$ satisfying the identities of Equation (\ref{ecu: Split coequalizer}), 
taking $p=\gamma$, $f=C\left(\gamma\right)$, and $g=\Delta_X$.
These cosplittings $h$ and $s$ are respectively given by the corresponding counits 
$$
\varepsilon_X: C(X)\to X \qquad \textrm{and} \qquad \varepsilon_{C(X)}:CC(X)\to C(X).
$$
Let us check that the identities in Equation (\ref{ecu: Split coequalizer}) hold.
The identity $hp = \id_A$ becomes $\varepsilon_X \circ \gamma =\id_X$, 
which holds because it is precisely the counital axiom of the  $C$-coalgebra $X$.
Similarly, 
the identity $sg = \id_B$ becomes $\varepsilon_{C(X)} \circ \Delta_X = \id_{C(X)},$
which is exactly the counit axiom at $C(X)$.
It remains to check the identity $sf=ph$, that is, 
$\varepsilon_{C(X)} \circ  C(\gamma) = \gamma \circ \varepsilon_X$. 
This follows from the fact $\varepsilon$ is a natural transformation and so one has the diagram
$$
\begin{tikzcd}
C(X) \arrow[r, "C(\gamma)"] \arrow[d, "\varepsilon_X"]& CC(X)\arrow[d, "\varepsilon_{C(X)}"] \\
X \arrow[r, "\gamma"] & C(X).
\end{tikzcd}
$$
We have checked the three identities of Equation (\ref{ecu: Split coequalizer}).
Therefore, the mentioned diagram is a cosplit equalizer after applying the forgetful functor to the underlying category.
To finish, 
one checks that the diagram is coreflexive in the category of $C$-coalgebras.
To do so, note that we are in the dual situation of \cite[VI.7.1(iii)]{Mac71}.
In fact, our statement and proof is dual to \cite[Lemma 4.3.3]{Borceux}.
Since the diagram is a cosplit fork diagram in the underlying category, it follows from 
\cite[Lemma 4.3.3]{Borceux} that the diagram is actually an equalizer in $C$-coalgebras.
\end{proof}

\begin{remark}
It would be interesting to check whether the arguments in \cite[Section 5]{Harper2010} dualize
to prove that the category of $C$-coalgebras in pointed spaces with the wedge product studied in this paper has all small limits.
\end{remark}

Finally, 
the following technical lemma allows us to directly compare $C_1$- and $C_2$-coalgebras in pointed spaces under certain conditions.
It constitutes an essential ingredient in the proof of Theorem \ref{teo : Recognition Principle}.

\begin{lemma}
\label{lemma: catcomonad}
Let $\alpha: C_1 \to C_2$ be a morphism of comonads in $\mathsf{Top_*}$ which is a deformation 
retract of pointed spaces at each level. 
If $C_1$ preserves equalizers,
then 
the counit $\alpha_\ast \alpha^!\to \id_{C_2-\mathsf{Coalg}}$ of the $\left(\alpha_\ast,\alpha^! \right)$ adjunction 
is a deformation retract of pointed spaces at each level.
In particular, for every $C_2$-coalgebra $X$, 
the underlying map of pointed spaces $\alpha_\ast \alpha^!(X)\to X$ is a deformation retract.
\end{lemma}

\begin{proof}
Let $X$ be a $C_2$-coalgebra.
Let us prove that the underlying map of pointed spaces of the $C_2$-coalgebra morphism 
$\alpha_\ast \alpha^!(X)\to X$ is a deformation retract.
Since $\alpha_*$ is the identity on the underlying pointed space, 
this underlying map is $\alpha^!(X)\to X$. 
Recall from Proposition \ref{prop: coalgebra structure map is cosplit equalizer}
that the $C_2$-coalgebra structure $\gamma$ on $X$ is given by presenting $X$ as the (coreflexive) equalizer of the following diagram:  
$$
\begin{tikzcd}
C_2(X) \arrow[r, "C_2\left(\gamma\right)", shift left] \arrow[r, "\Delta_X"', shift right] & C_2C_2(X).
\end{tikzcd}
$$
Here, $\Delta_X$ is the comultiplication of the $C_2$ comonad at $X$.
This equalizer is taken in $C_2-\mathsf{Coalg}$, 
but we can compute the underlying topological space via the same limit in the category of pointed topological spaces.
This is because this limit is a cosplit equalizer after applying the forgetful functor,
and therefore an equalizer which is preserved by the forgetful functor 
(see Proposition \ref{prop: Cosplit equalizer}). 
Since $C_1$ is assumed to preserve equalizers, by Proposition \ref{prop : Preserve limits}, 
and using a similar argument,
the underlying topological space of $\alpha^!(X)$ may be computed as the equalizer of the diagram
$$
\begin{tikzcd}[row sep = large, column sep= large]
C_1(X)  \arrow[rr, "C_1\left(\gamma\right)", shift left] \arrow[rr, "C_1\left(\alpha_X\right)\circ \Delta_{C_1}"', shift right] &  & C_1C_2(X) 
\end{tikzcd}
$$
in the category of pointed topological spaces.
The deformation retract provided by $\alpha$ thus 
extends to a map (in the category of pointed topological spaces) between the diagram defining $\alpha^!(X)$ and one defining $X$,
namely,
$$
\begin{tikzcd}[row sep = large, column sep= large]
C_1(X) \arrow[d, "\alpha_X"'] \arrow[rr, "C_1\left(\gamma\right)", shift left] \arrow[rr, "C_1\left(\alpha_X\right)\circ \Delta_{C_1}"', shift right] &  & C_1C_2(X) \arrow[d, "\alpha_{C_2(X)}"] \\
C_2(X) \arrow[rr, "C_2\left(\gamma\right)", shift left] \arrow[rr, "\Delta_X"', shift right]                                                          &  & C_2C_2(X)                             
\end{tikzcd}
$$
The corresponding map of limits is thus precisely the desired map $\alpha^!(X)\to X$. 
Since retracts are preserved under limits, we conclude that this map is a deformation retract of pointed spaces. 
\end{proof}

\subsection{The \texorpdfstring{$\Sigma^n\Omega^n$}{SnOn}-coalgebras are \texorpdfstring{$n$}{n}-fold reduced suspensions}

In this section, 
we completely characterize the coalgebras over the $\Sigma^n\Omega^n$-comonad (Theorem \ref{teo: SuspensionLoops coalgebras}).

\medskip

A warning on the notation:
in other parts of this paper, 
we have consistently denoted by $\Delta$ and $\varepsilon$ the comonadic structure maps of the comonad $C_n$
constructed from the little $n$-cubes operad; 
while $\Delta'$ and $\varepsilon'$ were used for the comonaic structure
maps of the comonad $\Sigma^n \Omega^n$.
Since a single comonad will play a role in this section,
namely $\Sigma^n \Omega^n$,
we make an exception here to simply the reading, 
and denote by $\Delta$ and $\varepsilon$ the structure maps of $\Sigma^n \Omega^n$ as a comonad.

\begin{theorem}
\label{teo: SuspensionLoops coalgebras}
Let $X$ be a $\Sigma^n\Omega^n$-coalgebra.
Then $X$ is naturally homeomorphic to 
the $n$-fold reduced suspension of a space $P_n(X)$ which can be computed as the equalizer of the following pair of maps:
\[
\begin{tikzcd}[column sep=large]
\Omega^n X \ar[r,shift left=.75ex,"\Omega^n \gamma"]
  \ar[r,shift right=.75ex,swap," \eta_{\Omega^n X} "]
&
\Omega^n \Sigma^n\Omega^n X.
\end{tikzcd}
\]
Here, $\eta$ is the unit of the $\left(\Sigma^n,\Omega^n\right)$ adjunction, and $\gamma$ is the $\Sigma^n\Omega^n$-coalgebra structure map of $X$. 
\end{theorem}

Theorem \ref{teo: SuspensionLoops coalgebras} is essentially a consequence of the fact that 
taking reduced suspensions preserve equalizers, despite this functor being a left adjoint.
Next, we give a proof of this elementary fact for completeness.

\begin{proposition}
\label{Prop: suspension commutes equalizers}
The $n$-fold reduced suspension functor $\Sigma^n : \mathsf{Top_*}\to \mathsf{Top_*}$ commutes with  equalizers. 
In other words, if \ $\operatorname{Eq}\left(f,g\right) \hookrightarrow X$ is the equalizer of the diagram 
\[
\begin{tikzcd}
X \ar[r,shift left=.75ex,"f"]
  \ar[r,shift right=.75ex,swap,"g"]
&
Y,
\end{tikzcd}
\]
then 
$\Sigma^n \operatorname{Eq}\left(f,g\right)\hookrightarrow \Sigma^n X $ is the equalizer of the diagram
\[
\begin{tikzcd}[column sep=large]
\Sigma^n X \ar[r,shift left=.75ex,"\Sigma^n f"]
  \ar[r,shift right=.75ex,swap,"\Sigma^n g"]
&
\Sigma^n Y.
\end{tikzcd}
\]
Since $\Omega^n$ is right adjoint and thus preserves limits, 
it further follows that $\Sigma^n\Omega^n$ preserves equalizers.
\end{proposition}

\begin{proof}
Recall that, as a set, the equalizer of $f$ and $g$ is given by
$$\operatorname{Eq}\left(f,g\right) = \left\{x\in X \mid f(x) = g(x)\right\}.$$
Since we tacitly work in the category $\mathsf{CGH}$ of compactly generated Hausdorff spaces,
the topology on this set is not necessarily the subspace topology,
but might be finer.
Explicitly, its topology is given by applying the $k$-ification functor $k(-)$, see 
for example \cite[Chapter 5]{May99}.
This functor is 
the right adjoint of the inclusion of $\mathsf{CGH}$ into ordinary topological spaces.
This change in the underlying topology is not an issue, 
because taking $n$-fold reduced suspension commutes with the $k$-ification functor.
Indeed, if $X$ and $Y$ are any compactly generated Hausdorff spaces and $X$ is locally compact,
then $X\times Y$ is a compactly generated Hausdorff space (\cite[Thm. 4.3]{Ste67}).
Since the sphere $S^n$ is locally compact, the product $S^n\times X$ is compactly generated Hausdorff 
for any compactly generated Hausdorff space $X$.
Since the smash product $S^n\wedge X$ is the pushout of the inclusion $S^n\vee X \hookrightarrow S^n\times X$
along the collapse map $S^n\vee X\to *$, it follows that $S^n\wedge X = \Sigma^n X$ is compactly generated Hausdorff.
Thus, $$\Sigma^n \operatorname{Eq}\left(f,g\right) = S^n \wedge \operatorname{Eq}\left(f,g\right).$$
Points in the suspension above are of the form $[t,x]$, with $t\in S^n$ and $x\in X$ such that $f(x)=g(x)$.
On the other hand,
$$
\operatorname{Eq}\left(\Sigma^n f,\Sigma^n g\right) = \left\{ [t,x] \in \Sigma^n X \mid \left[t,f(x)\right] = \left[t,g(x)\right]\right\}.
$$
Under the two identifications above, 
the natural map $$\Sigma^n \operatorname{Eq}\left(f,g\right) \to \operatorname{Eq}\left(\Sigma^n f,\Sigma^n g\right)$$
is a homeomorphism.
\end{proof}

Recall from Proposition \ref{prop: coalgebra structure map is cosplit equalizer}
that every coalgebra structure map is characterized as a cosplit equalizer.
In particular, we have the following result.

\begin{proposition}
\label{prop: cospliteq}
Let $X$ be a $\Sigma^n\Omega^n$-coalgebra with structure map $\gamma$.
Then, as a pointed space, $X$ is the (cosplit) equalizer of the following pairs of maps 
\[
\begin{tikzcd}[column sep=large]
\Sigma^n\Omega^n X \ar[r,shift left=.75ex,"\Sigma^n\Omega^n \gamma"]
  \ar[r,shift right=.75ex,swap,"\Delta_X"]
&
\Sigma^n\Omega^n \Sigma^n\Omega^n X.
\end{tikzcd}
\]
Here, $\triangle$ is the comonadic comultiplication of $\Sigma^n \Omega^n$.
\end{proposition}

\begin{proof}
As mentioned, this is a particular case of Proposition \ref{prop: coalgebra structure map is cosplit equalizer}.
The following diagram is a cosplit equalizer:
\begin{center}
    \label{Eq: cosplit equalizer}
\begin{tikzcd}[column sep=large]
X \rar{\gamma} &\Sigma^n\Omega^n X \ar[r,shift left=.75ex,"\Sigma^n\Omega^n \gamma"]
  \ar[r,shift right=.75ex,swap,"\Delta_X"]
&
\Sigma^n\Omega^n \Sigma^n\Omega^n X,
\end{tikzcd}
\end{center}
where the cosplittings $h$ and $s$ are respectively given by the corresponding counits 
\[\varepsilon_X: \Sigma^n\Omega^n X\to X \qquad \textrm{and} \qquad \varepsilon_{\Sigma^n\Omega^n X}:\Sigma^n\Omega^n \Sigma^n\Omega^n X\to \Sigma^n\Omega^n X. \vspace{-5mm}
\]
\end{proof}

Let us finally prove the main result of this section.

\begin{proof}[Proof of Theorem \ref{teo: SuspensionLoops coalgebras}]
Recall that $\eta$ is the unit of the adunction between $\Sigma^n$ and $\Omega^n$.
Use, 
in the order given, 
Proposition \ref{prop: cospliteq}, that the comonadic coproduct $\Delta_X$ is explicitly given by $\Sigma^n \eta_{\Omega^n(X)}$,
and Proposition \ref{Prop: suspension commutes equalizers} to obtain that
\begin{align*}
    X &= \operatorname{Eq}\left(\Sigma^n\Omega^n \gamma,  \Delta_X\right) 
    = \operatorname{Eq}\left(\Sigma^n\Omega^n \gamma,  \Sigma^n \eta_{\Omega^n X}\right) 
    = \Sigma^n \operatorname{Eq}\left(\Omega^n \gamma,  \eta_{\Omega^n X}\right). 
\end{align*}
This is exactly what we wished to prove.
\end{proof}

\subsection{A point-set description of the recognition principle}
\label{sec : The one point subcoalgebra}

We give here an alternative proof of the recognition principle
mentioned in the introduction to Section \ref{sec : Recognition principle}.
This proof has the advantage of explicitly characterizing the $n$-fold suspension onto which 
a $C_n$-coalgebra deformation retracts.

\begin{theorem}
\label{teo : Every C_n coalgebra is a suspension}
Let $X$ be a $C_n$-coalgebra.
Then, there is a pointed space $Z$ together with a homotopy equivalence of $C_n$-coalgebras $X\simeq \Sigma^nZ$.
\end{theorem}

The strategy of the proof is the following.
First, we show that every $C_n$-coalgebra $X$ contains a $C_n$-subcoalgebra $S(X)$,
and that $X$ deformation retracts onto $S(X)$ as a pointed space (Theorem \ref{teo : S(X) is subcoalgebra}).
Then, we show that $S(X)$ is furthermore a $\Sigma^n\Omega^n$-coalgebra (Theorem \ref{teo : S(X) is SigmaOmega coalgebra}).
Since every $\Sigma^n\Omega^n$-coalgebra $A$ is naturally homeomorphic to the $n$-fold suspension of a pointed space $P_n(A)$
(Theorem \ref{teo: SuspensionLoops coalgebras}),
it follows that $S(X)$ is an $n$-fold suspension, proving Theorem \ref{teo : Every C_n coalgebra is a suspension}.
Explicitly, $X \simeq \Sigma^nP_n(S(X))$.

\medskip

Let us proceed with the argument sketched before.
Recall that every $\Sigma^n\Omega^n$-coalgebra is the $n$-fold suspension of a pointed space 
(Theorem \ref{teo: SuspensionLoops coalgebras}),
so it follows from Proposition \ref{prop: suspensions have a single point cubical support}
that $\Sigma^n\Omega^n$-coalgebras considered as $C_n$-coalgebras have the property 
that the cubical support at each point is just a single point.
The next result proves the converse of this fact.
That is, every $C_n$-coalgebra of which the cubical support of every point (other than the base point)
is just a single point is not just a $C_n$-coalgebra,
but also a $\Sigma^n \Omega^n$-coalgebra.
It further turns out that the set of points whose cubical support is just a single point forms a $C_n$-subcoalgebra.

\begin{theorem}
\label{teo : S(X) is subcoalgebra}
Let $X$ be a $C_n$-coalgebra with coalgebra structure map $c:X\to C_n(X)$.
Then, the subspace $$S(X) = \left\{x\in X \mid \ |\ \operatorname{CSupp}(c(x))| = 1 \right\} \cup \{*\} \subseteq X$$
formed by the points of $X$ whose cubical support is a single point, together with the base point,
is such that the following assertions hold:
\begin{enumerate}
    \item The inclusion $S(X) \hookrightarrow X$ is a homotopy equivalence of pointed spaces.
    \item The subspace $S(X)$ is a $C_n$-subcoalgebra, and the inclusion is a morphism of $C_n$-coalgebras.
\end{enumerate}
Therefore, the inclusion $S(X)\hookrightarrow X$ is a homotopy equivalence of $C_n$-coalgebras.
\end{theorem}

The result above reduces the proof of Theorem \ref{teo : Every C_n coalgebra is a suspension}
to the task of showing that the $C_n$-subcoalgebra $S(X)$ is homotopy equivalent to an $n$-fold suspension as a $C_n$-coalgebra. 
We show next a sharper result which implies it.

\begin{theorem}
\label{teo : S(X) is SigmaOmega coalgebra}
Let $X$ be a $C_n$-coalgebra. 
Then, the $C_n$-subcoalgebra $S(X)$ of Theorem \ref{teo : S(X) is subcoalgebra} is a $\Sigma^n \Omega^n$-coalgebra.
\end{theorem}

Since every $\Sigma^n\Omega^n$-coalgebra is an $n$-fold suspension,
Theorem \ref{teo : Every C_n coalgebra is a suspension} is proven.
It suffices to show the two results mentioned, and we do that next.

\begin{proof}[Proof of Theorem \ref{teo : S(X) is subcoalgebra}]
Denote by $i:S(X) \hookrightarrow C_n(X)$ the inclusion and by $c:X\to C_n(X)$ the coalgebra structure map.

\medskip{}
\fbox{\strut \ Item 1.\ }
Let us give a deformation retraction (of spaces) $r: X\to S(X)$, 
that is, a continuous map $r$ such that $ri = \id_{S(X)}$ and a homotopy $H:X \times I \to X$ between $ir$ and $\id_X$.
The map $r$ is the composition 
\[
r : X \hookrightarrow C_n(X) \xrightarrow{\Psi_X} \Sigma^n\Omega^n X \xrightarrow{\alpha_X} C_n(X) \xrightarrow{\varepsilon_X} X.
\]
The maps above are, respectively, 
the coalgebra structure map of $X$, the natural transformations $\Psi$ and $\alpha$, and the counit $\varepsilon$
from Section \ref{sec:approximation theorem}. 
Since the map $\Psi_X$ reduces the cubical support of every point to a singleton and it is surjective,
the image of $r$ is exactly the subspace $S(X)$. 
It further follows that $ri$ is the identity on the subspace $S(X)$ because the map $\Psi_X$ does not change the cubical support of points whose  cubical support was already a single point. 

The homotopy $\mathcal{H}$ from Theorem \ref{teo: Approx Theorem} can also be used to induce a homotopy in this case. 
In particular,
we get the following homotopy
\[
\mathcal{H} : X\times I \hookrightarrow C_n(X) \times I \xrightarrow{\mathcal{H}_X} C_n(X) \xrightarrow{\varepsilon_X} X.
\]
It is straightforward to check that this is indeed a homotopy between $ir$ and $\id_X$ by using 
exactly the same arguments as in Theorem \ref{teo: Approx Theorem}.
Therefore the inclusion $S(X)$ is a homotopy equivalence of pointed spaces.

\medskip

\fbox{\strut \ Item 2.\ }
To show that $S(X)$ is a $C_n$-subcoalgebra, we must show it is closed under the coproduct.
That is, we must check that if $x\in S(X)$ then the image of the map $c(x):\mathcal C_n(1) \to X$ 
is contained in the subspace $S(X) \subseteq X$. 

We make the following observation to show that this is indeed the case.
If $d,d'\in \mathcal C_n(1)$ are two cubes such that $d \subseteq d'$, 
then $c(x)(d) \neq *$ implies that $c(x)(d') \neq *$. 
This is because of the coassociativity of the comonad.
Since $d=e \circ d'$ is the composition of $d'$ with some other little cube $e$,
we have that $c(x)(d)$ is equal to 
\[
\mathcal C_n(1)\xrightarrow{e} \mathcal C_n(1) \xrightarrow{c} X,
\]
evaluated at $d'$.
So $c(x)(d)=c(x)(e\circ d')=e(c(x))(d')$, where $e(c(x))$ is first the composition of $e$ in the comonad and then acting with this on the coalgebra. 
It therefore follows that if $d \subseteq d'$ and $c(x)(d) \neq *$, 
then $c(x)(d') \neq *$.
From this it is straightforward to deduce that if the cubical support of $c(x)$ is just a single point 
then the image of $c(x)$ is contained in $S(X)$;
otherwise the previous identity would be violated.
Therefore, $S(X)$ is a $C_n$-subcoalgebra and the inclusion map is a homotopy equivalence of $C_n$-coalgebras. 
\end{proof}

\begin{proof}[Proof of Theorem \ref{teo : S(X) is SigmaOmega coalgebra}]
To prove this result, 
we need to define a map $c':S(X) \to \Sigma^n \Omega^n S(X)$ and show that it satisfies the comonad identities.
We define $c':S(X) \to \Sigma^n \Omega^n S(X)$ a $c'(x) := [t,\ell]$,
where $t=\operatorname{Cent}\left(c(x)\right)$ and $\ell:S^n \to S(X)$ is given by 
$$
\ell(s) = c(x) \left(c_{s,\operatorname{Cent}\left(c(x)\right)}\right) = c(x) \left(c_{s,t}\right),
$$
where $c_{s,\operatorname{Cent}\left(c(x)\right)}$ is the cube from the proof of Theorem \ref{teo: Approx Theorem}.
Because $c'$ is a $C_n$-coalgebra map, it follows that it also satisfies the coassociativity axiom to be a $\Sigma^n \Omega^n$-coalgebra, which completes the proof.
\end{proof}

\begin{appendices}
\section{The map \texorpdfstring{$\alpha$}{a} is a morphism of comonads}
\label{sec: Appendix on morphism of comonads}

In this appendix, we give the necessary definitions and prove in full detail that the natural transformation $$\alpha_n : \Sigma^n \Omega^n \to C_n$$ appearing in  Theorem \ref{teo: Approx Theorem} defines a morphism of comonads.

\begin{definition} 
	\label{def:Morphism of comonads}
	A \emph{morphism of comonads} $\alpha:\left(C,\Delta,\varepsilon\right) \to \left(C',\Delta',\varepsilon'\right)$ in a category $\mathcal M$ is a natural transformation $\alpha : C\to C'$
	such that for every object  $X\in \mathcal M$, the following two diagrams commute:
	\begin{center}
\begin{tikzcd}[column sep= tiny]
	C(X) \arrow[rr, "\alpha_X"] \arrow[rd, "\varepsilon_X"'] &                                              & C'(X) \arrow[ld, "\varepsilon'_X"] &  & C(X) \arrow[rr, "\Delta_X"] \arrow[d, "\alpha_X"'] &                                                    & C(C(X)) \arrow[d, "\alpha^2_X"] \\
	& X                                            &                                    &  & C'(X) \arrow[rr, "\Delta'_X"]                      &                                                    & C'(C'(X))                       \\
	& \varepsilon'_X\circ \alpha_X = \varepsilon_X &                                    &  &                                                    & \alpha^2_X\circ \Delta_X = \Delta'_X\circ \alpha_X &                                
\end{tikzcd}
	\end{center}
The morphism $\alpha^2_X$ is defined by the following diagram,
which is commutative because $\alpha$ is a morphism of comonads.
\begin{center}
	\begin{tikzcd} 
		C(C(X)) \arrow[rr, "\alpha_{C(X)}"] \arrow[d, "C(\alpha_X)"'] \arrow[rrd, "\alpha_X^2", dashed] &  & C'C(X) \arrow[d, "C'(\alpha_X)"] \\
		C(C'(X)) \arrow[rr, "\alpha_{C'(X)}"]                                                           &  & C'(C'(X))
	\end{tikzcd}
\end{center}

\begin{equation}
\label{ecu:Commutativity of alpha squared}
\alpha^2_X = C'(\alpha_X)\circ \alpha_{C(X)} = \alpha_{C'(X)} \circ C(\alpha_X)
\end{equation}
\end{definition}

\bigskip

Next, we settle the morphism of comonads assertion made in Theorem \ref{teo: Approx Theorem}.

\begin{proposition}
	\label{Prop: Morphism of comonads}
	The natural transformation $\alpha_n:\Sigma^n \Omega^n \to C_n$ in Theorem \ref{teo: Approx Theorem} is a morphism of comonads.
\end{proposition}

\begin{proof}
Fix an integer $n\geq 1$,
and denote $\alpha_n$ by $\alpha$ to simplify the notation. 
Recall that object-wise, the natural transformation $\alpha$
is explicitly given by
 \begin{equation*}
	\alpha_X : \Sigma^n \Omega^n X \xrightarrow{\gamma} C_n\left(\Sigma^n \Omega^n X\right) \xrightarrow{C_n(\eta_X)} C_n\left(X\right),
\end{equation*}
where $\gamma$ 
is the $\mathcal C_n$-coalgebra structure map 
of $\Sigma^n \Omega^nX$ (Theorem \ref{teo: Iterated suspensions are coalgebras}), and $\eta_X $ is the 
evaluation at $X$ of the counit  $\eta: \Sigma^n\Omega^n \to \id_{\mathsf{Top_*}}$ 
of the adjunction $\left(\Sigma^n,\Omega^n\right)$.
Identify $$\Sigma^n \Omega^n X \cong S^n \wedge \Map_*\left(S^n,X\right).$$
Under this identification, the counit $\eta_X : \Sigma^n \Omega^n X\to X$ becomes the evaluation map, 
$$ev:S^n \wedge \Map_*\left(S^n,X\right) \to X \quad ev : [t,\ell]\mapsto \ell(t).$$

Next, identify $C_n\left(X\right)$ as a subspace of $\Map\left(\mathcal C_n(1),X\right)$.
Recall that 
under this identification, the value of $C_n(g)$ on a map $g:\mathcal C_n(1)\to X$ is the postcomposition with $g$ 
(Proposition \ref{Prop:Characterization of C(X) as a subspace of Map}). 
Then, the map $\alpha_X : \Sigma^n \Omega^n X\to C_n\left(X\right)$ is explicitly given on a point $[t,\ell]$  as the map $$\alpha_X[t,\ell]:\mathcal C_n(1) \to X$$ whose image on a little $n$-cube $c\in \mathcal C_n(1)$ is  
\begin{equation}
	\label{Ecu: Explicit alpha X}
	\alpha[t,\ell](c) =
	\begin{cases} \ell\left(c^{-1}(t)\right) &\mbox{if } t\in \mathring{c}\\
		*	& \mbox{otherwise} \end{cases}
\end{equation}
Geometrically, $\alpha_X$ is just re-scaling the evaluation map $ev : S^n\wedge \Map_*\left(S^n,X\right)$
by shrinking the points of $S^n = I^n/ \partial I^n$ according to the little $n$-cube $c$.

\medskip

We can now check the commutativity of the diagrams in Definition \ref{def:Morphism of comonads}.

\medskip

\fbox{\strut \ $\varepsilon'_X\circ \alpha_X = \varepsilon_X$ \ }

\medskip

Let $[t,\ell]\in \Sigma^n \Omega^n X$.
Since $\varepsilon'_X$ plugs the identity operation $\id\in \mathcal C_n(1)$, we have:
\begin{center}
\begin{tikzcd}[row sep=small]
	\varepsilon'_X\circ \alpha_X :\Sigma^n \Omega^n X \arrow[r, "\alpha_X"] & C_n\left(X\right) \arrow[r, "\varepsilon'_X"]    & X                                  \\
	{[t,\ell]} \arrow[r, maps to]             & {\alpha_X[t,\ell]} \arrow[r, maps to] & {\alpha_X[t,\ell](\id)=\ell(c(t))}
\end{tikzcd}
\end{center}
The composition above is exactly the definition of 
 $\varepsilon_X[t,\ell]$.

\medskip

\fbox{\strut \ $\alpha^2_X\circ \Delta_X = \Delta'_X\circ \alpha_X$\ }

\medskip

The map $\alpha_X^2$ can be written as two different compositions, see Diagram (\ref{ecu:Commutativity of alpha squared}). 
Here, we  prove that 
\begin{equation}\label{Ecu: Morphism of monads 2}
	\alpha_{C'(X)} \circ C\left(\alpha_X\right) \circ \Delta_X = \Delta'_X \circ \alpha_X,
\end{equation}
where $C=\Sigma^n\Omega^n \xrightarrow{\alpha_n} C'=C_n.$
The left hand side of Equation (\ref{Ecu: Morphism of monads 2}) is the composition
\begin{center}
	\begin{tikzcd}[column sep = huge]
		\Sigma^n\Omega^nX \arrow[r, "\Delta_X"] & \Sigma^n\Omega^n\left(\Sigma^n\Omega^nX\right) \arrow[r, "\Sigma^n\Omega^n\left(\alpha_X\right)"] & \Sigma^n\Omega^n\left(C_n\left(X\right)\right) \arrow[r, "\alpha_{C_n\left(X\right)}"] & C_n\left(C_n\left(X\right)\right).
	\end{tikzcd}
\end{center}
The maps in the composition above are given as follows.
\begin{itemize}
	\item Denote by $\eta_X:X\to \Omega^n \Sigma^n X$ the unit of the $\left(\Sigma^n , \Omega^n\right)$ adjunction.
	Then $\Delta_X = \Sigma^n \circ \eta_X \circ \Omega^n$.
	Thus, a point $[t,\ell]\in \Sigma^n\Omega^n X = S^n\wedge \Map_*\left(S^n,X\right)$ maps to the point $[t,\bar \ell]\in S^n\wedge \Map_*\left(S^n,\Sigma^n\Omega^n X\right)$, where $$\bar \ell : S^n \to \Sigma^n \Omega^n X \quad s \mapsto [s,\ell].$$  
	\item The second map $\Sigma^n\Omega^n\left(\alpha_X\right)$ maps the point $[t,\bar \ell]$ to the point $[t, \alpha_X\circ \bar \ell]$.
	\item The last map takes a point $[t,\ell']$, where $\ell':S^n\to C_n\left(X\right)$ is a loop,
	to the evaluation
	\begin{center}
		\begin{tikzcd}[row sep=small]
			{\alpha_{C_n\left(X\right)}[t,\ell']:\mathcal C_n(1)} \arrow[r] & C_n\left(X\right)                      \\
			c \arrow[r, maps to]                                 & \ell'\left(c^{-1}(t)\right)
		\end{tikzcd}
	\end{center}
\end{itemize}
Therefore, with the notation above, the full composition applied to a point $[t,\ell]$ yields
\begin{equation*}
	[t,\ell] \mapsto [t, \bar \ell] \mapsto [t,\alpha_X\circ \bar \ell] \mapsto \alpha_{C_n\left(X\right)}[t,\alpha \circ \bar \ell].
\end{equation*}
The resulting map 
$$\alpha_{C_n\left(X\right)}[t,\alpha \circ \bar \ell] : \mathcal C_n(1)\to C_n\left(X\right)$$
acts on a little $n$-cube 
$c\in \mathcal C_n(1)$ by producing 
$$c \mapsto \left(\alpha_X\circ \bar l\right)\left(c^{-1}(t)\right)=\alpha[c^{-1}(t),\ell] : \mathcal C_n(1)\to X,$$
where $c_2\in \mathcal C_n(1)$ gets mapped to $$\alpha[c^{-1}(t),\ell](c_2) = \ell\left(c_2^{-1}\left(c^{-1}(t)\right)\right).$$

The right hand side of Equation (\ref{Ecu: Morphism of monads 2}) is the composition
\begin{center}
	\begin{tikzcd}
		\Sigma^n\Omega^nX \arrow[r, "\alpha_X"] & C_n\left(X\right) \arrow[r, "\Delta'_X"] & C_n\left(C_n\left(X\right)\right)
	\end{tikzcd}
\end{center}
The first map in the composition above was given in Equation (\ref{Ecu: Explicit alpha X}).
The map $\Delta'_X$, described in Proposition \ref{Prop: Comonad structure maps}, 
 applies an arbitrary map $h:\mathcal C_n(1)\to X$ to the map $\bar h:\mathcal C_n(1)\to C_n\left(X\right)$ given by 
 $$\mu\in \mathcal C_n(1) \mapsto \bar h(\mu) : \mathcal C_n(1)\to X, \quad \bar h(\mu)(\theta) := h\left(\gamma\left(\mu;\theta\right)\right).$$
 In particular, $\Delta'_X$ applies the map $\alpha_X[t,\ell]$ to the map 
\begin{center}
\begin{tikzcd}[row sep=small]
\Delta_X'\left(\alpha_X[t,\ell]\right): \mathcal C_n(1) \arrow[r] & C_n\left(X\right)                    &                                       \\
c \arrow[r, maps to]        & \Delta_X'\left(\alpha_X[t,\ell]\right)(c) = \overline{\alpha[t,\ell]}(c):\mathcal C_n(1) \arrow[r] & X                                     \\
& c_2 \arrow[r, maps to]     & \ell\left(\gamma\left(c;c_2\right)^{-1}(t)\right)
\end{tikzcd}
\end{center}
Since, by definition of the composition in the little cubes operad, $$\ell\left(c_2^{-1}\left(c^{-1}(t)\right)\right) = \ell\left(\gamma\left(c;c_2\right)^{-1}(t)\right)$$
for all little cubes $c,c_2$, the claim is proven.
\end{proof}

\section{The homotopy \texorpdfstring{$\mathcal H$}{H} is continuous}
\label{sec: Appendix on continuity of homotopy}

In this appendix, we prove in full detail that the homotopy
$$
\mathcal H:C_n\left(X\right)\times I \to C_n\left(X\right)
$$
appearing in Theorem \ref{teo: Approx Theorem} takes values in continuous functions and is continuous. 
For simplicity, we first do the case of 1-cubes and explain later how the construction generalizes. 
For each $\left(f,t\right)\in C_n\left(X\right)\times I$,
the image of the homotopy is the map
\begin{center}
\begin{tikzcd}[row sep = small]
{\mathcal H\left(f,t\right) : \mathcal C_n(1)} \arrow[r] & X                           \\
\qquad \qquad c \arrow[r, maps to]                                     & f\left(\gamma^f_c(t)\right).
\end{tikzcd}
\end{center}
We first prove that this is continuous in $c$. 
As before, for simplicity, we first do the case of 1-cubes and it will be clear later how the construction generalizes.
To do this, we first impose a metric on $\mathcal C_1(1).$ 
For $c_1, c_2\in C_1(1),$ define 
\begin{equation}
    \label{ecu: distance on C_1(1)}
    d\left(c_1, c_2\right) = \operatorname{max}\left\{|a_1 -a_2|,|b_1-b_2|\right\},
\end{equation}
where $c_i(u) = (b_i-a_i)u + a_i$.
In other words, it is the largest distance between any two corresponding sides of either cube. 
The topology on $\mathcal C_1(1)$ induced by this metric coincides with its usual topology 
(i.e., the subspace topology inside $\operatorname{Map}(I,I)$ with the compact-open topology).

We next prove that the assignment
$$
c\mapsto \gamma^f_{c}(t)
$$
is continuous in $c$ with respect to this metric. 
Recall that 
$$
a(s) = \frac{a^2-ab-a\alpha(s) + \alpha(s)p}{a-b} \qquad \textrm{and} \qquad b(s) = \frac{ab-b^2-b\alpha(s)+\alpha(s)p}{a-b}.
$$
Then, we have
\begin{multline*}
    |a_1(s)-a_2(s)|\leq |a_1-a_2|+ \alpha(s)\left|\frac{p-a_1}{b_1-a_1}- \frac{p-a_2}{b_2-a_2}\right|\\\leq 
    |a_1-a_2|+ \frac{\alpha(s)}{|(b_1-a_1)(b_2-a_2)|}\left(p|b_1-b_2|+p|a_1-a_2| + a_1|b_1-b_2| + b_1|a_1-a_2| \right),
\end{multline*}
where we have used the triangle inequality repeatedly. 
The left hand side goes to zero as $d\left(c_1, c_2\right)\to 0$.
A similar computation 
for $|b_1(s)-b_2(s)|$ proves that $d\left(\gamma^f_{c_1}(t), \gamma^f_{c_2}(t)\right)\to 0$ 
as $d\left(c_1, c_2\right)\to 0$ for all $t\geq 0$. 
It then follows from the continuity of $f$ that $\mathcal H\left(f,t\right)$ is a continuous function. 

\medskip

The method outlined clearly extends to the general case of the little $n$-cubes operad for $n>1$.

\medskip

Next we verify that the homotopy $\mathcal H$ itself is continuous. First, we shall show that for $f\in C_n(X),$ the function $\operatorname{Cent}\left(f\right)$ depends continuously on $f.$
Recall that this constructed as follows:
first, one computes the cubical support of $f$. 
If $f$ is nontrivial, this forms an
$n$-rectangle $R$ and then one computes the the center of $R$. 
The procedure of computing the center clearly depends continuously on the choice of rectangle. 
It therefore suffices to show that the function
$$
f \mapsto \operatorname{CSupp}\left(f\right)
$$
is continuous. 
This has been shown in Proposition \ref{prop: The cubical support is continuous}.
Then it follows from our explicit formulae that rectilinear expansion $\gamma^f_{c}(t)$, 
viewed as a function $C_n(1)\times I\to C_n(1)$, 
depends continuously on $\operatorname{Cent}\left(f\right)$.
We therefore have that the function $c\mapsto f\left(\gamma_{c}^f(t)\right)$ depends continuously on $f$ 
as it is the composition of two continuous functions in $f$.

\section{An explicit description of the map $G$}
\label{sec: The explicit description of the map G}

In this appendix, we explicitly construct the function $G$ whose existence is claimed in Section \ref{sec: Center and cubical support}.

\smallskip

We construct the map $G$ from geometric arguments.
We will start with the 1-dimensional case, and explain the higher-dimensional case at the end.
First, identify the space $\mathcal C_1(1)$ with a right triangle $T$ missing a cathetus,
$$T = \bigcup_{x\in (0,1]} \big(\{x\}\times [0,1-x)\big) \subseteq \mathbb R^2.$$
The homeomorphism $\mathcal C_1(1)\cong T$ maps the little $1$-cube $c=[a,b]$ to the point $\left(b-a, a\right)$.
In the triangle, the $x$-coordinate represents the size of the cube,
and the $y$-coordinate its distance from the origin $0$.
The inverse homeomorphism maps the point $(x,y)$ to the little $1$-cube $[y,x+y]$.
Under this point of view, the whole left side of the triangle is missing because it corresponds to the limit points of shrinking intervals.
The right vertex $(1,0)$ corresponds to the identity cube $\id\in \mathcal C_1(1)$.

We will derive $G$ from geometric arguments applied to the triangle $T$
and then pulling back the formulas to the space $\mathcal C_1(1)$.
The idea is simple once we interpret the three properties required to $G$ in the triangle $T$.
\begin{enumerate}
    \item \emph{The function $G$ is the identity when restricted to $A$.}
    
    The set $A$ of little $1$-cubes touching the boundary $\partial I = \{0,1\}$
    of $I$ is the union of the two solid sides of the triangle.
    We refer to this set as the boundary of the triangle, 
    as it is literally the topological boundary of $T$ as a subspace of $\mathbb R^2$.
    The bottom side of the triangle corresponds to the little cubes touching 0,
    and the top side (hypotenuse) corresponds to the little 1-cubes touching 1.
    The little 1-cubes corresponding to interior points of the triangle do not touch the boundary $\partial I$.
    Therefore, the first requirement amounts to asking the homeomorphism $G$ to be the identity in the boundary of the triangle. 

    \item \emph{For each sequence of cubes outside of $A$ on which the function $q$ tends to zero,
    the image of the map $G$ on this sequence tends to the identity cube.}
   
   Recall that $q$ is the quotient $\operatorname{rad}/\left(\operatorname{rad}+d_b\right)$.
    Therefore, the condition that $q \to 0$ on sequences of cubes is equivalent 
    to  requiring that the sequence of points in $T$ corresponding to these cubes tends to the left (missing) side of the triangle.
    It is important to remark that this requirement is only for sequences not in $A$.

    \item \emph{There is an inclusion on images $\operatorname{Im}(c)\subseteq \operatorname{Im}\left(G(c)\right)$ 
    for every little interval $c\in \mathcal C_1\left(1\right)$.}
    
    If $c=[a,b]$ is a little 1-cube, then $G(c) = [f(a),g(b)]$ for some $f,g:[0,1]\to [0,1]$, 
    and the condition $\operatorname{Im}(c) \subseteq \operatorname{Im}\left(G(c)\right)$ is equivalent to the requirement 
    \begin{equation}
        \label{ecu: restrictions}
        0 \leq f(a) \leq a \qquad \textrm{and} \qquad b \leq g(b) \leq 1.
    \end{equation}
    In particular, we deduce from the equations above that $G$ cannot shrink the size of cubes: 
    $b-a \leq g(b)-f(a) \leq 1$.
    Next, let us interpret these conditions under the identification $\mathcal C_1(1) \cong T$.
    Let $c = (x,y)\in T$.
    Then $G(c) = \left(h(x),k(y)\right)$ for some $h,k:[0,1]\to [0,1]$ such that if
    \[
    c = (x,y) = [\underbrace{y}_{a},\underbrace{x+y}_{b}],
    \]
    then 
    \[
    G(c) = (h(x),k(y)) = [\underbrace{k(y)}_{f(a)},\underbrace{h(x)+k(y)}_{g(b)}].
    \]
    Thus, using Equation \eqref{ecu: restrictions}, the maps $h$ and $k$ satisfy
    \[
    0 \leq k(y) \leq y \qquad \textrm{and} \qquad x+y \leq h(x)+k(y)\leq 1.
    \]
    An elementary operation with the equations above yields that $x \leq h(x)$.
    Recalling that the $x$-coordinate of the point $c$ corresponds to its size, the inequality $x \leq h(x)$
    codifies the fact that $G$ cannot shrink cubes.
\end{enumerate}

The geometric constraints summarized above suggest several ideas to construct an explicit such $G$.
We will chase the following idea.
Let $c=[a,b]$ be a little $1$-interval, identified with the point $(x_0,y_0)$ of the triangle $T$.
If $c$ belongs to the boundary of the triangle $A$, then $G(c) = c$.
Therefore, assume $c\notin A$.
We need several auxiliary constructions and definitions that we collect as steps below.

\smallskip

\noindent \emph{Step 1: the line segments.}
Consider the line segment $\ell_c$ in the triangle containing $c$ and the identity cube $\id = (1,0)$.
The slope-intercept equation of $\ell_c$ gives that 
\[
\ell_c =  \left\{ y = \frac{y_0}{x_0-1}(x-1) \mid x\in [0,1)\right\}.
\]
The line segment $\ell_c$ depends only on the slope of $c$, denoted $\operatorname{slope}(c)$,
and it is a continuous function of $c$.
The possible slopes are in the range $[-1,0]$.

\begin{figure}[h!]\centering
		\includegraphics[scale=0.4]{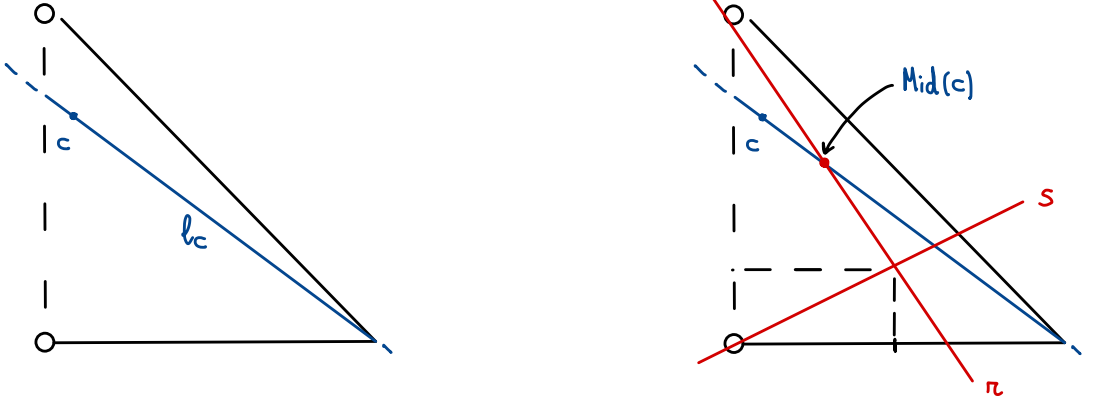}
        \caption{{\footnotesize{Left: \ The line $\ell_c$. \ Right: The lines $r$ and $s$, and the midpoint $\operatorname{Mid}(c)$.}}}
\end{figure}

\bigskip

\smallskip

\noindent \emph{Step 2: the midpoint of the line segments.}
To each line segment $\ell_c$, we will associate a "midpoint" that depends continuously on the slope of $\ell_c$. 
To do so, we need two auxiliary lines $r$ and $s$ that we define now.
Let $r$ be the affine line through the points $(0,1)$ and $(1/2,1/4)$. 
Its slope-intercept equation is 
\[
r \ \equiv \ \big\{ y = -\frac{3}{2}x+1.
\]
Let $s$ be the line through the points $(0,0)$ and $(1/2,1/4)$. 
Explicitly, 
\[
s \ \equiv \  \big\{ y =  \frac{1}{2}x.
\]
The midpoint of a little 1-cube $c$, denoted $\operatorname{Mid}(c)$, is the point in $T$ given by
\[
\operatorname{Mid}(c) = 
\begin{cases}
    r \cap \ell_c & \mbox{if } \operatorname{slope}(c) \leq -\frac{1}{2} \\
    s \cap \ell_c & \mbox{otherwise.}
\end{cases}
\]
It is a continuous function of $c$ whose expression in coordinates when $\operatorname{slope}(c) \leq -\frac{1}{2}$ is
\[
\operatorname{Mid}(c) 
= \left(\frac{2(x_0+y_0-1)}{3x_0+2y_0-3},\frac{-y_0}{3x_0+2y_0-3}\right),
\]
and otherwise it is
\[
\operatorname{Mid}(c) 
= \left(\frac{2y_0}{2y_0-x_0+1},\frac{y_0}{2y_0-x_0+1}\right). 
\]

\smallskip

\noindent \emph{Step 3: the final expression of $G$.}
The function $G$ is the identity on those cubes lying on the line $s$ or below it,
and on the line $r$ or above it.
That is, if $c=(x_0,y_0)$, then $G(c)= c$ if 
\[
y_0 \leq \frac{1}{2} x_0 \qquad \textrm{or} \qquad y_0 \geq -\frac{3}{2} x_0.
\]
Otherwise, $G$ slides the cube $c$ along the line connecting $c$ to the identity cube $(1,0)$ in inverse proportion:
\begin{equation}
    \label{ecu: formula for G}
    G(c) = \operatorname{Mid}(c) + \lambda \cdot \overrightarrow{v}.
\end{equation}
\begin{figure}[h!]\centering
		\includegraphics[scale=0.4]{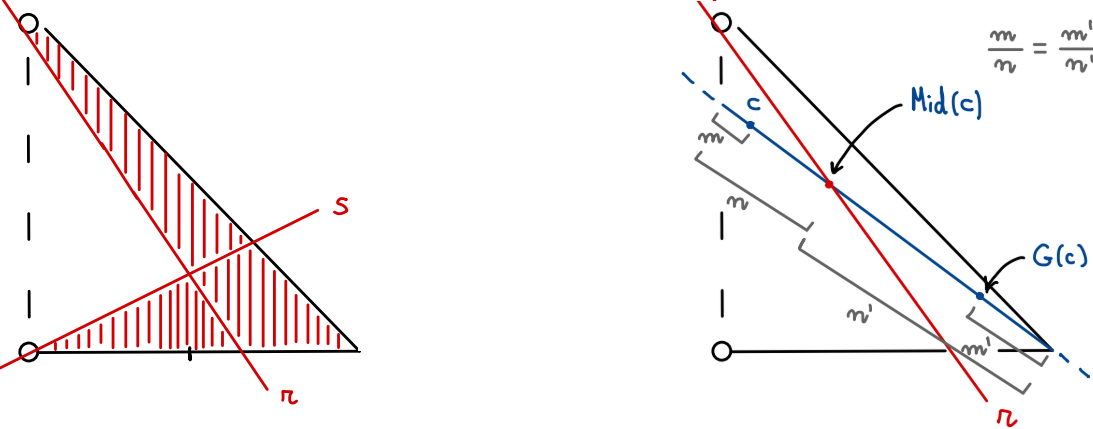}
        \caption{{\footnotesize{\ Left: The region where $G  =\id$ in red.\ Right: The image $G(c)$ using the  proportion factor.}}}
\end{figure}

Here, 
\[
\overrightarrow{v} = \overrightarrow{\operatorname{Mid}(c),\operatorname{id}} = \overrightarrow{\operatorname{Mid}(c),(1,0)} 
= (1,0) - \operatorname{Mid}(c)
= \begin{cases}
    \left(\frac{x_0 - 1}{3x_0+2y_0-3}, \frac{y_0}{3x_0+2y_0-3}\right)           & \mbox{if } \operatorname{slope}(c) \leq -\frac{1}{2} \\[0.3cm]
    \left(\frac{1 - x_0}{2y_0 - x_0 + 1}, \frac{-y_0}{2y_0 - x_0 + 1}\right)    & \mbox{otherwise}
\end{cases}
\]
is the direction vector of the line through $\operatorname{Mid}(c)$ and $(1,0)$, 
and
\[
\lambda 
= \frac{\operatorname{dist}\left(c,\ell_c\cap \left(x=0\right)\right)}
{\operatorname{dist}\left(\operatorname{Mid}(c),\ell_c\cap \left(x=0\right)\right)}
\]
is the proportion factor.
Technically, $\ell_c$ is defined as a line segment inside $T$, but here we are abusing the notation so that
$\ell_c$ denotes the whole line containing this segment.
Explicitly,
\[
\lambda 
= \frac{\operatorname{dist}\left(c,\left(0,-\frac{y_0}{x_0-1}\right)\right)}
{\operatorname{dist}\left(\operatorname{Mid}(c),\left(0,-\frac{y_0}{x_0-1}\right)\right)}
= 
\begin{cases}
\begin{aligned}
\frac{\sqrt{x_0^2 + \left(\frac{y_0}{x_0-1}+y_0\right)^2}}
    {\sqrt{ \left(x_0 - \frac{2(x_0 + y_0 - 1)}{3x_0 + 2y_0 - 3}\right)^2 + \left(y_0 + \frac{y_0}{3x_0 + 2y_0 - 3}\right)^2}} 
    & \qquad  \mbox{if } \operatorname{slope}(c) \leq -\frac{1}{2} \\[2ex]
\frac{\sqrt{x_0^2 + \left(\frac{y_0}{x_0-1}+y_0\right)^2}}
    {\sqrt{\left(x_0 + \frac{2y_0}{x_0 - 2y_0 - 1}\right)^2 + \left(y_0 + \frac{y_0}{x_0 - 2y_0 - 1}\right)^2 }} 
    &\qquad  \mbox{otherwise.}
\end{aligned}
\end{cases}
\]

The explicit form of $G$ on points outside $A$ can be obtained directly from Equation \eqref{ecu: formula for G} 
by substituting the terms with the expressions previously computed. 
We refrain from presenting the resulting formulas, as they would not contribute further to the clarity of the exposition.

\smallskip

A similar strategy provides explicit formulas for the map $G: \mathcal C_n(1) \to \mathcal C_n(1)$ in the higher-dimensional case.
Indeed, there is a homeomorphism 
\[
\mathcal C_n(1) \cong \mathcal C_1(1)\times \cdots \times \mathcal C_n(1)
\]
mapping a little $n$-cube $c=\left(f_1,...,f_n\right)$ to the product $f_1\times \cdots \times f_n$, 
where each $f_i: I\to I$ is the rectilinear embedding at the $i$-th coordinate, 
\[
f_i (t) = tb_i +(1-t)a_i.
\]
In terms of geometric cubes,  
it maps the little $n$-cube $c$ to the product of intervals $c_1\times \cdots \times c_n = [a_1,b_1]\times \cdots \times [a_n,b_n]$.
The homeomorphism aboves provides an identification
\[
\mathcal C_n(1) \xrightarrow{\ \cong \ } T\times \cdots \times T = T^n \subseteq \mathbb R^{2n} 
\]
\[
c = [a_1,b_1]\times \cdots \times [a_n,b_n] \mapsto \left(b_1-a_1,a_1, ..., b_n-a_n,a_n\right).
\]
Applying the 1-dimensional function $G$ at each coordinate above yields the higher-dimensional map.

\end{appendices}

\bibliographystyle{plain}
\bibliography{MyBib}

\medskip

\noindent\sc{Oisín Flynn-Connolly}\\
\noindent\sc{Leiden Institute of Advanced Computer Science,\\
Leiden University,\\
Leiden, The Netherlands}\\
\noindent\tt{o.c.flynn-connolly@liacs.leidenuniv.nl}\\

\noindent\sc{José Manuel Moreno Fernández}\\ 
\noindent\sc{Departamento de Álgebra, Geometría y Topología, \\ Universidad de Málaga, 29080, Málaga, Spain}\\
\noindent\tt{josemoreno@uma.es}\\

\noindent\sc{Felix Wierstra}\\ 
\noindent\sc{Korteweg-de Vries Institute for Mathematics \\
University of Amsterdam \\ Science Park 105-107, 1098 XG Amsterdam, Netherlands}\\
\noindent\tt{felix.wierstra@gmail.com}

\end{document}